\documentclass{amsart}
\usepackage{amsmath,bm,xcolor}
\usepackage{amsthm}
\usepackage{amssymb}
\usepackage{amsfonts}
\usepackage{graphicx}
\usepackage{mathrsfs,bm,amscd}
\usepackage{latexsym}
\usepackage{color}
\usepackage{marginnote}
\usepackage{hyperref}
\usepackage{mathtools}
\usepackage{enumerate}
\usepackage{MnSymbol}

\usepackage[all]{xy}

\numberwithin{equation}{section}
\theoremstyle{plain}
\newtheorem{thm}{Theorem}[section]
\newtheorem{prop}[thm]{Proposition}
\newtheorem{cor}[thm]{Corollary}
\newtheorem{lem}[thm]{Lemma}

\theoremstyle{definition}
\newtheorem{definition}[thm]{Definition}
\newtheorem{rem}[thm]{Remark}
\newtheorem*{convention}{Convention}
\newtheorem{question}[thm]{Question}

\newtheorem*{bprob*}{Bonus problem}

\def\R{\mathbb{R}}

\def\N{\mathbb{N}}

\def\e{\epsilon}
\def\g{\gamma}
\def\d{\delta}
\def\D{\Delta}

\def\V{\mathbb{V}}

\newcommand{\T}{\mathbb{T}}

\DeclareMathOperator{\dist}{dist}

\newcommand{\bx}{\mathbf x}

\DeclareMathOperator{\diam}{diam}
\DeclareMathOperator{\ba}{\bf{a}}
\DeclareMathOperator{\val}{Val}

\DeclareMathOperator{\tang}{Tan}

\newcommand{\card}{\operatorname{card}}

\newcommand{\interior}{\mathrm{int}}

\newcommand{\sm}{\setminus}

\title{Universal quasiconformal trees}

\author{Efstathios-K. Chrontsios-Garitsis}
\author{Fotis Ioannidis}
\author{Vyron Vellis}

\thanks{V.V. was partially supported by NSF DMS grant 2154918.}
\date{\today}
\subjclass[2020]{Primary 30L05; Secondary 30L10, 05C05, 28A80}
\keywords{Quasiconformal tree, geodesic tree, quasisymmetric embedding, universal space, Vicsek fractal, continuum self similar tree}

\address{Department of Mathematics\\ The University of Tennessee\\ Knoxville, TN 37966}
\email{echronts@utk.edu}
\address{Department of Mathematics\\ The University of Tennessee\\ Knoxville, TN 37966}
\email{fioannid@vols.utk.edu}
\address{Department of Mathematics\\ The University of Tennessee\\ Knoxville, TN 37966}
\email{vvellis@utk.edu}

\begin{document}
	
\begin{abstract}
A quasiconformal tree is a doubling (compact) metric tree in which the diameter of each arc is comparable to the distance of its endpoints. We show that for each integer $n\geq 2$, the class of all quasiconformal trees with uniform branch separation and valence at most $n$, contains a quasisymmetrically ``universal'' element, that is, an element of this class into which every other element can be embedded quasisymmetrically. We also show that every quasiconformal tree with uniform branch separation quasisymmetrically embeds into $\R^2$. Our results answer two questions of Bonk and Meyer in \cite{BM22}, in higher generality, and partially answer one question of Bonk and Meyer in \cite{BM20}.
\end{abstract}
	
\maketitle
	
\section{Introduction}
	
A collection $\mathscr{X}$ of metric spaces is \emph{topologically closed} if it has the property that if $X\in \mathscr{X}$ and $Y$ is a metric space homeomorphic to $X$, then $Y \in \mathscr{X}$. An element $X_0$ of a topologically closed collection $\mathscr{X}$ is \emph{universal}, if every element of $\mathscr{X}$ embeds into $X_0$. For example, the standard Cantor set is a universal set for the class of compact metric spaces with topological dimension 0, while the Menger sponge is a universal set for the class of compact metric spaces with topological dimension at most 1. More generally, it is known that for any $n\in \N_0$, there exists a universal set for the class of compact metric spaces with topological dimension at most $n$ \cite{Stanko71, Bestvina84}.
	
Another interesting and well studied topologically closed class is that of metric trees. A metric space $X$ is called a \emph{metric tree} (or \emph{dendrite}) if it does not contain simple closed curves, and it is a Peano continuum, i.e., compact, connected, and locally  connected. In other words, any two points $x,y$ can be joined by a unique arc that has $x,y$ as endpoints. On the one hand, trees are among the most simple 1-dimensional connected metric spaces, but on the other hand they have infinitely many topological equivalence classes. Nadler \cite{Nadler} proved that the class of metric trees contains a universal element.
	
Recently, there has been great interest in a subclass of metric trees, that of \emph{quasiconformal trees}, which plays an important role in analysis on metric spaces. See, for instance,  \cite{BM20,BM22,DV,DEBV,FreeGart1,FreeGart2} for a non-exhaustive list of references in the topic. A quasiconformal tree is a tree $T$ that satisfies two simple geometric conditions:
\begin{itemize}
\item $T$ is \emph{doubling}: there is a constant $N\geq 1$ such that any ball in $T$ can be covered by at most $N$ balls of half its radius, and
\item $T$ is \textit{bounded turning}: there is a constant $C\geq 1$ such that each pair of points $x,y \in T$ can be joined by a continuum whose diameter is at most $Cd(x,y)$.
\end{itemize}
The two conditions above are invariant under quasisymmetric maps, making the class of quasiconformal trees $\mathscr{QCT}$ \emph{quasisymmetrically closed}. Quasisymmetric maps are the natural generalization of conformal maps in the abstract metric setting. Roughly speaking, a homeomorphism $f$ between two connected and doubling metric spaces is quasisymmetric if for three points $x,y,z$ in the domain that have comparable mutual distances, the images $f(x),f(y),f(z)$ also have comparable mutual distances; see Section \ref{sec:prelim} for the precise definition.
	
Quasiconformal trees were introduced by Kinneberg \cite{Kinneberg}, and the term first appeared in the work of Bonk and Meyer \cite{BM20}. Quasiconformal trees have appeared in several fields of analysis and dynamics. First, if $T\subset \R^2$ is a tree, then $\R^2\setminus T$ is a \emph{John domain} if and only if $T$ is a quasiconformal tree \cite[Theorem 4.5]{NaVa}. Second, Julia sets of semihyperbolic polynomials (e.g., $z^2 +i$) are quasiconformal trees; see for example \cite{CJY} and \cite[p.95]{CG-dynamics}. In fact, a combination of \cite{CJY} and \cite{NaVa} gives that a point $c\in \mathbb{C}$ is a Misiurewicz point if and only if the Julia set of $P_c(z)=z^2+c$ is a quasiconformal tree, yielding a plethora of examples (see \cite{BM22} for a longer discussion on this direction). Third, quasiconformal trees $T$ in $\R^2$ (often called \emph{Gehring trees}) were recently characterized by Rohde and Lin \cite{RohdeLin} in terms of the laminations of the conformal map $f: \mathbb{C}\setminus\mathbb{D} \to \mathbb{C}\setminus T$. Fourth, planar quasiconformal trees can be classified in the setting of Kleinian groups \cite{McMullen}. Fifth, there exist planar quasiconformal trees (``\textit{antenna sets}''), whose conformal dimension is not attained by any quasisymmetric map \cite{BishopTyson}; see also \cite{Azzam}. Finally, some planar quasiconformal trees with finitely many branch points are \emph{conformally balanced}, that is, for every edge, each side has exactly the same harmonic measure with pole at infinity \cite{Bishop}.
	
Quasiconformal trees generalize two well-studied types of spaces. On the one hand, quasiconformal trees that are arcs (i.e., have no branch points) are called \emph{quasi-arcs}, and have been studied in complex analysis and analysis on metric spaces for decades \cite{GH}. On the other hand, quasiconformal trees generalize doubling geodesic trees. Geodesic trees are trees in which, for each pair of points $x, y$, the unique arc joining them has (finite) length equal to the distance of $x, y$. Thus, in geodesic trees all paths are ``straight'' (isometric to intervals in the real line), whereas paths in quasiconformal trees may be fractal, like the von Koch snowflake. Geodesic trees (and their non-compact counterparts, the $\R$-trees) are standard objects of study in geometric group theory \cite{Bestvinabook} and in computer science \cite{LeGall}. A theorem of Bonk and Meyer \cite{BM20} states that every quasiconformal tree is quasisymmetrically equivalent to a geodesic doubling tree.
	
Two of the most well-known and used quasiconformal trees are the \emph{continuum self similar tree} (CSST) and the \emph{Vicsek fractal}. See Figure \ref{fig:CSST}. Both are attractors of iterated function systems of similarities in $\R^2$, but we omit their formal definitions. The CSST is a trivalent quasiconformal tree that has been studied in \cite{BT_CSST, BM22} and is almost surely homeomorphic to the \emph{(Brownian) continuum random tree} introduced by Aldous \cite{Aldous1,Aldous2}. The \emph{Vicsek fractal}, is a 4-valent metric tree of great importance in analysis, probability, and physics; see for instance \cite{Vicsek,Metz,HamblyMetz,Zhou,CSW,BaudoinChen1,BaudoinChen2}.
	
	\begin{figure}[h]
		\centering
		\begin{minipage}{0.55\textwidth}
			\centering
			\includegraphics[width=0.9\textwidth]{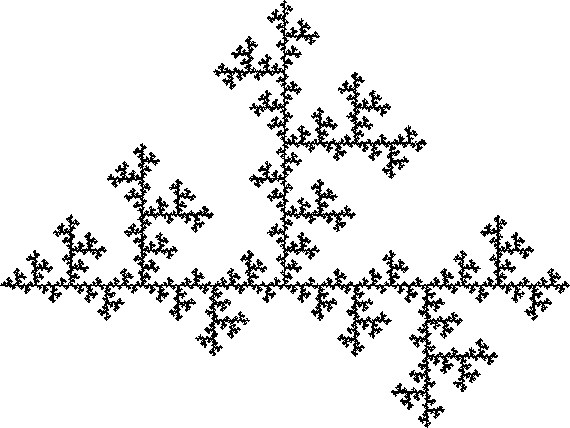}
		\end{minipage}\hfill
		\begin{minipage}{0.35\textwidth}
			\centering
			\includegraphics[width=0.9\textwidth]{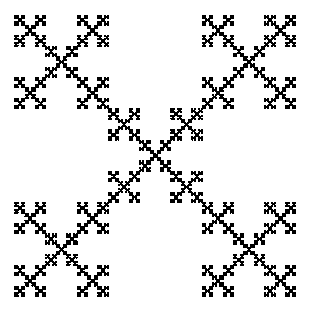}
		\end{minipage}
		\caption{The CSST (left) and the Vicsek fractal (right).}
		\label{fig:CSST}
	\end{figure}
	
Given that quasisymmetric maps ``quasi-preserve" many interesting geometric properties, it is a natural question to ask whether there exists a quasiconformal tree which is \emph{quasisymmetrically universal} for the class $\mathscr{QCT}$. In other words, does there exist a quasiconformal tree $T$ such that every element in $\mathscr{QCT}$ quasisymmetrically embeds into $T$? The answer can easily be seen to be negative. Indeed, if there was a bounded turning tree such that every element of $\mathscr{QCT}$ quasisymmetrically embeds in, then its valence would be infinite, which implies that it could not have been doubling \cite[Lemma 3.5]{BM22}. Consequently, it is natural to pose this question for the more restricted class $\mathscr{QCT}(n)$, the class of quasiconformal trees of valence at most $n$, for some integer $n\geq 3$. Unfortunately, even with this restriction, the answer is still negative already for the smallest class $\mathscr{QCT}(3)$.
	
\begin{thm}\label{prop:nouniv}
There exists no quasiconformal tree $T$ for which every element in $\mathscr{QCT}(3)$ quasisymmetrically embeds into $T$.
\end{thm}
	
The reasoning behind Theorem \ref{prop:nouniv} is that, although trees in $\mathscr{QCT}(3)$ have valence 3, infinitesimally, they can have arbitrarily large valencies. As a result, similarly to the above observation, if there was a bounded turning tree $T$ in which every element of $\mathscr{QCT}(3)$ quasisymmetrically embeds, then, infinitesimally, $T$ would have infinite valence and it could not have been doubling. See Section \ref{sec:noQSembed} for the proof.
	
In \cite{BM22}, Bonk and Meyer studied quasiconformal trees where large branches can not be too close. To make this precise, we first define the height of a tree at a branch point. Given a metric tree $T$ and a  point $p\in T$, denote by $B_{T}^1(p), B_{T}^2(p),\dots$ the components of $T\setminus \{p\}$, which we call \emph{branches} of $T$ at $p$. If there are at least three branches of $T$ at $p$, we say that $p$ is a \textit{branch point}. By \cite[Section 3]{BT_CSST}, only finitely many of the branches of $p$ can have a diameter exceeding a given positive number. Hence, we can label the branches so that $\diam{B_{T}^1(p)} \geq \diam{B_{T}^2(p)} \geq \cdots$ and define the \emph{height} of $p$ in $T$ as 
\begin{equation}\label{eq:height}
H_{T}(p):=\diam{B_{T}^3(p)}.
\end{equation}
	
\begin{definition}[{Uniform relative separation of branch points}]\label{def:unifsep}
A metric tree $(T,d)$ is said to have \textit{uniformly (relatively) separated branch points} if there exists $C\geq 1$ such that for all distinct branch points $p,q\in T$,
$$d(p,q)\geq C^{-1}\min\{H_{T}(p),H_{T}(q)\}.$$ 
\end{definition}
	
We denote by $\mathscr{QCT}^*$ the collection of all quasiconformal trees of uniform branch separation, and by $\mathscr{QCT}^*(n)$ the collection of all elements in $\mathscr{QCT}^*$ that have valence at most $n$. The class $\mathscr{QCT}^*$ (and each class $\mathscr{QCT}^*(n)$) is quasisymmetrically closed by \cite[Lemma 4.3]{BM22}. Moreover, most examples of quasiconformal trees appearing in analysis \cite{CG-dynamics,CJY,BishopTyson,Bishop,RohdeLin} are of uniform branch separation, making $\mathscr{QCT}^*$ a very natural class.
	
Bonk and Meyer \cite[Problem 9.2]{BM20} asked whether the CSST is a quasisymmetrically universal element for $\mathscr{QCT}^*(3)$. More generally, they asked if $\mathscr{QCT}^*(n)$ contains a quasisymmetrically universal element and whether there exists a bounded turning tree in which we can quasisymmetrically embed every element of $\mathscr{QCT}^*$ \cite[Problem 9.3]{BM22}. In our main theorem below, we answer all these questions.
	
\begin{thm}\label{thm:main}
For each $\e>0$ there exists a nested family of geodesic, self-similar, and Ahlfors regular metric trees $\mathbb{T}^{2} \subset \mathbb{T}^{3} \subset \mathbb{T}^{4} \subset \cdots$ with the following properties.
\begin{enumerate}[(i)] 
\item Each $\mathbb{T}^n$ is quasisymmetrically equivalent to a quasiconvex tree in $\R^2$. In particular, $\T^2$ is isometric to the segment $[0,1]$, $\T^3$ is quasisymmetrically equivalent to the CSST, and $\T^4$ is quasisymmetrically equivalent to the Vicsek fractal.
\item The sequence of trees $(\mathbb{T}^n)_{n=2}^{\infty}$ converges in the Gromov-Hausdorff sense to a geodesic tree $\mathbb{T}^{\infty}$ whose Hausdorff dimension is at most $1+\e$.
\item For each $n\geq 3$, $\mathbb{T}^n$ is quasisymmetrically universal in $\mathscr{QCT}^*(n)$.
\end{enumerate}
\end{thm}
	
Recall that quasiconvexity is a notion slightly weaker than geodesicity: a metric space is \emph{quasiconvex} if any two points can be joined by a path with length at most a fixed multiple of the distance of the points. Since Euclidean spaces do not contain geodesic trees (other than line segments), quasiconvexity is the closest to geodesicity in the Euclidean setting.
	
A few remarks are in order. First, each $\T^n$ is topologically universal in the class of all metric trees that have valence at most $n$ \cite{Charatonik}. Second, $\T^{\infty}$ is homeomorphic to the universal tree of Nadler \cite{Nadler} and, hence, is itself a topologically universal tree in the class of all metric trees. 
	
Claims (i) and (iii) of Theorem \ref{thm:main} imply that the CSST is quasisymmetrically universal in $\mathscr{QCT}^*(3)$ and that the Vicsek fractal is quasisymmetrically universal in $\mathscr{QCT}^*(4)$. The former answers \cite[Problem 9.2]{BM22}. Claims (ii) and (iii) of Theorem \ref{thm:main} imply that every tree in $\mathscr{QCT}^*(n)$ quasisymmetrically embeds into $\T^{\infty}$, which answers \cite[Problem 9.3]{BM22}.
	
\begin{cor}
If $T$ is a quasiconformal tree with uniformly separated branch points, then $T$ quasisymmetrically embeds into $\mathbb{T}^{\infty}$.
\end{cor}
	
By the doubling property, every quasiconformal tree embeds quasisymmetrically into some $\R^n$ (see for instance \cite[Theorem 12.1]{Heinonen}). On the other hand, no quasiconformal tree (unless it is an arc) can be embedded into $\R$. Thus, a very natural question of Bonk and Meyer \cite[Section 10]{BM20} asks whether every quasiconformal tree admits a quasisymmetric embedding into $\R^2$, and whether the embedded image is quasiconvex. A positive answer to this question could be instrumental in the study of quasiconformal trees as Julia sets of holomorphic or uniformly quasiregular maps of $\R^2$. Theorem \ref{thm:main} gives a partial answer to this question.
	
\begin{cor}
If $T$ is a quasiconformal tree with uniformly separated branch points, then $T$ quasisymmetrically embeds into $\R^2$ and the image is quasiconvex.
\end{cor}
	
As mentioned in Theorem \ref{thm:main}, $\T^2$ is a line segment. The precise construction of trees $\T^n$ with $n\geq 3$ is given in Section \ref{sec:construction}, but we briefly describe an equivalent way of constructing them, motivated by the methods in \cite[Section 9]{BM22}. Fix an integer $n\geq 3$ and numbers $\frac12 \geq a_3 \geq \cdots \geq a_n>0$. Let $T_1^n$ be a line segment of length 1 and let $p$ be the midpoint. We geodesically glue  $n-2$ many segments at $p$ of lengths $a_3,\dots, a_n$ (see Definition \ref{def:geodesic gluing}), and equip the resulting space $T_2^n$ with the path metric. Note that $T_2^n$ is a geodesic tree with 1 branch point and $n$ segments of diameters $\frac12,\frac12,a_3,\dots,a_n$. We repeat the same process to each of the $n$ segments to obtain $T_3^n$, and we keep iterating the construction infinitely. It is easy to see that each $T_i^n$ is a geodesic tree that isometrically embeds into $T_{i+1}^n$. The Gromov-Hausdorff limit of the sequence $(T_i^n)_{i\in\N}$ is the space $\T^n$. Despite the simplicity of this process, many of the desired properties for $\T^n$ seem much harder to establish this way. On the other hand, the equivalent construction of $\T^n$ described in Section \ref{sec:construction} allows the use of various tools from \cite{DV}, significantly shortening many arguments (see, for instance, Lemma \ref{lem:metrictree}, Lemma \ref{lem:ss}, Section \ref{sec:geodesicity}, Section \ref{sec:dimension}).

\begin{figure}[h]
\centering
\includegraphics[width=0.7\linewidth]{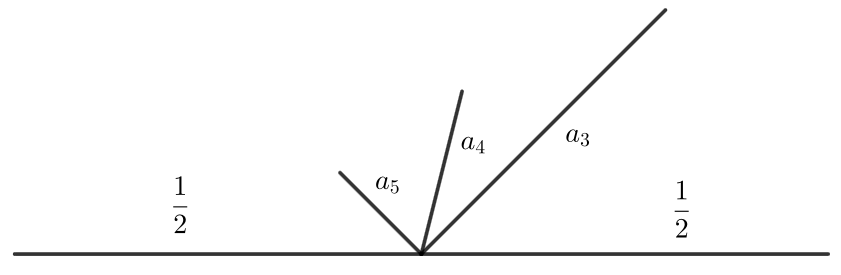}
\caption{Geodesic tree $T_2^n$ for $n=5$. To obtain $T_3^5$, replace all segments by rescaled copies of $T_2^5$.}
\label{fig:simpleT5}
\end{figure}
	
The proof of Theorem \ref{thm:main} comprises of two main steps. The first step is a quasisymmetric uniformization theorem for quasiconformal trees with uniform branching in the spirit of \cite[Theorem 1.4]{BM22}. Before we state it, we require certain definitions. 
	
\begin{definition}[{Uniform relative density of branch points}]\label{def:unifden}
A metric tree $(T,d)$ is said to have \textit{uniformly (relatively) dense branch points} if there exists $C\geq 1$ such that for all $x,y\in T$, $x\neq y$, there exists a branch point $p\in [x,y]$ with 
$$H_{T}(p)\geq C^{-1} d(x,y).$$
\end{definition}
	
\begin{definition}[{Uniform branch growth}]\label{def:unifgrowth} 
A metric tree $(T,d)$ is said to have \textit{uniform branch growth} if there exists $C\geq 1$ such that for all branch points $p\in T$ and all $i\geq 3$, we have 
\[ \diam{B^i_T(p)} \geq C^{-1}\diam{B^3_T(p)}.\]
\end{definition}
	
Definitions \ref{def:unifsep}, \ref{def:unifden}, and \ref{def:unifgrowth} combined lead to the concept of uniform branching.
	
\begin{definition}[{Uniformly $n$-branching trees}]
Given $n\in\{3,4,\dots\}$, a tree is called \textit{uniformly $n$-branching} if it has uniform branch growth, uniformly relatively separated branch points, uniformly relatively dense branch points, and all of its branch points have valence equal to $n$.
\end{definition}
	
For completeness, we say that a quasiconformal tree $T$ is \emph{uniformly $2$-branching} if it is an arc. It can be shown that for each $n\geq 2$, the quasiconformal tree $\T^n$ is uniformly $n$-branching. The main result of \cite{BM22} states that a quasiconformal tree is quasisymmetric to the CSST if, and only if, it is uniformly 3-branching. We also prove the analogue of this result for all valencies, where CSST is replaced by $\T^n$. Note that in the case $n=3$, the uniform branch growth condition is trivially true, while it is  essential for $n\geq 4$, as a quasisymmetric invariant property of $\T^n$.
	
\begin{thm}\label{thm:uniformization}
Let $n\in\{2,3,\dots\}$. A quasiconformal tree $T$ is uniformly $n$-branching if, and only if, it is quasisymmetrically equivalent to $\T^n$.
\end{thm}

The second ingredient in the proof of Theorem \ref{thm:main} is the construction of three quasisymmetric embeddings in Section \ref{sec:QSembed}. First, we show that each quasiconformal tree with uniform branch separation quasisymmetrically embeds into a quasiconformal tree with uniform branch separation and uniform branch growth. Second, we show that each quasiconformal tree with uniform branch separation and uniform branch growth quasisymmetrically embeds into a quasiconformal tree with uniform branch separation, uniform branch growth, and all of its branch points have the same valence $n$. Finally, we employ a tiling of quasiconformal trees from \cite{BM20} to show that each quasiconformal tree with uniform branch separation, uniform branch growth, and whose branch points all have valence $n$, quasisymmetrically embeds into a uniformly $n$-branching tree.  
	
The work of this paper motivates a couple of questions. First, as mentioned earlier, the class $\mathscr{QCT}$ cannot have a quasisymmetrically universal element, due to such an element not being doubling. One can then drop the doubling requirement and study the more general class of bounded turning trees, which is quasisymmetrically closed, and also contains all geodesic trees. In this class, we ask whether there exists a universal element. It is easy to see that if such an element exists, it can not be $\T^{\infty}$, as $\T^{\infty}$ has uniform branch separation, and every metric tree that quasisymmetrically embeds therein must also have uniform branch separation.
	
\begin{question}\label{quest:1}
Does there exist a bounded turning metric tree in which all bounded turning trees quasisymmetrically embed? Does every bounded turning metric tree with uniform branch separation quasisymmetrically embed into $\T^{\infty}$? 
\end{question}
	
In the absence of the doubling property, Question \ref{quest:1} can also be studied for the larger class of weakly quasisymmetric maps (see Section \ref{sec:prelim}).
	
The second question is concerned with the \emph{bi-Lipschitz} version of universality. It was proved in \cite{DEBV} by David, Eriksson-Bique, and the third author that  every quasiconformal tree bi-Lipschitz embeds in some Euclidean space $\R^M$, where $M$ only depends on the doubling and bounded turning constants of the tree (see also \cite{BLtrees_Lee_Naor, BLtrees_Gupta}). It is natural to consider the same question, where the Euclidean target space is replaced by a suitable model tree. Note that, contrary to quasisymmetric mappings, bi-Lipschitz mappings preserve all notions of dimension (Hausdorff, Minkowski, Assouad). Therefore, one has to consider a class of spaces of uniformly bounded dimension (see Lemma \ref{lem:doubling}).
	
\begin{question}\label{quest:2}
Let $\mathscr{GT}(n,c)$ be the class of all geodesic trees that have valencies at most $n$ and uniform branch separation with constant $c$. Does there exist an element $T_0 \in \mathscr{GT}(n,c)$ in which every element of $\mathscr{GT}(n,c)$ bi-Lipschitz embeds? If such a $T_0$ exists, is it bi-Lipschitz homeomorphic to $\T^n$?
\end{question}
	
We note that the class $\mathscr{GT}(n,c)$ is not bi-Lipschitz closed, but one could instead consider the ``bi-Lipschitz closure'' of this class, i.e., all metric trees that are bi-Lipschitz equivalent to some tree in $\mathscr{GT}(n,c)$. We also note that, in addition to the question of the quasisymmetric embedability of all quasiconformal trees in $\R^2$, it is unknown and an interesting problem whether all quasiconformal trees with Assouad dimension less than 2 bi-Lipschitz embed into $\R^2$ \cite[Question 1]{Kinneberg}.
	
\subsection*{Outline of the paper}
In Section \ref{sec:prelim} we present preliminary definitions and results. In Section \ref{sec:noQSembed}, employing the notions of metric tangents by Gromov, we prove Theorem \ref{prop:nouniv}. 
	
In Section \ref{sec:construction}, following the language and techniques from \cite{DV}, we construct for each $n\in\N\cup\{\infty\}$ and for each choice of weights $\ba = (a_1,a_2,\dots)$ (with $\frac12 = a_1 = a_2 \geq a_3 \geq \cdots$) a self-similar metric tree $\T^{n,\ba}$. In Section \ref{sec:geodesicity} we show that the trees $\T^{n,\ba}$ are geodesic, and we show that the sequence $(\T^{n,\ba})_{n\geq 2}$ converges in the Gromov-Hausdorff sense to $\T^{\infty,\ba}$.
	
In Section \ref{sec:branching} we study the branch points of trees $\T^{n,\ba}$ and show that $\T^{n,\ba}$ are uniformly $n$-branching. In Section \ref{sec: Vicsek unif branching} we show that the Vicsek fractal is uniformly $4$-branching. In Section \ref{sec:dimension} we show that the trees $\T^{n,\ba}$ are Ahlfors regular (which implies doubling), and estimate the dimension of the limit tree $\T^{\infty,\ba}$. The proof of a quantitative version of Theorem \ref{thm:uniformization} is given in Section \ref{sec:unif}, by following the techniques from \cite{BT_CSST} and \cite{BM22}. In Section \ref{sec:BLembed} we show that trees $\T^{n,\ba}$ quasisymmetrically embed into $\R^2$, with the embedded image being a quasiconvex tree. 
	
In Section \ref{sec:glue} we establish some preliminary lemmas for geodesic trees and in  Section \ref{sec:QSembed} we prove that every quasiconformal tree with uniformly separated branch points and valence at most $n$ quasisymmetrically embeds into $\T^{n,\ba}$. Finally, in Section \ref{sec:Finale} we prove Theorem \ref{thm:main} and Theorem \ref{thm:uniformization}, by putting together all the results from the earlier sections.

\subsection*{Acknowledgments} The authors wish to kindly thank the anonymous referees for their helpful comments that improved the exposition of the paper.
	
\section{Preliminaries}\label{sec:prelim}
We write $\N_0=\N\cup \{0\}.$ Given two non-negative quantities $A$ and $B$, we write $A\lesssim B$ if there is a \textit{comparability constant}  $C=C(\lesssim)$ such that $A\leq C B$. Similarly,  we write $A\gtrsim B$ if there is $C=C(\gtrsim)$ such that $A\geq  B/C$. If $A\lesssim B$ and $A\gtrsim B$ we write $A\simeq B$.
	
Given a set $E$ in a metric space $X$ and $r>0$, we write
\[ N_r(E) = \{x\in X : \dist(x,E) < r\}.\]
	
\subsection{Convergence of metric spaces}
Recall that the \emph{Hausdorff distance} between two subsets $X_1,X_2$ of the same  space $X$ is defined by
\begin{equation*}
\dist_H(X_1,X_2):= \inf \{ r>0: X_2 \subset N_r(X_1)\text{ and }X_1 \subset N_r(X_2)\} 
\end{equation*}
The \emph{Gromov-Hausdorff distance} of two metric spaces $X$ and $Y$ is defined by
\begin{equation*} 
\dist_{GH}(X,Y) := \inf_{f,g}\{\dist_H(f(X),g(Y))\}
\end{equation*}
where $f : X \to Z$ and $g : Y \to Z$ are isometric embeddings into some ambient metric space $Z$. A sequence of compact metric spaces $(X_n)_{n=1}^{\infty}$ converges in the Gromov-Hausdorff sense to a compact metric space $X$ (and we write $X_n \xrightarrow{GH} X$) if $\dist_{GH}(X_n,X) \to 0$ as $n \to \infty$. Notice that the limit space $X$ is unique up to isometries. If the spaces $X_n$ and $X$  are in the same ambient space for all $n\in\N$, then $\dist_H(X_n,X) \to 0$ implies $X_n \xrightarrow{GH} X$.
	
A map $f:X\to Y$ is called an $\e$-isometry if 
\[ \dist(f):=\sup\{|d_X(x_1,x_2)-d_Y(f(x_1),f(x_2)| : x_1,x_2 \in X \} \leq \e.\]
A pointed metric space is a triple $(X, p, d)$, where $(X, d)$ is a metric space with a base point $p\in X$. The \emph{pointed Gromov-Hausdorff convergence} is an analog of Gromov-Hausdorff convergence appropriate for non-compact spaces. A sequence of pointed metric spaces $(X_n, p_n, d_n)$ converges in the pointed Gromov-Hausdorff sense to a complete pointed metric space $(X, p, d_X)$, and we write 
\[ (X_n, p_n, d_n) \xrightarrow{GH} (X, p, d),\] 
if for every $r > 0$ and every $\e\in(0,r)$ there exists $n_0\in\N$ such that for every integer $n > n_0$ there exists an $\e$-isometry $f : B(p_n, r) \to X$ with $f(p_n)=p$ and $B(p,r-\e) \subset N_{\e}(f(B(p_n,r)))$.
	
Recall that if a metric  space $X$ is doubling with a doubling constant $C\geq 1$, we say that $X$ is $C$-doubling. The next lemma gives a relation between the doubling constants of the metric spaces and their pointed Gromov-Hausdorff limit. For the proof see \cite[Theorem 8.1.10]{BBI}.
	
\begin{lem}\label{lem:doubling of weak tangent}
Let $(X_n,p_n,d_{X_n})$ be a sequence of $C$-doubling spaces. Then, the pointed Gromov-Hausdorff limit is also $C$-doubling. 
\end{lem}
	
\subsection{Self-similarity}\label{sec:ss}
We say that a compact metric space $X$ is \emph{self-similar} if there exists a finite collection of similarity maps $\{\phi_{i}: X \to X\}_{i=1,\dots,k}$ such that $X = \bigcup_{i=1}^k\phi_{i}(X)$, and we call $X$  the \textit{attractor} of the collection. Moreover, we say that $X$ satisfies the \textit{open set condition} if there exists a nonempty open set $U \subset X$ such that $\phi_{i}(U) \subset U$ and $\phi_{i}(U)\cap \phi_{j}(U) = \emptyset$ for all distinct $i,j \in \{1,\dots,k\}$. Throughout the paper, all self-similar sets are assumed to satisfy the open set condition, even if not explicitly stated.
	
\subsection{Metric trees}\label{sec: metric trees} 
	
We say a metric space $T$ is a \textit{(metric) tree} if $T$ is  compact, connected, and locally connected, with at least two distinct points, and such that for any $x,y\in T$  there is a unique arc in $T$ with endpoints $x, y$. We denote this unique arc by $[x,y]$. We say a point $p\in T$ has \textit{valence} equal to $m\in \N\cup \{\infty\}$ and write $\val(p,T)=m$ if the complement $T\setminus\{p\}$ has $m$ connected components, called \textit{branches of $T$ at $p$}. If $\val(p,T)=1$ we say $p$ is a \textit{leaf} of $T$, if $\val(p,T)=2$ we say $p$ is a \textit{double point} of $T$, and if $\val(p,T)\geq 3$ we say $p$ is a \textit{branch point} of $T$. We define the \textit{valence} of a tree $T$ by
\[ \val(T):=\sup \{\val(p,T): \, p \in T\}.\] 
If $T$ has no branch points, we say that $T$ is a \textit{metric arc}. If $T$ has at least one branch point and there exists $m\in\{3,4,\dots\}$ such that $\val(p,T)=m$ for all branch points $p$ of $T$, then we say that $T$ is \textit{$m$-valent}. 
	
A subset $S$ of a tree $T$ is called \textit{subtree} of $T$ if $S$ equipped with the restricted metric of $T$ is also a tree. By \cite[Lemma 3.3]{BT_CSST}, $S$ is a subtree of $T$ if and only if $S$ contains at least 2 points and it is closed and connected.
	
The following lemma is used towards the proof of Theorem \ref{thm:uniformization}. It is an analogue of \cite[Theorem 5.4]{BT_CSST}.
	
\begin{lem}\label{lem:homeo of leaves}
Let $S$ and $S'$ be $m$-valent trees, $m\geq 3$, with dense branch points. Moreover, suppose that $p_1,p_2,p_3$ and $q_1,q_2,q_3$ are three distinct leaves of $S$ and $S'$ respectively. Then there exists a homeomorphism $f:S\to S'$ such that $f(p_k)=q_k$ for $k=1,2,3$. 
\end{lem}
	
The following proposition is needed in the proof of Lemma \ref{lem:homeo of leaves}.
	
\begin{prop}\cite[Proposition 2.1]{BT_CSST}\label{prop: homeo}
Let $(X,d_X)$ and $(Y,d_Y)$ be compact metric spaces. Suppose that for each $n\in\N$, the spaces $X,Y$ admit decompositions $X=\bigcup_{i=1}^{M_n}X_{n,i},Y=\bigcup_{i=1}^{M_n}Y_{n,i}$ as finite union of non-empty compact subsets $X_{n,i},Y_{n,i}$, $i=1,\dots, M_n\in\N$, with the following properties for all $n,i$ and $j$:
\begin{enumerate}
\item Each set $X_{n+1,j}$ is the subset of some set $X_{n,i}$ and each $Y_{n+1,j}$ is the subset of some set $Y_{n,i}$.
\item Each set $X_{n,i}$ is equal to the union of some of the sets $X_{n+1,j}$ and each set $Y_{n,i}$ is equal to the union of some of the sets $Y_{n+1,j}$.
\item We have that $\max_{1\leq i\leq M_n}\diam(X_{n,i})\to 0$ and $\max_{1\leq i\leq M_n}\diam(Y_{n,i})\to 0$ as $n\to\infty$.
\end{enumerate}  
Moreover, assume that $X_{n+1,j}\subset X_{n,i}$ if and only if $Y_{n+1,j}\subset Y_{n,i}$ and $X_{n,i}\cap X_{n,j}\neq \emptyset$ if and only if $Y_{n,i}\cap Y_{n,j}\neq\emptyset$ for all $n,i,j$.
		
Then there exists a unique homeomorphism $f:X\to Y$ such that $f(X_{n,i})=Y_{n,i}$ for all $n$ and $i$. 
\end{prop}

\begin{proof}[Proof of Lemma \ref{lem:homeo of leaves}]
We need the terminology on \textit{marked leaves} from \cite{BT_CSST}. Fix a set $A=\{1,\dots, m\}$ and let $T$ be an arbitrary metric tree. Suppose we have defined subtrees $T_u$ of $T$ for all levels $n\in\N$ and all $u\in A^n$. The boundary $\partial T_u$ of $T_u$ in $T$ will consist of one or two points that are leaves of $T_u$ and branch points of $T$. We consider each point in $\partial T_u$ as a \textit{marked leaf} in $T_u$ and will assign to it an appropriate sign $-$ or $+$ so that if there are two marked leaves in $T_u$, then they carry different signs. Accordingly, we refer to the points in $\partial T_u$ as the \textit{signed marked leaves} of $T_u$. The same point may carry different signs in different subtrees. We write $p^{-}$ if a marked leaf $p$ of $T_u$ carries the sign $-$, and $p^{+}$ if it carries the sign $+$. If $T_u$ has exactly one marked leaf, we call $T_u$ a \textit{leaf-tile}, and if there are two marked leaves we call $T_u$ an \textit{arc-tile}.
		
We now describe how to define subtrees of $S$ of level $1$. Let $p_1$ carry the sign $-$, while $p_2$ and $p_3$ carry the sign $+$. We create a decomposition of $S$ and a decomposition of $S'$ satisfying the conditions of Proposition \ref{prop: homeo}, which provides the desired homeomorphism. In particular, choose a branch point $c\in S$ so that the leaves $p_1,p_2,p_3$ lie in distinct branches $S_1,S_2,S_3$ of $S$ respectively. To find such a branch point, one travels from $p_1$ along $[p_1,p_2]$ until one first meets $[p_2,p_3]$ in a point $c$. Then the sets $[p_1,c),\ [p_2,c),\ [p_3, c)$ are pairwise disjoint. 
		
For $k,l\in \{1,2,3\}$ with $k\neq l$ the set $[p_k,c)\cup \{c\}\cup (c,p_l]$ is an arc with endpoints $p_k$ and $p_l$, and so it must agree with $[p_k,p_l]$. In particular, $c\in [p_k,p_l]$. Since each point $p_k$ is a leaf, it easily follows from \cite[Lemma 3.2(iii)]{BT_CSST} that $c\neq p_1,p_2,p_3$. We conclude that the connected sets $[p_1,c),\ [p_2,c),\ [p_3,c)$ are non-empty and must lie in different branches $S_1,S_2,S_3$ of $S$ at $c$. Hence, the point $c$ has at least 3 branches, and since $S$ is an $m$-valent tree, $c$ has to be a branch point of $S$ with a total of $m$ branches. We can choose the labels so that $p_k\in S_k$ for $k\in \{1,2,3\}$. We then have the set of marked leaves $\{p_1^{-}, c^{+}\}$ in $S_1$, $\{c^{-}, p_2^{+}\}$ in $S_2$, and $\{c^{-}, p_3^{+}\}$ in $S_3$.
		
We now continue inductively as in \cite[Section 5, pp. 178-182]{BT_CSST}. If we already constructed a subtree $S_u$ for some $n\in\N$ and $u\in A^n$ with one or two signed marked leaves, then we decompose $S_u$ into $m$ branches labeled $S_{u1},\dots, S_{um}$ by using a suitable branch point $c\in S_u$. Namely, if $S_u$ is a leaf-tile and has one marked leaf $a\in S_u$, we choose a branch point $c\in S_u\setminus\{a\}$ with maximal height $H_S(c)$. If $S_u$ is an arc-tile with two marked leaves $\{a,b\}\subset S_u$ we choose a branch point $c\in S_u$ of maximal height on $(a,b)\subset S_u$.
		
A marked leaf $x^{-}$ of $S_u$ is passed to $S_{u1}$ with the same sign. Similarly, a marked leaf $x^{+}$ of $S_u$ is passed to $S_{u2}$ with the same sign. If we continue in this manner, we obtain a  subtree $S_u$ with one or two signed marked leaves for all levels $n\in\N$ and $u\in A^n$.
		
We apply the same procedure for the tree $S'$ and its leaves $q_1,q_2,q_3$. Then we apply \cite[Lemmas 5.1, 5.2, 5.3]{BT_CSST} to the decompositions $S$ and $S'$ obtained this way. Hence, by Proposition \ref{prop: homeo} there exists a homeomorphism $f:S\to S'$ such that
\begin{equation}\label{f(S_u)=S'_u}
f(S_u)=S_u' \qquad \text{for all $n\in\N$ and $u\in A^n$.} 
\end{equation}
Note that $p_1$ carries the sign $-$ to $S_{1^{(n)}}$ for all $n\in\N$, and since by \cite[Lemma 5.3]{BT_CSST} the diameters of our subtrees $S_u$, $u\in A^n$, approach 0 uniformly as $n\to\infty$, we have $\{p_1\}=\bigcap_{n=1}^{\infty}S_{1^{(n)}}$. The same argument shows that $\{q_1\}=\bigcap_{n=1}^{\infty}S_{1^{(n)}}'$. Hence, \eqref{f(S_u)=S'_u} implies that $f(p_1)=q_1$. 
		
Similarly, the points $p_2,p_3,q_2,q_3$ carry the sign $+$ in their respective trees. Therefore, by applying the same arguments we have that $f(p_2)=q_2$ and $f(p_3)=q_3$. The proof is complete.
\end{proof}
	
\begin{rem}
\cite[Theorem 5.4]{BT_CSST} is stated only for $3$-valent trees and the difference in their inductive decomposition (\cite[Section 5, pp. 178-182]{BT_CSST}) is only in the basis step. In our case, the same techniques apply for $m$-valent trees, as long as we map three leaves of the tree $S$ to three leaves of the tree $S'$.
\end{rem}
	
The next lemma, which is extensively used in Section \ref{sec:QSembed}, demonstrates how geodesicity and uniform separation on trees relate to the doubling condition.
\begin{lem}\label{lem:doubling}
Given $N\in \{2,3,\dots\}$, if $T$ is a geodesic metric tree with $\val(T) \leq N$ and uniform branch separation with constant $C\geq 1$, then there is $D=D(N,C)>0$ such that $T$ is $D$-doubling.
\end{lem}
	
\begin{proof}
Let $x_0\in T$ and let $r>0$. We construct 3 finite sets $\mathcal{S}_1, \mathcal{S}_2, \mathcal{S}_3 \subset \overline{B}(x_0,r)$. 
		
For each $x\in T$ with $d(x,x_0)\geq r$, let $p_x$ be the unique point of $[x,x_0]$ such that $d(p_x,x_0)=3r/4$. Let $\mathcal{S}_1=\{p_x : x\in T, \, d(x,x_0)\geq r\}$. Let $q_1,\dots,q_l \in \mathcal{S}_1$ be distinct points and let $K$ be the  convex hull of $\{x_0,q_1,\dots,q_l\}$ in $T$ (in the geodesic sense). Then $x_0\in K$ and $\diam{K}$ is either $3r/4$ or $3r/2$; the former being true exactly when $x_0$ is a leaf of $K$. Without loss of generality, we assume for the rest of the proof that $\diam{K}=3r/2$ and, consequently, that $d(q_1,q_2) = 3r/2$. For each $i\in\{3,\dots,l\}$, let $z_i \in [q_1,q_2]$ be the point with $[q_i,x_0]\cap [q_1,q_2] = [z_i,x]$. Note that each $z_i$ is a branch point with $H_{T}(z_i) \geq r/4$. Since all $z_i$ are on the geodesic $[q_1,q_2]$ of length $3r/2$, by uniform branch separation we have that $\card\{z_3,\dots,z_l\}\leq 6C+1$. Note that for $i, j\in \{3, \dots, l\}$ with $i\neq j$, we do not necessarily have that $z_i\neq z_j$, hence $\card\{z_3,\dots,z_l\}\leq l$. However, each $z_i$ can have at most $N$ branches, so it follows that 
$$l=\card\{q_1,\dots, q_l\} \leq (N-2)\card\{z_3,\dots,z_l\} + 2\leq (N-2)(6C+1) + 2.$$
Since $l$ was the cardinality of an arbitrary subset of distinct points in $\mathcal{S}_1$, the above proves that 
\[ \card{\mathcal{S}_1} \leq (N-2)(6C+1) + 2.\]
		
For each $x\in T$ with $d(x,x_0) \geq \frac34 r$, let $p_x'$ be the unique point in $[x_0,x]$ with $d(p_x',x_0) = r/2$. Let also $\mathcal{S}_2 =\{p_x': x\in T ,\, d(x,x_0)\geq 3r/4\}$. Working as in the preceding paragraph, we can show that
\[ \card{\mathcal{S}_2} \leq (N-2)(6C+1) + 2. \]
		
Finally, for each $x\in T$ with $d(x,x_0) \geq r/2$, let $p_x''$ be the unique point in $[x_0,x]$ with $d(p_x'',x_0) = r/4$. Let also $\mathcal{S}_3 =\{p_x'': x\in T, \, d(x,x_0)\geq r/2\}$. As before,
\[ \card{\mathcal{S}_3} \leq (N-2)(6C+1) + 2. \]
		
Define now $\mathcal{S} = \{x_0\}\cup \mathcal{S}_1\cup \mathcal{S}_2\cup \mathcal{S}_3$. We claim that for all $x\in \overline{B}(x_0,r)$, there exists $p\in \mathcal{S}$ such that $d(x,p) \leq r/2$. To this end, fix $x\in \overline{B}(x_0,r)$. If $d(x,x_0) \leq r/2$, then we can choose $p=x_0$. If $r/2 < d(x,x_0) < 3r/4$, then $d(x,p_{x}'') \leq r/2$. If $3r/4 \leq d(x,x_0) < r$, then $d(x,p_{x}') \leq r/2$. Finally, if $d(x,x_0) =r$, then $d(x,p_{x}) \leq r/4$. Therefore,
\[ \overline{B}(x,r) \subset \bigcup_{p\in \mathcal{S}}\overline{B}(p,r/2)\]
with $\mathcal{S} \subset \overline{B}(x,r)$ and $\card{\mathcal{S}} \leq D:=3(N-2)(6C+1)+7$.
\end{proof}
	
\subsection{Quasisymmetric maps}
	
A homeomorphism $f:(X,d_X) \to (Y,d_Y)$ between metric spaces is said to be \emph{quasisymmetric} (or $\eta$-quasisymmetric) if there exists a homeomorphism $\eta \colon [0,\infty) \to [0,\infty)$ such that for all $x,a,b \in X$ with $x\neq b$ 
\[ \frac{d_Y(f(x),f(a))}{d_Y(f(x),f(b))} \leq \eta \left ( \frac{d_X(x,a)}{d_X(x,b)} \right ). \]
The composition of two quasisymmetric maps and the inverse of a quasisymmetric map are quasisymmetric. Two spaces $X,Y$ are quasisymmetrically equivalent if there exists a quasisymmetric map between them. 
	
For doubling connected metric spaces it is known that the quasisymmetric condition is equivalent to a weaker (but simpler) condition known in literature as \emph{weak quasisymmetry} \cite[Theorem 10.19]{Heinonen}. A homeomorphism $f:(X,d_X) \to (Y,d_Y)$ between metric spaces is said to be \emph{weakly quasisymmetric} if there exists $H\geq 1$ such that for all $x,a,b \in X$,
\begin{equation}\label{eq:weakQS3}
d_X(x,a) \leq  d_X(x,b) \qquad\text{implies}\qquad d_Y(f(x),f(a)) \leq H d_Y(f(x),f(b)).
\end{equation}
	
Quasisymmetric maps were introduced by Tukia and V\"ais\"al\"a \cite{TV80} as the analogue of quasiconformal maps in the abstract metric setting. For more background see \cite[Chapters 10-12]{Heinonen}.
	
It is well known that the properties ``doubling'' and ``bounded turning'' are preserved under quasisymmetric maps quantitatively. Bonk and Meyer showed that the properties ``uniformly branch separation'' and ``uniform branch density'' are also preserved \cite[Lemmas 4.3 and 4.5]{BM22}. Below we show that the ``uniform branch growth'' property, which we introduced in the Introduction, is also preserved.
	
\begin{lem}\label{lem:qs maps preserve comparability of branches}
Let $T$ be a metric tree with uniform branch growth. If $T'$ is quasisymmetrically homeomorphic to $T$, then $T'$ has uniform branch growth.
\end{lem}

\begin{proof}
Let $f:T \to T'$ be an $\eta$-quasisymmetric homeomorphism. Let $p\in T'$ be a branch point. If $p$ has only three branches, then there is nothing to prove. We assume for the rest that $p$ has at least 4 branches. We have that $f^{-1}(p)$ is a branch point of $T$ with at least 4 branches, which we enumerate $B_1, B_2, \dots$ so that $\diam{B_i} \geq \diam{B_{i+1}}$ for all $i$.
		
Fix indices $i,j$ with $i>j$. Let $x_i \in B_i$ such that $d_{T'}(f(x_i),p) \geq \frac12\diam{f(B_i)}$ and let $x_j \in B_j$ such that $d_{T}(f^{-1}(p),x_j) \geq \frac12\diam{B_j}$. Then, 
\begin{align}\label{eq:heights-comp-1}
\frac{\diam{f(B_i)}}{\diam{f(B_j)}} \leq \frac{2d_{T'}(f(x_i),p)}{d_{T'}(f(x_j),p)} \leq 2\eta\left( \frac{d_T(x_i,f^{-1}(p))}{d_T(x_j,f^{-1}(p))} \right) \leq 2\eta\left(2\frac{\diam{B_i}}{\diam{B_j}}\right) \leq 2\eta(2).
\end{align}
		
Similarly, if $i>j \geq 3$, then 
\begin{equation}\label{eq:heights-comp-2}
\frac{\diam{f(B_j)}}{\diam{f(B_i)}} \leq 2\eta(2C),
\end{equation}
where $C$ is the uniform branch growth constant of $T$.
		
Suppose now that $i_1,i_2,i_3,i_4$ are distinct indices such that 
\[ \diam{f(B_{i_1})} \geq \diam{f(B_{i_2})} \geq \diam{f(B_{i_3})} \geq \diam{f(B_{i_4})}. \]
We proceed with a case study.
		
\emph{Case 1.} Assume that $i_3,i_4 \geq 3$. Then, uniform branch growth of $T'$ follows immediately from either \eqref{eq:heights-comp-1} (if $i_3>i_4$) or \eqref{eq:heights-comp-2} (if $i_3<i_4$).
		
\emph{Case 2.} Assume that $i_3 \in \{1,2\}$. Then, at least one of $i_1,i_2$, say $i_1$, is in $\{3,4,\dots\}$ and we have 
\[ \frac{\diam{f(B_{i_3})}}{\diam{f(B_{i_4})}} \leq \frac{\diam{f(B_{i_1})}}{\diam{f(B_{i_4})}} \leq 2\eta(2C),\]
by either \eqref{eq:heights-comp-1} (if $i_1>i_4$), or \eqref{eq:heights-comp-2} (if $i_1<i_4$). 
		
\emph{Case 3.} Assume that $i_4 \in \{1,2\}$. Then, at least one of $i_1,i_2$, say $i_1$, is in $\{3,4,\dots\}$ and we have 
\[ \frac{\diam{f(B_{i_3})}}{\diam{f(B_{i_4})}} \leq \frac{\diam{f(B_{i_1})}}{\diam{f(B_{i_4})}} \leq 2\eta(2)\]
by \eqref{eq:heights-comp-1} since $i_1>i_4$.
\end{proof}

\subsection{Quasi-visual subdivisions}
The following three definitions and proposition from \cite{BM22} are the key factors for the proof of Theorem \ref{thm:uniformization}.
	
\begin{definition}[{Quasi-visual approximations}]\label{def: qv approx} Let $S$ be a bounded metric space. A \textit{quasi-visual approximation} of $S$ is a sequence $(\textbf{X}^n)_{n\in\N_0}$, where $\textbf{X}^0=\{S\}$, and $\textbf{X}^n$ is a finite cover of $S$, for any $n\in\N$,  with the following properties. Note that the implicit constants in what follows are independent of $n, X, Y$:
\begin{enumerate}[(i)]
\item $\diam(X)\simeq\diam(Y)$ for all $X,Y\in \textbf{X}^n$ with $X\cap Y\neq \emptyset$.
\item $\dist(X,Y)\gtrsim\diam(X)$ for all $X,Y\in\textbf{X}^n$ with $X\cap Y=\emptyset$.
\item $\diam(X)\simeq\diam(Y)$ for all $X\in\textbf{X}^n,Y\in\textbf{X}^{n+1}$ with $X\cap Y\neq \emptyset$.
\item For some constants $k_0\in\N$ and $\lambda\in (0,1)$ independent of $n$ we have $\diam(Y)\leq \lambda \diam(X)$ for all $X\in\textbf{X}^n$ and $Y\in\textbf{X}^{n+k_0}$ with $X\cap Y\neq \emptyset$.
\end{enumerate}
We write $(\textbf{X}^n)$ instead of $(\textbf{X}^n)_{n\in\N_0}$ with the index set $\N_0$ for $n$ understood. We call the elements of $\textbf{X}^n$ the \textit{tiles of level $n$}, or simply the \textit{$n$-tiles} of the sequence $(\textbf{X}^n)$.
\end{definition}
	
\begin{definition}[{Subdivisions}]\label{Subdivison} A \textit{subdivision} of a compact metric space $S$ is a sequence $(\textbf{X}^n)_{n\in\N_0}$ with the following properties:
\begin{enumerate}[(i)]
\item $\textbf{X}^0=\{S\}$, and $\textbf{X}^n$ is a finite collection of compact subsets of $S$ for each $n\in\N$.
\item For each $n\in\N_0$ and $Y\in\textbf{X}^{n+1}$, there exists $X\in\textbf{X}^n$ with $Y\subset X$.
\item For each $n\in\N_0$ and $X\in\textbf{X}^n$, we have 
$$X=\bigcup\{Y\in \textbf{X}^{n+1}:Y\subset X\}.$$
\end{enumerate}
\end{definition}
	
\begin{definition}[{Quasi-visual subdivisions}]\label{qv subd} 
A subdivision $(\textbf{X}^n)$ of a compact metric space $S$ is called \textit{quasi-visual subdivision} of $S$ if it is a quasi-visual approximation of $S$.
\end{definition}
	
Let $(\textbf{X}^n)$ and $(\textbf{Y}^n)$ be subdivisions of compact metric spaces $S$ and $T$, respectively. We say $(\textbf{X}^n)$ and $(\textbf{Y}^n)$ are \textit{isomorphic subdivisions} if there exist bijections $F^n:\textbf{X}^n\to \textbf{Y}^n$ , $n\in\N_0$, such that for all $n\in\N_0$, $X,Y\in\textbf{X}^n$, and $X'\in\textbf{X}^{n+1}$ we have 
\begin{equation*}
X\cap Y\neq\emptyset\  \text{if and only if}\ F^n(X)\cap F^n(Y)\neq \emptyset
\end{equation*}
and
\begin{equation*}
X'\subset X\ \text{if and only if}\ F^{n+1}(X')\subset F^n(X).
\end{equation*}
We say that the isomorphism between $(\textbf{X}^n)$ and $(\textbf{Y}^n)$ is \textit{given} by the family $(F^n)$. We say that an isomorphism $(F^n)$ is induced by a homeomorphism $F:S\to T$ if $F^n(X)=F(X)$ for all $n\in\N_0$ and $X\in\textbf{X}^n$.
	
\begin{prop}\cite[Proposition 2.13]{BM22}\label{prop: qs match}
Let $S$ and $T$ be compact metric spaces with quasi-visual subdivisions $(\textbf{X}^n)$ and $(\textbf{Y}^n)$ that are isomorphic. Then there exists a unique quasisymmetric homeomorphism $F:S\to T$ that induces the isomorphism between $(\textbf{X}^n)$ and $(\textbf{Y}^n)$.
\end{prop}

\subsection{Combinatorial graphs and trees}\label{sec:graphprelim}
	
An \emph{alphabet} $A$ is an at most countable set that is either equal to $\N$ or of the form $\{1, \dots, M\}$ for some integer $M\geq 2$. For each integer $k\geq 0$ denote by $A^k$ the set of words formed by $k$ letters of $A$, with the convention $A^0 = \{\e\}$, where $\e$ denotes the empty word. We set  $A^* = \bigcup_{k\geq 0} A^k$ to be the collection of words of finite length and $A^{\N}$ to be the collection of words of infinite length. If $w\in A^*$, then we denote by $|w|$ the \textit{length} of $w$, i.e., the number of letters that $w$ consists of, with the convention $|\e| =0$. Given two finite words $w=i_1\cdots i_n$ and $v=j_1\cdots j_m$ in $A^*$, we denote by $wu = i_1\cdots i_n j_1\cdots j_m$ their concatenation.
	
Given $w\in A^*$ and integer $k\geq |w|$, denote
\[ A^k_{w} = \{wu : u\in A^{k-|w|}\}, \quad A^*_w = \{wu : u\in A^*\}, \quad A^{\N}_w = \{wu: u\in A^{\N}\}.\]
Given $n\in\N$ and $w\in A^{\N}$ denote by $w(n)$ the unique word $u\in A^n$ such that $w=uw'$ for some $w'\in A^{\N}$. Similarly, if $n\in\mathbb{N}$ and $w\in A^{*}$, $w(n)$ denotes the initial subword of $w$ of length $n$, and we set $w(n)=w$ if $n\geq |w|$. Given $i\in A$ and $k\in\N$, we denote by $i^{(k)}$ the word $i\cdots i \in A^k$ and by $i^{(\infty)}$ the infinite word $ii\cdots \in A^{\N}$.
	
A \emph{combinatorial graph} is a pair $G=(V,E)$ of a finite or countable vertex set $V$ and an edge set 
$$E \subset \left\{ \{v,v'\} : v,v' \in V\text{ and }v\neq v' \right\}.$$
If $\{v,v'\}\in E$, we say that the vertices $v$ and $v'$ are \textit{adjacent} in $G$.
	
A combinatorial graph $G' = (V',E')$ is a \textit{subgraph} of $G=(V,E)$, and we write $G\subset G'$, if $V'\subset V$ and $E'\subset E$. We commonly generate subgraphs of $G=(V,E)$ by starting with a vertex set $V'\subset V$ and considering the \textit{subgraph of $G$ induced by $V'$}, i.e., the graph $G'=(V',E')$ where $E'$ is the set of all edges between two vertices of $V'$.
	
A \emph{combinatorial path} in $G$ is an ordered set $ \g = (v_1,v_2,\dots,v_n) \in V^n$ such that $v_i$ is adjacent to $v_{i+1}$ for all $i\in\{1,\dots,n-1\}$; in this case we say that $\g$ joins $v_1$ to $v_n$. The path $ \g = (v_1,v_2,\dots,v_n)$ is a \emph{combinatorial arc} or \emph{simple path} if for all $i, j \in \{1,\dots,n\}$, $v_i = v_j$ if and only if $i=j$; in this case we say that the endpoints of the arc $\g$ are the points $v_1,v_n$. The \emph{length} of a path $\gamma$ is the number of vertices that it contains. If $\g = (v_1,\dots, v_k)$ and $\g'=(u_1,\dots,u_m)$ are two paths in $G$ with $v_k=u_1$, then we denote by $\g\g'$ the concatenation path $(v_1, \dots, v_k, u_2,\dots, u_m)$. If $\g = (v_1,\dots, v_k)$ we denote by $\g \setminus(v_1):= (v_2,\dots,v_k)$ and similarly we define $\g\setminus (v_k)$. 
	
A combinatorial graph $G = (V,E)$ is connected, if for any distinct $v,v' \in V$ there exists a path $\g$ in $G$ that joins $v$ with $v'$. Given $v\in V$, we write $G\setminus\{v\}$  to be the subgraph of $G$ induced by $V\setminus\{v\}$, and a \emph{component} of $G\setminus\{v\}$ is a maximal connected subgraph of $G\setminus\{v\}$. Note that, if $T$ is a combinatorial tree, then every component of $T\setminus\{v\}$ is a combinatorial tree.
	
A graph $T = (V,E)$ is a \emph{combinatorial tree} if for any distinct $v,v'$ there exists unique combinatorial arc $\g$ whose endpoints are $v$ and $v'$. The unique arc in a combinatorial tree $T$ that starts from $v$ and ends in $v'$ is denoted by $P_T(v,v')$. Given a combinatorial tree $T = (V,E)$ and a point $v\in V$, define the valencies 
\[ \text{Val}(T,v) := \card\{e \in E : v\in e\} \qquad\text{and}\qquad  \text{Val}(T) := \sup_{v\in V} \text{Val}(T,v)\]
and the set of leaves $L(T) := \{v\in V : \text{Val}(T,v) = 1\}$. Here $\card$ denotes the cardinality of a finite or countable set, taking values in $\mathbb{N} \cup \{\infty\}$.
	
\section{Weak tangents and the proof of Theorem \ref{prop:nouniv}}\label{sec:noQSembed}
	
Here we prove Theorem \ref{prop:nouniv}. The proof uses the notion of \emph{weak tangents} from \cite{Gromov1, Gromov2}, which we recall here.
	
Let $(X,d)$ be a metric space and let $p\in X$. We call a metric space $(Y,d')$ a \emph{weak tangent of $X$ at $p$} if there exists $q\in Y$, a sequence $(r_n)_{n\in\N}$ of positive numbers with $r_n \to 0$, and a sequence $(p_n)_{n\in\N}$ of points in $X$ with $p_n\to p$ such that $(X, p_n, r_n^{-1}d) \xrightarrow{GH} (Y,q,d')$. The existence of weak tangents for any doubling metric space at any point is guaranteed (see \cite[Theorem 8.1.10]{BBI}). We denote by $\tang(X,p)$ the collection of all weak tangents of $X$ at $p$. 
	
Recall that a metric space $(X, d)$ is an $\R$-tree if for any $x,y\in X$ there exists a unique arc from $x$ to $y$, and this arc is a geodesic. In our first lemma we show that pointed Gromov-Hausdorff limits of $\R$-trees are $\R$-trees.
	
\begin{lem}\label{lem:limit of trees is tree}
Let $(T_m,p_m,d_m)$ be a sequence of $\R$-trees that converges to a complete metric space $(T,q,d)$ in the pointed Gromov-Haudorff sense. Then, $T$ is an $\R$-tree.   
\end{lem}
	
\begin{proof}
By \cite[Theorem 8.1.9]{BBI} $T$ is a length space. To show that $T$ is an $\R$-tree it suffices to show that $T$ does not have any simple closed curve. 
		
In this proof we denote by $B_T(x,r)$ (resp. $B_{T_m}(x,r)$) the open ball in $T$ (resp. in $T_m$) centered at $x$ and of radius $r$.
		
Assume for a contradiction that $\Gamma\subset T$ is a simple closed curve. Let $x,y\in T$ such that $d(x,y)=\diam \Gamma$. Write $\Gamma\setminus\{x,y\}=\{\alpha,\alpha'\}$ where $\alpha,\alpha'$ are arcs that connect $x,y$. Fix $\e<10^{-1}\diam \Gamma$ and let $x_1,x'_1$ be the last points such that $x_1\in \alpha\cap \partial B_T(x,\e)$  and $x_1'\in\alpha' \cap\partial B_T(x,\e)$, and similarly define $y_1\in \alpha\cap B_T(y,\e)$ and $y_1'\in\alpha' \cap\partial B_T(y,\e)$. Let $\beta\subset \alpha$ and $\beta'\subset\alpha'$ be the subarcs of $\Gamma$ that connect $x_1,x_1'$ and $y_1,y_1'$ respectively. Then, $\dist(\beta,\beta')>0$ since $\Gamma$ is a simple closed curve. 
		
Fix $\e'<10^{-1}\dist(\beta,\beta')$. There exist points $\{z_i\}_{i=1}^N$ in $\beta$, and points $\{z_j'\}_{j=1}^{N'}$ in $\beta'$ (both sets enumerated according to the orientations of $\beta$ and $\beta'$) such that $z_1=x_1$, $z_N=y_1$, $z_1'=x_1'$, $z_{N'}'=y_1'$, and for all $i\in\{1,\dots,N-1\}$ and $j\in\{1,\dots,N'-1\}$
\[ \e' \leq d(x_i,x_{i+1}) \leq 2\e' \quad\text{and}\quad \e' \leq d(x_j',x_{j+1}') \leq 2\e'.  \]
Let also $x=z_0=z_0'$ and $y=z_{N+1}=z_{N'+1}'$. 
		
Set $\e''<\frac19\min\{\dist(\beta,\beta')-4\e',\diam\Gamma-4\e\}$. Choose $r>0$ big enough such that $\Gamma\subset B(q,r-\e'')$. Then there exists $M\in\N$ such that for all $m\geq M$ we can find $f:B(p_m,r)\to T$ and points $\{w_{z_i}\}_{i=0}^{N+1}$, $\{w_{z_j'}\}_{j=0}^{N'+1}$ in $B(p_m,r)$ such that 
\[ d(f(w_{z_i}),z_i)<\e'' \quad\text{and}\quad d(f(w_{z_j'}),z_j')<\e''\] 
for all $i\in\{0,\dots,N+1\}$ and $j\in\{0,\dots,N'+1\}$. Let 
\[ \gamma=[w_{z_1},w_{z_2}]\cup\dots\cup[w_{z_{N-1}},w_{z_N}]\quad \text{and}\quad \gamma'=[w_{z_1'},w_{z_2'}]\cup\dots\cup[w_{z_{N'-1}'},w_{z_{N'}'}].\] 
We claim that $\dist(\gamma,\gamma')>0$. To see that, fix $a_1\in\gamma$, $a_2\in\gamma'$ and let $i\in\{1,\dots,N-1\}$ and $j\in\{1,\dots,N'-1\}$ such that $a_1\in[w_{z_i},w_{z_{i+1}}]$ and $a_2\in[w_{z_j'},w_{z_{j+1}'}]$. Then
\begin{align*}
d_m(a_1,a_2)&\geq d_m(w_{z_i},w_{z_j'})-d_m(a_1,w_{z_i})-d_m(a_2,w_{z_j'})\\
&\geq d(f(w_{z_i}),f(w_{z_j'}))-\e'' - d_m(w_{z_i},w_{z_{i+1}}) - d_m(w_{z_j'},w_{z_{j+1}'})\\
&\geq d(z_i,z_j')-3\e'' - (d(f(w_{z_i}),f(w_{z_{i+1}}))+\e'') - (d(f(w_{z_i'}),f(w_{z_{i+1}'}))+\e'')\\
&> \dist(\beta,\beta') -3\e'' - (d(z_i,z_{i+1})+3\e'') - (d(z_j',z_{j+1}')+3\e'')\\
&\geq \dist(\beta,\beta') - 9\e'' - 4\e'\\
&>0.
\end{align*}
		
Let $\zeta=[w_{z_1},w_x]\cup(w_x,w_{z_1'}]$ and $\zeta'=[w_{z_N},w_y]\cup(w_y,w_{z_{N'}'}]$. Repeating the calculations above, we get that $\dist(\zeta,\zeta')>0$. To obtain a contradiction, we claim that the curve $\gamma\cup\zeta'\cup\gamma'\cup\zeta$ (which is in the tree $T_m$) contains a simple closed curve. Note that 
\[ [w_{z_1},w_{z_N}]\cup\zeta'\cup[w_{z_{N'}'},w_{z_1'}]\cup\zeta\subset \gamma\cup\zeta'\cup\gamma'\cup\zeta.\] 
Denote with $\sigma$ (resp. $\sigma'$) the arc in $\zeta$ (resp. $\zeta')$ that joins $[w_{z_1},w_{z_N}]$ with $[w_{z'_1},w_{z'_{N'}}]$ at the points $w_1,w_2$ respectively (resp. $w'_1,w'_2$). From the above discussion, points $w_1,w_2,w'_1,w'_2$ are all different from each other. It follows that
$$[w_1,w_2]\cup(w_2,w'_2]\cup(w'_2,w'_1]\cup(w'_1,w_1]$$
is a simple closed curve.
\end{proof}
	
For the remainder of this section, for each $N\in\N$ define the planar set
\[ T_N = \ell_0 \cup \bigcup_{n=2}^{\infty}\bigcup_{k=1}^N \ell_{n,k}\]
where
\[\ell_0 = [0,1]\times\{0\} \quad\text{and}\quad \ell_{n,k} = \left\{2^{-n+1}+\tfrac{k-1}{N}4^{-n}\right\} \times [0,3^{-n}].\]
	
It is easy to see that $T_N$ is compact (segments $\ell_{n,k}$ accumulate at $(0,0)$), connected, locally connected (the length of $\ell_{n,k}$ goes to 0 as $n$ goes to infinity) and contains no simple closed curves. Therefore $T_N$ is a metric tree. Moreover, every branch point has valence 3 (since segments $\ell_{n,k}$ are mutually disjoint). Finally, the total Hausdorff $1$-measure of $T_N$ is
\[ \mathcal{H}^1(T_N) = 1+ N\sum_{n=2}^{\infty}3^{-n} < \infty.\]
Therefore, we can equip $T_N$ with the geodesic metric $d$.
	
\begin{lem}\label{lem:T_N doubling}
For each $N\in\N$, $(T_N,d)$ is a doubling geodesic tree of valence 3.
\end{lem}
	
\begin{proof}
We only need to show the doubling property. Towards this end, fix $x\in T_N$ and $r>0$. Since $T_N$ is compact, we may assume that $r<1/2$.
		
Suppose that $x\in \ell_{n,k}$ for some $n,k$ and let $x'$ be the unique point in $\ell_{n,k} \cap \ell_0$. If $r\leq d(x,x')$, then $B(x,r)$ is a line segment. If $r\geq d(x,x')$, then $B(x,r) \subset B(x',2r)$. Therefore, we may assume for the rest of the proof that $x\in \ell_0$. 
		
Let $m\in \N$ be the unique integer with $2^{-m-1} < r \leq 2^{-m}$. Let also $0 \leq y < z \leq 1$ be such that $B(x,r) \cap \ell_0 = [y,z]\times\{0\}$. It is easy to see that there exist at most $4N$ many segments $\ell_{n,k}$ that intersect with $[y,z]\times\{0\}$ and have length greater or equal to $r/4$. Label these segments by $\ell^{1},\dots,\ell^l$. 
		
On one hand, $\ell_0\cup \ell^1\cup\cdots \cup \ell^l$ equipped with the geodesic metric is $C$-doubling with $C$ depending only on $N$ and not on the segments $\ell^j$. On the other hand, if $\ell_{n,k}$ intersects with $[y,z]\times\{0\}$ and has length less than $r/4$, then, by design of $T_N$, it is contained in $B(y,r/2)$. This completes the proof of the doubling property.
\end{proof}
	
The final ingredient in the proof of Theorem \ref{prop:nouniv} is a result of Li \cite{Li2021} which states that quasisymmetric embeddability is hereditary: if $X$ quasisymmetrically embeds into $Y$, then every weak tangent of $X$ quasisymmetrically embeds into some weak tangent of $Y$.
	
\begin{lem}[{\cite[Theorem 1.1]{Li2021}}]\label{lem:Li}
Let $X,Y$ be proper, doubling metric spaces and $f: (X,p,d_X) \to (Y,q,d_Y)$ be an $\eta$-quasisymmetric map. For any weak tangent $T \in \tang(X,p)$, there exists a weak tangent $T' \in \tang(Y,q)$ such that $T$ is $\eta$-quasisymmetric equivalent to $T'$.
\end{lem}
	
\begin{proof}[Proof of Theorem \ref{prop:nouniv}]
Assume for a contradiction, that there exists a quasiconformal tree $T$ such that every element of $\mathscr{QCT}(3)$ quasisymmetrically embeds in $T$. Since every quasiconformal tree is quasisymmetrically equivalent to a geodesic tree \cite{BM20}, we may assume that $T$ is geodesic and $C$-doubling.
		
Fix $N\in\N$ such that $N+2 >C$, let $p=(0,0)\in T_N$, and let $\omega_1,\dots,\omega_{N+2}$ be the $(N+2)$-roots of unity, and let $S$ be the $(N+2)$-star
\[ \bigcup_{j=1}^{N+2}\{t\omega_j : t\geq 0\} \subset \R^2\]
equipped with the geodesic metric $d_S$ and let $q=(0,0)$.
		
We claim that $S \in \tang(T_N,p)$. Assuming the claim, by Lemma \ref{lem:T_N doubling}, $T_N$ is in $\mathscr{QCT}(3)$, and by Lemma \ref{lem:Li} $S$ quasisymmetrically embeds into a weak tangent $S'$ of $T$. By Lemma \ref{lem:doubling of weak tangent} and Lemma \ref{lem:limit of trees is tree} $S'$ is a $C$-doubling $\R$-tree. However, if $B$ is a ball with center the embedded image of the branch point $q$ and radius $\delta$, then any $\delta/2$-separated set in $B$ has $N+2$ elements (simply by taking the points on the boundary of the ball of radius $\delta/2$). It follows that $N+2\leq C$ which is contradiction.
		
To prove the claim, set $p_n=2^{-n+1}$ and $r_n=(7/2)^{-n}$ for each $n\in \N$. Fix $\e>0$ and $R>0$. We may assume that $R\geq 1$ since otherwise for large $n$ the balls $B(p_n,R)$ will collapse on the point $q$. There exists $n_0\in\N$ such that for all $n\geq n_0$
\[ R< (7/6)^n, \quad \text{and} \quad \tfrac{N-1}{N}(7/8)^n<\e. \]
Note that $r_n < 2^{-n-2}$ holds for all $n\geq n_0$. Thus, for all $n>n_0$, if $y\in \partial B(p_n,r_nR)$, then $d(p_n,y)=r_n R< 3^{-n}$. Moreover, 
\[ B(p_n,r_nR) = h_{n}^1 \cup h_n^2 \cup h_n^3 \cup \ell_{n,1}'\cup\cdots\cup \ell_{n,N}'\]
where $h_{n}^1 = (p_n-r_nR,p_n]\times\{0\}$, $h_{n}^2 = [p_n,p_n+ \frac{N-1}{N}4^{-n}]\times\{0\}$, 
\[ h_{n}^3 = [p_n+ \tfrac{N-1}{N}4^{-n},p_n+r_nR)\times\{0\},\] 
and 
\[ \ell_{n,j}' = \{a_{n,j}\}\times [0,t_{n,j}) = B(p_n,r_nR)\cap \ell_{n,j}.\]
		
Define a map $f:\overline{B}(p_n,r_n R)\to S$ so that 
\begin{enumerate}
\item $f(a_{n,j},t)=t\omega_j$ for $t\in [0,t_{n,j})$ and $j\in \{1,\dots,N\}$,
\item $f(p_n-t,0) = t\omega_{N+1}$ for $t\in [0,r_nR)$,
\item $f(p_n+\frac{N-1}{N}4^{-n} + t,0) = t\omega_{N+2}$ for $t\in [0,r_nR-\frac{N-1}{N}4^{-n})$, and
\item $f([p_n,p_n+\frac{N-1}{N}4^{-n}]) = (0,0)$. 
\end{enumerate}
		
Let $x_1,x_2\in B(p_n,r_nR)$. If $x_1,x_2\in \ell_{n,j}'$ for some $j\in\{1,\dots,N\}$, then 
\[ |d_S(f(x_1),f(x_2))-r_n^{-1}d(x_1,x_2)|=0.\] 
Similarly, if $x_1,x_2\in h_n^1$ or if $x_1,x_2\in h^3_n$. If $x_1,x_2\in h_n^2$, then 
\begin{align*}
|d_S(f(x_1),f(x_2))-r_n^{-1}d(x_1,x_2)|=|r_n^{-1}d(x_1,x_2)|&\leq \tfrac{N-1}{N}(\tfrac27)^n4^{-n}=\tfrac{N-1}{N}(\tfrac78)^n <\e.
\end{align*}
If $x_1\in \ell_{n,j}'$ for some $j\in\{1,\dots,N\}$  and $x_2\in h_n^1$, then by geodesicity of $T_N$ and $S$ we have 
\begin{align*}
|&d_S(f(x_1),f(x_2))-r_n^{-1}d(x_1,x_2)|\\
&= |d_S(f(x_1),q)+d_S(q,f(x_2)) - r_n^{-1}(d(x_1,(a_{n,j},0))+d((a_{n,j},0),p_n)+d(p_n,x_2))| \\
&= |r_n^{-1}d((a_{n,j},0),p_n)| \\
&\leq \tfrac{N-1}{N}(\tfrac78)^n \\
&< \epsilon.
\end{align*}
All the other cases are similar. 
		
Lastly, we need to verify that $B(q,R-\e)\subset N_{\e}(f(B(p_n,r_nR)))$. Let $x = t\omega_j \in B(q,R-\e)$ for some $j\in\{1,\dots, N\}$ (the cases where $j\in\{N+1,N+2\}$ are similar). Let also $(a_{n,j},y)$ be the unique point in $\partial B(p_n,r_n R)\cap\ell_{n,j}$. From the construction of $f$ and geodesicity of $T_N$,  
\begin{align*}
d_S(q,f(a_{n,j},y)) &= r_n^{-1}(d(p_n, (a_{n,j},y)) - d(p_n, (a_{n,j},0))) \\
&\geq r_n^{-1}(d(p_n, (a_{n,j},y)) - \tfrac{N-1}{N}(\tfrac78)^n) \\
& > R - \epsilon.
\end{align*}
Due to geodesicity of $S$, $x$ is on the arc that joins $q$ with $f(y)$. Thus, there exists $x'\in B(p_n,r_nR)$ such that $x=f(x')$. It follows that 
\[ B(q,R-\e)\subset f(B(p_n,r_nR))\subset N_{\e}(f(B(p_n,r_nR))).\qedhere\]
\end{proof}
	
\section{Construction of trees $\T^{A,\ba}$}\label{sec:construction}
For each alphabet $A$ (finite or infinite), and for each choice of weights $\ba = (a_1,a_2,\dots)$ with $\frac12 =a_1=a_2 \geq a_3\geq \cdots$ and $\lim a_n =0$, we construct an associated metric tree $\T^{A,\ba}$. The construction is done in several steps. Given an alphabet $A$, in \textsection\ref{sec:graphs} we define a sequence of combinatorial graphs $(G_n^A)_{n\in\N}$ so that each graph $G^A_n$ has the set of words $A^n$ as the set of vertices. Using these graphs, following \cite[(1.1)]{DV} we define in \textsection\ref{sec:combinter} a notion of ``intersection'' between sets $A^{\N}_w$ and $A^{\N}_u$ for $w,u\in A^*$. This notion of intersection (which is not a set intersection), along with the weights $(a_n)_{n\in\N}$, gives rise in \textsection\ref{sec:metric} to a pseudometric on the set $A^{\N}$. The pseudometric yields an equivalence relation $\sim$ on $A^{\N}$ and a metric $d_{A,\ba}$ on $A^{\N}/\sim$. We then proceed to define $\T^{A,\ba} = (A^{\N}/\sim,d_{A,\ba})$. Finally, in \textsection\ref{sec:selsim}, we show that each metric $\T^{A,\ba}$ is a self-similar 1-bounded turning tree. These trees are shown to be universal later on.
	
\subsection{Graphs}\label{sec:graphs}
Given an alphabet $A$, for each $k\in \N$ we construct a graph $G_k^A = (A^k,E_k^A)$ inductively. 
\begin{enumerate}[(i)]
\item Let $E_1^A = \{\{1,i\}:i\in A\setminus\{1\}\}$. 
\item Assume that for some $k\in \N$ we have defined a graph $G_k^A = (A^k, E_k^A)$. 
Set $G_{k+1}^A = (A^{k+1},E_{k+1}^A)$ where
\[ E_{k+1}^A = \left\{\{12^{(k)},i1^{(k)}\} : i\in A\setminus\{1\}\right\} \cup \bigcup_{i\in A}\left\{\{iw, iu\} : w,u\in A^k \text{ and }\{w,u\}\in E_k^A\right\}.\]
\end{enumerate}
Intuitively, the graph $G_{k+1}$ consists of $\card A$ copies of graphs isomorphic to $G_k$, which are connected through edges meeting at $12^{(k)}$ and $i1^{(k)}$ for $i\in A\setminus\{1\}$.
	
\begin{lem}\label{lem:prop.of.graphs}
Let $k\geq 0$, $w,u \in A^k$.
\begin{enumerate}[(i)]
\item For distinct $i,j \in A$ we have $\{iw,ju\} \in E_{k+1}^A$ if and only if one of the words is of the form $12^{(k)}$ and the other of the form $l1^{(k)}$ with $l\in A\setminus \{1\}$.
\item For any $v\in A^*$ we have that $\{w,u\}\in E_k^A$ if and only if $\{vw,vu\}\in E_{k+|v|}^A$.
\item The subgraph of $G_{k+1}^A$ generated by $\{wi : i \in A\}$ is connected.
\item We have that $\{w,u\} \in E_k^A$ if and only if there exist $l\in\{0,\dots,k-1\}$, $i\in A\setminus\{1\}$, and $v\in A^{k-l-1}$ such that one of the words is of the form $vi1^{(l)}$ and the other of the form $v12^{(l)}$.
\item If $\{w,u\}\in E_k$, then there exist $i,j\in\{1,2\}$ such that $\{wi,uj\}\in E_{k+1}$.
\end{enumerate}
\end{lem}
	
	\begin{figure}\label{fig: G_k for n=3}
		\centering
		\includegraphics[width=0.85\textwidth]{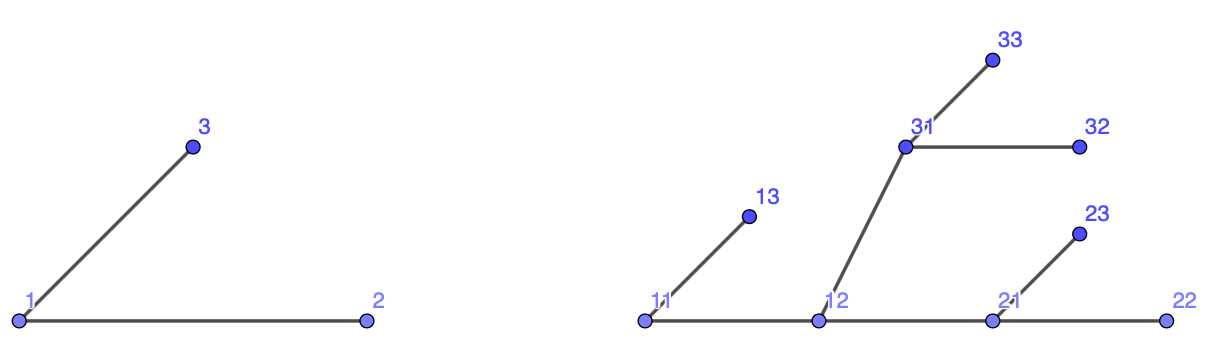}\\
		\includegraphics[width=0.6\textwidth]{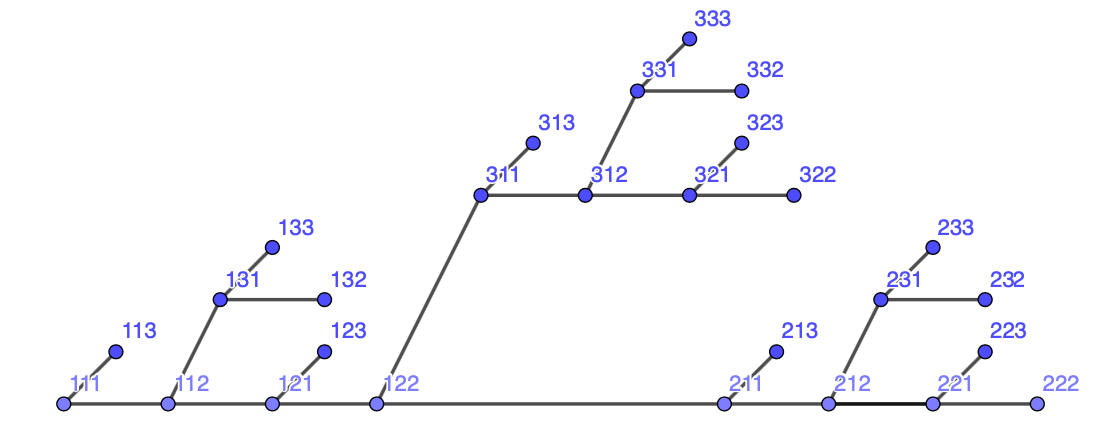}
		\caption{Graphs $G_1^A,G_2^A,G_3^A$ in the case that $A=\{1,2,3\}$.}
	\end{figure}
	\begin{figure}\label{fig: G_k for n=4}
		\centering
		\includegraphics[width=0.8\textwidth]{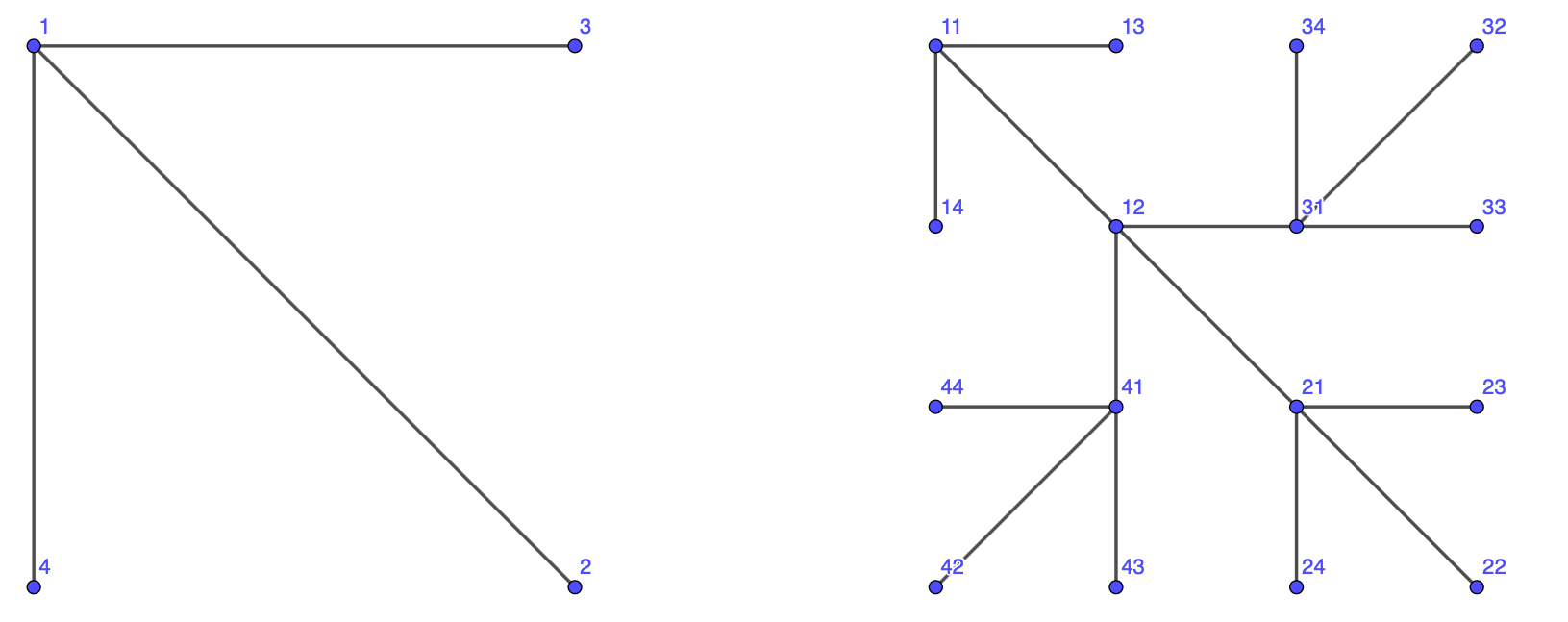}\\
		\includegraphics[width=0.5\textwidth]{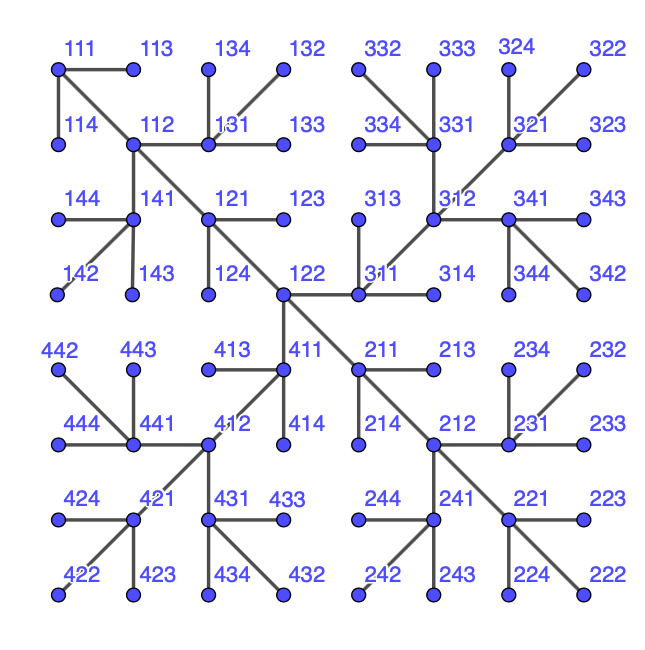}
		\caption{Graphs $G_1^A,G_2^A,G_3^A$ in the case that $A=\{1,2,3,4\}$.}
	\end{figure}
	
\begin{proof}
The first claim is clear from the design of graphs $G_k^A$.
		
We prove property (2) by induction on the length of $v$. The claim is trivially true for $v=\e$. Inductively, suppose that for some $m\geq 0$ the claim is true for all words $v\in A^m$. Let now $v=iv' \in A^{m+1}$ with $v'\in A^m$ and $i\in A$. Assume first that $\{w,u\}\in E_{k}^A$. By the inductive hypothesis, $\{v'w,v'u\}\in E_{k+m}^A$, and by design of $E_{k+m+1}^A$ we have that $\{iv'w,iv'u\}\in E_{k+m+1}^A$. Conversely, assume that $\{iv'w,iv'u\}\in E_{k+m+1}^A$. By design of $E_{k+m+1}^A$ we have that $\{v'w,v'u\}\in E_{k+m}^A$ and by the inductive hypothesis, $\{w,u\}\in E_{k}^A$.
		
Property (3) follows immediately from property (2) and the fact that $G_1$ is connected. 
		
For the fourth property, write $w=viw'$ and $u=vju'$ where $v\in A^l$ with $l\in\{0,\dots,k-1\}$, $u,u' \in A^{k-l-1}$, and $i,j\in A$ with $i< j$. By property (2) we have that $\{iw',ju'\} \in E_{k-l}^A$ if and only if $\{w,u\} \in E_k^A$. The claim now follows from property (1).
		
Assume that $\{w,u\} \in E_k^A$. By property (4), there exist $l\in\{0,\dots,k-1\}$, $i\in A\setminus\{1\}$, and $v\in A^{k-l-1}$ such that one of the words (say $w$) is of the form $vi1^{(l)}$ and the other (say $u$) of the form $v12^{(l)}$. Then $\{w1,u2\} \in E_{k+1}^A$.  
\end{proof}
	
\begin{rem}
Properties (3) and (5) from Lemma \ref{lem:prop.of.graphs} imply that the collection of graphs $(G_k^A)_{k\in\N}$ are \textbf{combinatorial data} in the sense of \cite[Definition 1.1]{DV}.
\end{rem}
	
\begin{lem}
For each $k\in\N$, the graph $G_k^A$ is a combinatorial tree.
\end{lem}
	
\begin{proof}
The proof is by induction on $k$. For $k=1$, the claim is clear. Assume the claim to be true for some $k\in\N$. Note that for each $i\in A$ the subgraph $G_{k+1,i}^A$ of $G_{k+1}^A$ generated by the vertices $A_i^{k+1}$ is isomorphic to $G_k^A$; hence it is a tree.
		
We first show that $G_{k+1}^A$ is connected. Let $v_1,v_2 \in A^{k+1}$. If there exists $i\in A$ such that $v_1,v_2 \in A^{k+1}_i$, then $v_1,v_2$ can be joined by path in $G_{k+1,i}^A$ by the inductive assumption. Suppose now that $v_1\in A^{k+1}_i$ and $v_2\in A^{k+1}_j$ with $i,j\in A$ distinct. Without loss of generality, assume that $i,j \neq 1$. By the inductive assumption, there exists a path $\gamma_1$ from $v_1$ to $i1^{(k)}$ and a path $\gamma_2$ from $j1^{(k)}$ to $v_2$. Then the concatenation of $\g_1$, the edges $\{i1^{(k)},12^{(k)}\}$, $\{j1^{(k)},12^{(k)}\}$, and the path $\gamma_2$ is a path joining $v_1$ to $v_2$. Therefore, $G_{k+1}^A$ is connected.
		
To show that $G_{k+1}^A$ is a tree, assume for a contradiction, that there exists a simple closed path  $\gamma$. Since each subgraph $G_{k+1,i}^A$ is a tree, there exist distinct $i,j\in A$ such that the path $\gamma$ intersects $G_{k+1,i}^A, G_{k+1,j}^A$. Without loss of generality, assume that $i\neq 1$. Let $\sigma$ be a maximal subarc in $\gamma \cap G_{k+1,i}^A$. But then the two endpoints of $\sigma$ are both $i1^{(k)}$ which implies that $\gamma$ is not simple.
\end{proof}
	
\begin{lem}\label{lem:graphembed}
If $A \subset A'$ and $n\in\N$, then $G_n^A$ is a subgraph of $G_n^{A'}$. 
\end{lem}
	
\begin{proof}
Fix two alphabets $A,A'$ with $A\subset A'$. For each $n\in\N$, it is clear that the vertices of $G^A_n$ are also vertices of $G^{A'}_n$. It suffices to show that for each $n\in\N$, $E^A_n \subset E^{A'}_n$. The proof is by induction on $n$. 
		
For $n=1$, $E_1^A = \{\{1,i\}:i\in A\setminus\{1\}\}\subset $ $ \{\{1,i\}:i\in A'\setminus\{1\}\}=E_1^{A'}$ since $A\subset A'$. 
		
Assume now that $E_n^A \subset E_n^{A'}$ for some $n\in\N$. We show that $E_{n+1}^A \subset E_{n+1}^{A'}$. Let $\{w,u\}\in E_{n+1}^A$. If $\{w,u\}=\{12^{(n)},i1^{(n)}\}$, where $i\in A\setminus\{1\}$ then $\{w,u\}\in E_{n+1}^{A'}$ since $A\setminus \{1\}\subset A'\setminus \{1\}$. Assume now that $\{w,u\}=\{iw',iu'\}$ for some $i\in A$ with $w',u'\in A^n$ and $\{w',u'\}\in E_n^A$. That means that there exists an $i\in A'$ such that $\{w,u\}=\{iw',iu'\}$ with $w',u'\in A^{'n}$ since $A^n\subset A^{'n}$. Also, since $\{w',u'\}\in E_n^{A'}$ it follows by inductive hypothesis that $\{w,u\}\in E_{n+1}^{A'}$. 
\end{proof}
	
\subsection{Combinatorial intersection}\label{sec:combinter} Given $u_1,u_2 \in A^*$, not necessarily of the same length, define
\begin{align*}
&A^{\N}_{u_1} \wedge A^{\N}_{u_2}\\ 
&:= \{w \in A^{\N}_{u_1} : \forall n > \max\{|u_1|,|u_2|\} \text{ there exists  $u\in A^n_{u_2}$ with $\{w(n),u\}\in E_n^A\}$}\\
&\quad \cup \{w \in A^{\N}_{u_2} : \forall n >\max\{|u_1|,|u_2|\} \text{ there exists  $u\in A^n_{u_1}$ with $\{w(n),u\}\in E_n^A\}$}.\nonumber   
\end{align*} 
The set $A^{\N}_{u_1} \wedge A^{\N}_{u_2}$ is called the \textit{combinatorial intersection} of $A^{\N}_{u_1}$ and $A^{\N}_{u_2}$.
	
\begin{rem}\label{rem:1}
It is easy to see that $A^{\N}_{u_1} \wedge A^{\N}_{u_2} = A^{\N}_{u_2} \wedge A^{\N}_{u_1}$. Also, if $u\in A_w^*$, then $A_u^{\N}\subset A_{w}^{\N}\wedge A^{\N}_u$
\end{rem}
	
With this notion of combinatorial intersection, we can describe how to move between two infinite words. Given two words $w,w' \in A^{\N}$ we say that $\{A^{\N}_{w_1},\dots, A^{\N}_{w_N}\}$ is a \emph{chain joining $w$ with $w'$} if $w\in A^{\N}_{w_1}$, $w'\in A^{\N}_{w_N}$ and for every $i=1,\dots,N-1$, we have $A^{\N}_{w_i}\wedge A^{\N}_{w_{i+1}} \neq \emptyset$. 
	
\subsection{A metric}\label{sec:metric}
A \emph{weight} is a non-increasing function $\ba:\N \to (0,1/2]$ such that $\ba(1)=\ba(2)=1/2$ and $\lim_{i\to\infty}\ba(i) = 0$. Given a weight $\ba$ and the alphabet $\N$, define the associated ``diameter function'' $\D_{\ba}:\N^* \to (0,1]$ by $\D_{\ba}(\e) =1$ and 
\[ \D_{\ba}(i_1\cdots i_k) = \ba(i_1)\cdots \ba(i_k).\]
	
\begin{rem}
For any weight $\ba$,
\[ \lim_{n\to \infty}\max\{\D_{\ba}(w):w\in \N^n\} \leq \lim_{n\to \infty}2^{-n} =0.\]
\end{rem}
	
Fix now an alphabet $A$ and a weight $\ba$. We define a function $\rho_{A,\ba} :A^\mathbb{N}\times A^\mathbb{N} \to \R$ by:
\begin{equation}\label{eq:pseudometric}
\rho_{A,\ba}(w, u) = \inf \left\{ \sum_{i=1}^N \D_{\ba}(v_i) : \{A^\N_{v_0},\dots, A^\N_{v_N}\}\text{ is a chain joining $w$ with $u$} \right\}.
\end{equation}
Following the arguments in \cite[Lemma 3.8]{DV}, we obtain that $\rho_{A,\ba}$ is a pseudometric.
	
Taking the quotient space $A^\N/\sim$ under the equivalence relation $w \sim w'$ whenever $\rho(w,w')=0$, we obtain a metric space
$$ \mathbb{T}^{A,\ba} = (A^\N/\sim, d_{A,\ba}), \qquad d_{A,\ba}([w],[u]) = \rho_{A,\ba}(w, u).$$
For any $w\in A^*$, write $\mathbb{T}^{A,\ba}_w := A^{\N}_w/\sim$. Under this metric we have that $\diam{\mathbb{T}^A_w}\leq \D_{\ba}(w)$ \cite[Lemma 3.9]{DV}. 
	
\begin{rem}\label{rem:equivalent classes of 12^oo}
Note that for $v\in A^*$, $ \{ A_{v12^{(n)}}^\N, A_{vj1^{(n)}}^\N \}$ is a chain joining $v12^{(\infty)}$ with $vj1^{(\infty)}$ for $j\in\{2,\dots,m\}$. Hence, $\rho_{A,\ba}(v12^{(\infty)}, vj1^{(\infty)})\leq \Delta(v)(2^{-(n+1)}+\ba(j)2^{-n})$ for all $n\in \N$, which yields that $[v12^{(\infty)}]=[vj1^{(\infty)}]$ for all $j\in \{2,\dots,m\}$.
\end{rem}

\begin{rem}
It is worth commenting on the relationship of combinatorial intersection and set intersection. Given $w,u\in A^*$, it is not hard to see that if $v \in A^{\N}_w \wedge A^{\N}_u$, then $[v] \in \mathbb{T}^{A,\ba}_w\cap \mathbb{T}^{A,\ba}_u$. The converse, however, is false as for $A=\{1,2,3\}$, we have that $[31^{(\infty)}] \in \mathbb{T}^{A,\ba}_2\cap \mathbb{T}^{A,\ba}_3$ but the combinatorial intersection $A^{\N}_2 \wedge A^{\N}_3$ is empty.
\end{rem}
    
\begin{lem}\label{lem:metrictree}
The space $\mathbb{T}^{A,\ba}$ is a 1-bounded turning metric tree.
\end{lem}
	
Lemma \ref{lem:metrictree} is essentially a variant of \cite[Proposition 3.10]{DV}. The only difference is that in \cite{DV} the following condition is assumed when $A=\N$. 
\begin{equation}\label{eq:CD}
\text{For any $w\in A^*$, $\Delta_{\ba}(wi)=0$ for all but finitely many $i\in A$.}
\end{equation}
In our setting, property \eqref{eq:CD} is equivalent to the following property.
\begin{equation}\label{eq:CD2}
\text{For all but finitely many $i\in A$ we have $\ba(i)=0$.}
\end{equation}
	
Note that \eqref{eq:CD2} is trivial if $A$ is finite. However, none of \eqref{eq:CD} and \eqref{eq:CD2} are assumed in Lemma \ref{lem:metrictree}.
	
\begin{proof}[{Proof of Lemma \ref{lem:metrictree}}]
We start by proving that $\mathbb{T}^{A,\ba}$ is compact. It is enough to focus on the case $\card A=\infty$, in view of \eqref{eq:CD2} and \cite[Lemma 3.12]{DV}. Let $([w_k])$ be a sequence in $\mathbb{T}^{A,\ba}$. For each $k\in\N$ write $w_k=i^k_1i^k_2\cdots$. There are two cases to consider.

\textit{Case 1}: Suppose that for each $j\in\N$, $(i^k_j)_k$ is a bounded sequence. In this case we can work as in \cite[Lemma 3.12]{DV}. There exists a subsequence $([w_{1,k}])_k$ of $([w_k])_k$ such that the first letter of all $w_{1,k}$ is some fixed $l_1\in A$. Inductively, assume that we have defined a subsequence $([w_{n,k}])_k$ of $([w_k])_k$. There exists a subsequence $([w_{n+1,k}])_k$ of $([w_{n,k}])_k$ such that the $n+1$ letter of all $w_{n+1,k}$ is some fixed $l_{n+1}\in A$. Define now the subsequence $([w_{k,k}])_k$ of $([w_k])_k$ and note that for all $k\in\N$, $[w_{k,k}]\in \mathbb{T}^{A,\ba}_{l_1\cdots l_k}$. Therefore, setting $v=l_1l_2\cdots \in \mathbb{T}^{A,\ba}$ we have
\[ d_{A,\ba}([w_{k,k}],[v]) \leq \diam{\mathbb{T}^{A,\ba}_{l_1\cdots l_k}} \leq 2^{-k}\xrightarrow{k\to \infty} 0. \]
Therefore, $[w_{k,k}] \to [v]$.

\textit{Case 2}: There exists smallest $j\in \N$ such that the sequence $(i^k_j)_k$ is unbounded. Passing to a subsequence, we may further assume that $\lim_{k\to \infty}i^k_j = \infty$. Since sequences $(i^k_1)_k,\dots,(i^k_{j-1})_k$ are bounded, passing to a subsequence again, we may assume that sequences $(i^k_1)_k,\dots,(i^k_{j-1})_k$ are constant and we write $v=i_1^1\cdots i^1_{j-1}$. By Remark \ref{rem:equivalent classes of 12^oo}, $[w_k], [v12^{(\infty)}] \in \mathbb{T}^{A,\ba}_{v i_k}$ for all $k\in\N$, so
\[ d_{A,\ba}([w_{k}],[v12^{(\infty)}]) \leq \diam{\mathbb{T}^{A,\ba}_{vi_k}} \leq \D_{\ba}(v)\ba(i_k) \xrightarrow{k\to \infty} 0\]
which gives that $[w_k] \to [v12^{(\infty)}]$.
		
		
The fact that $\mathbb{T}^{A,\ba}$ is connected, locally connected, and path-connected follows by \cite[Lemma 3.14]{DV}; note that \eqref{eq:CD} is not used in the proof of that part. 
		
The fact that $\mathbb{T}^{A,\ba}$ is 1-bounded turning follows from \cite[Lemma 3.15]{DV}; note that \eqref{eq:CD} is not used in the proof of that part. 
		
Finally, the fact that $\mathbb{T}^{A,\ba}$ is a metric tree follows from \cite[Lemma 3.17]{DV}. While property \eqref{eq:CD} is mentioned in \cite[Claim 3.20]{DV} to prove the compactness of sets $X_j$ therein, the compactness of these sets does not play a role in the proof of the said claim. 
\end{proof}
	
\subsection{Self-similarity}\label{sec:selsim}
	
In the next lemma we define similarity maps on the trees $\T^{A,\ba}$, which are essential for various arguments in following sections.
	
\begin{lem}\label{lem:ss}
For each alphabet $A$, each weight $\ba$, and each $i\in A$, the map
\[ \phi^{A,\ba}_i : \mathbb{T}^{A,\ba} \to \mathbb{T}^{A,\ba}_i,\qquad \phi^{A,\ba}_i([w]) = [iw]\] 
is a similarity with scaling factor $\ba(i)$.
\end{lem}
	
\begin{proof}
By Lemma \ref{lem:prop.of.graphs}(ii), for any $i\in A$, any $k\in\N$, and any two words $w,u\in A^k$ we have that $w$ is adjacent to $u$ in $G_k$ if and only if $iw$ is adjacent to $iu$ in $G_{k+1}$. Therefore, $A^{\N}_{w}\wedge A^{\N}_u \neq \emptyset$ if and only if $A^{\N}_{iw}\wedge A^{\N}_{iu} \neq \emptyset$.
		
Fix now $w,u \in A^{\N}$ and $i\in A$. We show that
\[ \rho_{A,\ba}(iw,iu) = \ba(i)\rho_{A,\ba}(w,u).\]
To see that, let first $A^{\N}_{v_1},\dots, A^{\N}_{v_N}$ be a chain joining $w$ with $u$. Then $A^{\N}_{iv_1},\dots, A^{\N}_{iv_N}$ is a chain joining $iw$ with $iu$. Therefore, $\rho_{A,\ba}(iw,iu) \leq \ba(i)\rho_{A,\ba}(w,u)$.
		
For the reverse inequality, we show that for any chain $A^{\N}_{v_1},\dots,A^{\N}_{v_N}$ joining $iw$ with $iu$, there exists another chain $A^{\N}_{iu_1},\dots A^{\N}_{iu_k}$ joining them that satisfies
\[ \D_{\ba}(iu_1)+\cdots+\D_{\ba}(iu_k) \leq \D_{\ba}(v_1)+\cdots+\D_{\ba}(v_N).\]
Noting that $A^{\N}_{u_1},\dots,A^{\N}_{u_k}$ is also a chain joining $w$, with $u$, assuming the claim, we readily have that $\rho_{A,\ba}(iw,iu) \geq \ba(i)\rho_{A,\ba}(w,u)$.
		
To prove the claim, assume without loss of generality that $i=1$ and fix a chain $A^{\N}_{v_1},\dots,A^{\N}_{v_N}$ joining $1w$ with $1u$. If there exists $j \in\{1,\dots,N\}$ such that $|v_j|\leq 1$, then $v_j\in \{1,\e\}$ and it is clear that the chain consisting only of $A^{\N}_1$ satisfies the conclusions of the claim. Assume now that $|v_j|\geq 2$ for all $j\in\{1,\dots,N\}$. Let $l,s$ be the first and last, respectively, indices $j$ such that $v_j \not\in A^{*}_1$. We know that $2\leq l \leq s \leq N-1$. Write $v_j = 1u_j$ for $j \in\{1,\dots,l-1,s+1,\dots,N\}$. Since $A^{\N}_{1u_{l-1}}\wedge A^{\N}_{v_l} \neq \emptyset$, there exist $i'\in A\setminus \{1\}$, integer $m\geq \max\{1+|u_{l-1}|,|v_l|\}$, $u'\in A^{m-|u_{l-1}|}$, and $u''\in A^{m-|v_{l}|}$ such that $\{1u_{l-1}u',i'v_{l}u''\} \in E_{m+1}^A$. By Lemma \ref{lem:prop.of.graphs}(1), $1u_{l-1}u' = 12^{(m)}$ which implies that $1u_{l-1}=12^{(m_1)}$ for some integer $m_1\geq 0$. Similarly, $1u_{s+1} = 12^{(m_2)}$ for some integer $m_2\geq 0$. Therefore, by Remark \ref{rem:1}, we have that $A^{\N}_{1u_{l-1}}\wedge A^{\N}_{1u_{s+1}} \neq \emptyset$ and the collection $A^{\N}_{1u_1},\dots,A^{\N}_{1u_{l-1}},A^{\N}_{1u_{s+1}},\dots,A^{\N}_{1u_N}$ is a chain satisfying the conclusions of the claim.
\end{proof}

Note that if $A=\{1,\dots, m\}$ is a finite alphabet, then Lemma \ref{lem:ss} is enough to show that all $\T^{A,\ba}$ are self-similar in the sense of \textsection\ref{sec:ss}. Indeed, $\T^{A,\ba}$ is the attractor of $\{ \phi^{A,\ba}_i: \mathbb{T}^{A,\ba} \to \mathbb{T}^{A,\ba} \}_{i\in A}$, and  $U:=\T^{A,\ba}\setminus\{ [1^{(\infty)}, 2^{(\infty)}] \}$ is the open set required for the open set condition.

\section{Geodesicity of trees $\T^{A,\ba}$}\label{sec:geodesicity}
	
The focus of this section is to show that the metric trees $\T^{A,\ba}$ defined above are all geodesic.
	
\begin{prop}\label{prop:geodesic}
For each alphabet $A$ and each weight $\ba$ the metric tree $\T^{A,\ba}$ is geodesic.  
\end{prop}
	
Fix for the rest of this section an alphabet $A$ and a weight $\ba$. Given $k\in\N$ and $w,u\in A^k$, we denote by $P_{A,k}(w,u)$ the unique combinatorial arc in $G_k^A$ with endpoints $w$, $u$.
	
\begin{rem}\label{rem:shiftpath}
By design of the graphs $(G_k^A)_{k\in\N}$ we have that for every $i\in A$, $k\in \N$, and $w,u\in A^k$, if $P_{A,k}(w,u) = (v_1,\dots,v_n)$, then $(iv_1,\dots,iv_n)$ is also an arc which we denote by $iP_{A,k}(w,u)$. Therefore,
\[ P_{A,k+1}(iw,iu) = iP_{A,k}(w,u). \]
\end{rem}
	
Recall the definition of length $\ell(\gamma)$ of a combinatorial path $\g$, and the concatenation $\g\cdot\g'$ of paths $\g,\g'$ in a graph $G$ from \textsection\ref{sec:graphprelim}. 
	
\begin{lem}\label{lem: maximal path}
For each $n\in\N$, 
\[ P_{A,n}(1^{(n-1)}3,2^{(n)}) = P_{A,n}(1^{(n-1)}3,12^{(n-1)})\cdot P_{A,n}(12^{(n-1)},2^{(n)}).\] 
Moreover, for any arc $\g$ in $G_n^A$,
\[ \ell(\gamma) \leq \ell(P_{A,n}(1^{(n-1)}3,2^{(n)})) = 2^n+1.\]
\end{lem}
	
\begin{proof}
For the first claim, fix $n\in\N$. By Lemma \ref{lem:prop.of.graphs}(i) we have that if $1w$ is adjacent to $2u$ in some graph $G^A_n$ (here $w,u\in A^{n-1}$), then $w=2^{(n-1)}$ and $u=1^{(n-1)}$. Since $G_n^A$ is a combinatorial tree, every path in $G^A_{n}$ that goes from $1^{(n-1)}3$ to $2^{(n)}$ must contain the (directed) edge $(12^{(n-1)},21^{(n-1)})$. This is by construction, due to the ``gluing'' between the isometric copies of $G_{n-1}^A$ with edges meeting at $12^{(n-1)}$ (see \textsection4.1). Therefore, 
\[ P_{A,n}(1^{(n-1)}3,2^{(n)}) = P_{A,n}(1^{(n-1)}3,12^{(n-1)})\cdot P_{A,n}(12^{(n-1)},2^{(n)}).\]
		
The proof of the second claim is by induction on $n$. For $n=1$, $P_{A,1}(3,2)=(3,1,2)$ is the arc of longest length in $G_{1}^A$ and $\ell(P_{A,1}(1,2))=3$. Assume now that the second claim is true for some $n\in\N$. First note that the only point adjacent to $1^{(n)}3$ in $G^A_{n+1}$ is $1^{(n+1)}$. Therefore, by the first claim and by Remark \ref{rem:shiftpath},
\begin{align*}
P_{A,n+1}(&1^{(n)}3,2^{(n+1)})\\ 
&= P_{A,n+1}(1^{(n)}3,12^{(n)}) \cdot (12^{(n)},21^{(n)})\cdot P_{A,n+1}(21^{(n)},2^{(n+1)})\\
&= (1P_{A,n}(1^{(n-1)}3,2^{(n)})) \cdot (12^{(n)},21^{(n)})\cdot (2P_{A,n}(1^{(n)},2^{(n)}))\\
&= (1P_{A,n}(1^{(n-1)}3,2^{(n)})) \cdot (12^{(n)},21^{(n)})\cdot (2P_{A,n}(1^{(n-1)}3,2^{(n)}) \setminus\{1^{(n-1)}3\})
\end{align*}
which by the inductive assumption implies that $\ell(P_{A,n+1}(1^{(n)}3,2^{(n+1)})) = 2^{n+1}+1$. 
		
Let now $P$ be an arc in $G_{n+1}$. If $P$ crosses vertices contained entirely in the vertices of $jG_n$ for some $j\in A$, then $P=jP'$ for some arc $P'$ in $G_n$. By the inductive hypothesis,
$$\ell(P)= \ell(P') \leq \ell(P_{A,n}(1^{(n-1)}3,2^{(n)}))= 2^n+1.$$ 
		
Assume for the rest of the proof that $P$ crosses vertices of at least two sub-graphs $jG_n$, $kG_n$ of $G_{n+1}$ for distinct $j, k\in A$. Note that $P$ does not cross vertices lying on three or more distinct sub-graphs, unless one of them is $1G_n$, which it only crosses at $12^{(n)}$. Indeed, suppose it contains vertices from the sub-graphs $jG_n$, $kG_n$, $lG_n$ of $G_{n+1}$ with distinct $j, k, l\in A$. By the inductive definition of $G_{n+1}$, any path joining vertices of two distinct sub-graphs has to contain $12^{(n)}$. This implies that one of the sub-graphs is $1G_n$. Without loss of generality, assume that $l=1$. Moreover, if $v_{i_1}$, $v_{i_2}$, $v_{i_3}$ are vertices of $jG_n$, $kG_n$, $1G_n\setminus\{ 12^{(n)}\}$, respectively, then the sub-paths of $P$ joining $v_{i_1}$ with $v_{i_2}$ and joining $v_{i_2}$ with $v_{i_3}$ both contain $12^{(n)}$, since $G_{n+1}$ is a combinatorial tree. Therefore, 
$$P=(v_1,\dots, v_{i_1},\dots, 12^{(n)}, \dots, v_{i_2}, \dots, 12^{(n)}, \dots, v_{i_3}, \dots v_M),$$ 
which is a contradiction. As a result, $P$ crosses vertices of either exactly two sub-graphs $jG_n$, $1G_n$ of $G_{n+1}$ for $j\in A\setminus\{1\}$, or exactly three sub-graphs $jG_n$, $kG_n$, $1G_n$ of $G_{n+1}$ for distinct $j, k\in A\setminus\{1\}$ and the only vertex of $1G_n$ it contains is $12^{(n)}$. In the latter case, note that $12^{(n)}$ cannot be an end-point of $P$, since then there is a vertex $q1^{(n)}$ of $G_{n+1}$ with $q\in A\setminus\{1, j, k\}$, i.e. which is not a vertex of $jG_n$, $kG_n$, and $\{12^{(n)}, q1^{(n)}\}\in E_{n+1}$, implying that the length of $P$ is not maximal in $G_{n+1}$.
		
Suppose $P$ intersects exactly three sub-graphs (the case where it intersects exactly two is similar). We decompose $P=(jP_1)\cdot((12^{(n)}))\cdot (kP_2)$, where $j, k\in A\setminus\{1\}$ are distinct and $P_1,P_2$ are paths in $G_n$. If none of $P_1$, $P_2$ is maximal in $G_n$, then by the inductive hypothesis
$$\ell(P)=\ell(jP_1)+1+\ell(kP_2)=\ell(P_1)+1+\ell(P_2)\leq 2^n+1+2^n=2^{n+1}+1.$$ 
Note that $\{v,12^{(n)}\}\in E_{n+1}$ if, and only if, $v=i1^{(n)}$ for some $i\in A\setminus \{1\}$. As a result, both $P_1$ and $P_2$ have $1^{(n)}$ as an end-point. However, we have $\{1,2\}, \{1,3\}\in E_{1}$, which implies $\{1^{(n)},1^{(n-1)}2\}, \{1^{(n)},1^{(n-1)}3\}\in E_n$ by Lemma \ref{lem:prop.of.graphs}(ii). Therefore, not all endpoints of $P_1$ and $P_2$ are leaves of $G_n$, implying that they are both not maximal in $G_n$. Hence, for every path $P$ in $G_{n+1}$, $\ell(P)\leq 2^{n+1}+1$.
\end{proof}
	
The choice of $3$ in the above Lemma was arbitrary, and all that was essentially needed was a letter distinct from $1$ and $2$, to demonstrate the ``in-between" combinatorial length between $1^{(n)}$ and $2^{(n)}$. This is further emphasized in the following corollary.
	
\begin{cor}\label{cor:1^nto2^n}
For every $n\in \N$, $\ell(P_{A,n}(1^{(n)},2^{(n)}))=2^n$.
\end{cor}
	
\begin{proof}
By Lemma \ref{lem:prop.of.graphs}(ii), for each $n\in\N$ we have $\{1^{(n-1)}3,1^{(n)}\}\in E_n^A$ and 
\[ P_{A,n}(1^{(n-1)}3,2^{(n)}) = (1^{(n-1)}3,1^{(n)})\cdot P_{A,n}(1^{(n)},2^{(n)}). \]
Hence, $\ell(P_{A,n}(1^{(n)},2^{(n)}))=2^n$.
\end{proof}
	
Given $w,u\in A^{\N}$, let $P_{A,n}(w(n),u(n)) = (v_1,\cdots,v_k)$ and denote by $\mathcal{C}_{A,n}(w,u)$ the chain $\{A^{\N}_{v_1},\dots,A^{\N}_{v_k}\}$. Moreover, given a chain $C = \{A^{\N}_{w_1},\dots,A^{\N}_{w_m}\}$ we set
\[ \ell_{\ba}(C) = \D_{\ba}(w_1) + \cdots + \D_{\ba}(w_m).\]
	
\begin{cor}\label{cor:shortchains}
Let $w\in A^*$, $n>|w|$, and $u, u'\in A^\N_w$. Then,
\[ \ell_{\ba}(\mathcal{C}_{A,n}(u,u')) \leq \D_{\ba}(w)(1+2^{|w|-n}).\]
\end{cor}

\begin{proof}
Since the subgraph of $G_n^A$ induced by vertices $A^{n}_w$ is connected, the combinatorial arc $P_{A,n}(u(n),u'(n))$ has all its vertices in $A^{n}_w$. Therefore, we can write $\mathcal{C}_{A,n}(u,u') = \{A^{\N}_{wv_1},\dots,A^{\N}_{wv_k}\}$ where $(v_1,\dots,v_k)$ is a combinatorial arc in $A^{n-|w|}$. By Lemma \ref{lem: maximal path}, $k\leq 2^{n-|w|}+1$. Since $\D_{\ba}(v_i) \leq 2^{|w|-n}$ for all $i$
\[ \ell_{\ba} (\mathcal{C}_{A,n}(u,u')) = \D_{\ba}(w)\left( \D_{\ba}(v_1)+\cdots+\D_{\ba}(v_k) \right) \leq \D_{\ba}(w)(1+ 2^{|w|-n}). \qedhere \]
\end{proof}
	
\begin{lem}\label{lem: geodesicity}
For all $w,u\in A^\N$, $\rho_{A,\ba}(w, u)=\lim_{n\to\infty}\ell_{\ba }(\mathcal{C}_{A,n}(w,u))$.
\end{lem}

\begin{proof}
Fix $w,u\in A^{\N}$. It suffices to show that for any chain $C$ connecting $w$ and $u$ and for any $\e \in (0,1)$ there exists $N\in\N$ such that for all $n\geq N$,
\begin{equation}\label{eq:formulaofmetric}
\ell_{\ba}(\mathcal{C}_{A,n}(w,u)) \leq \ell_{\ba}(C) + \e. 
\end{equation}
		
To this end, fix a chain $C=\{A^\N_{w_1},\dots,A^\N_{w_k}\}$ connecting $w$ and $u$ and $\e\in (0,1)$. Choose $N\in\N$ large enough such that 
\[ N > \frac{\log(k/\e)}{\log{2}} + \max_{i=1,\dots,k}|w_i|.\]
		
For each $i\in\{2,\dots,k\}$ there exists $v_{i-1}' \in A^{N}_{w_{i-1}}$ and $v_i \in A^{N}_{w_i}$ such that $v_{i-1}'$ is adjacent to $v_i$ in $G^A_N$. Set also $v_1=w(N)$ and $v_k'=u(N)$. Note that 
\[ \mathcal{C}_{A,N}(v_1,v_1') \cup \cdots \cup \mathcal{C}_{A,N}(v_k,v_k')\]
is a chain joining $w$ with $u$ and it contains $\mathcal{C}_{A,N}(w,u)$. By Corollary \ref{cor:shortchains} and the choice of $N$,
\begin{align*} 
\ell_{\ba}(\mathcal{C}_{A,N}(w,u)) &\leq \ell_{\ba}(\mathcal{C}_{A,N}(v_1,v_1')) + \cdots + \ell_{\ba}(\mathcal{C}_{A,N}(v_k,v_k'))\\
&\leq \D_{\ba}(w_1) (1+2^{|w_1|-N}) + \cdots + \D_{\ba}(w_k) (1+2^{|w_k|-N})\\
&\leq \D_{\ba}(w_1) (1+\e/k) + \cdots + \D_{\ba}(w_k) (1+\e/k)\\
&\leq \ell_{\ba}(C) + \e.
\end{align*}

By \eqref{eq:formulaofmetric}, we have that $\limsup_{n\to\infty}\ell_{\ba }(\mathcal{C}_{A,n}(w,u)) \leq \rho_{A,\ba}(w, u)$. On the other hand, $\ell_{\ba}(\mathcal{C}_{A,n}(w,u)) \geq \rho_{A,\ba}(w, u)$ for all $n$ so 
\[ \liminf_{n\to\infty}\ell_{\ba }(\mathcal{C}_{A,n}(w,u)) \geq \rho_{A,\ba}(w, u). \qedhere\]
\end{proof}
	
We are now ready to prove Proposition \ref{prop:geodesic}.
	
\begin{proof}[{Proof of Proposition \ref{prop:geodesic}}]
Fix $[w],[u] \in \T^{A,\ba}$ and let $v\in A^{\N}$ so that $[v]$ lies on the unique arc that connects $[w]$ with $[u]$.
Following the proof of \cite[Lemma 3.17]{DV}, we may assume that $v(n)\in P_{A,n}(w(n),u(n))$ for all $n\in \N$. Recall that $P_{A,n}(w(n),u(n)) = P_{A,n}(w(n),v(n))\cdot P_{A,n}(v(n),u(n))$. By Lemma \ref{lem: geodesicity},
\begin{align*}
d_{A,\ba}([w],[v]) + d_{A,\ba}([v],[u]) &= \lim_{n\to \infty} (\ell_{\ba}(\mathcal{C}_{A,n}(w,v)) + \ell_{\ba}(\mathcal{C}_{A,n}(v,u)))\\ 
&= \lim_{n\to \infty} (\ell_{\ba}(\mathcal{C}_{A,n}(w,u)) - \D_{\ba}(v(n)))\\
&= d_{A,\ba}([w],[u])
\end{align*}
which proves the geodesicity of $\T^{A,\ba}$. 
\end{proof}
	
\begin{cor}\label{cor:dist of 1^oo to 2^oo}
We have that $d_{A,\ba}([1^{(\infty)}],[2^{(\infty)}])=1$.
\end{cor}
	
\begin{proof}
Note that for each $n\in\N$, every vertex in $P_{A,n}(1^{(n)},2^{(n)})$ is in $\{1,2\}^n$ due to Lemma \ref{lem:graphembed}. Therefore, every vertex $v$ in $P_{A,n}(1^{(n)},2^{(n)})$ satisfies $\D_{\ba}(v) = 2^{-n}$. By Lemma \ref{lem: geodesicity} and Corollary \ref{cor:1^nto2^n}, 
\[ d_{A,\ba}([1^{(\infty)}],[2^{(\infty)}]) = \lim_{n\to \infty}\ell_{\ba}(\mathcal{C}_{A,n}(1^{(n)},2^{(n)})) = 1. \qedhere \]
\end{proof}
	
\begin{cor}\label{cor: diameters of tiles}
For any $w\in A^*$, $\diam\mathbb{T}_w^{A,\ba}=\Delta_{\ba}(w)$.
\end{cor}
	
\begin{proof}
Fix $w\in A^*$. On the one hand, $\diam{\mathbb{T}^A_w}\leq \D_{\ba}(w)$ \cite[Lemma 3.9]{DV}. On the other hand, by Lemma \ref{lem:ss} and Corollary \ref{cor:dist of 1^oo to 2^oo},
\[\diam\mathbb{T}_w^{A,\ba} \geq d_{A,\ba}([w1^{(\infty)}],[w2^{(\infty)}])=\Delta_{\ba}(w)\rho_{A,\ba}(1^{(\infty)},2^{(\infty)})=\Delta_{\ba}(w). \qedhere\]
\end{proof}
	
\begin{rem}\label{rem:leaves}
Corollary \ref{cor:dist of 1^oo to 2^oo} and Corollary \ref{cor: diameters of tiles}, along with the fact that $\T^{A,\ba}$ is geodesic, imply that both $[1^{(\infty)}]$ and $[2^{(\infty)}]$ are leaves of $\T^{A,\ba}$. This might at first seem surprising for $[1^{(\infty)}]$  based on Figures 3 and 4. However, the points $1^{(n)}j$, $j\in A$, are essentially identified  due to \eqref{eq:pseudometric} and $a_n\rightarrow 0$.
\end{rem}
	
\subsection{Isometric embeddings}\label{sec:isomembed}
	
Here we show that if $A\subset A'$ and $\ba$ is a weight, then $\T^{A,\ba}$ is essentially a subset of $\T^{A',\ba}$.
	
\begin{lem}\label{lem:isom-embed}
Let $A,A'$ be two alphabets with $A\subset A'$, and let $\ba,\ba'$ be two weights with $\ba'|_A = \ba|_A$. Then the map
\[ \Phi : \mathbb{T}^{A,\ba} \to \mathbb{T}^{A',\ba'}, \qquad  \Phi([w]) = [w]\]
is an isometric embedding. Moreover, if $A\subsetneq A'$ then
\[ \dist_H(\Phi(\mathbb{T}^{A,\ba}),\mathbb{T}^{A',\ba'}) \leq \max_{j\in A'\setminus A}\ba'(j).\]
\end{lem}
	
\begin{proof}
For the first claim of the lemma, fix $w,u\in A^\N \subset (A')^{\N}$. We show that $\rho_{A',\ba'}(w,u)=\rho_{A,\ba}(w,u)$. To this end, fix $n\in \N$ and write $P_{A,n}(w(n),u(n))=\{v_0,\dots,v_k\}$ and $P_{A',n}(w(n),u(n))=\{v'_0,\dots,v'_{k'}\}$. By Lemma \ref{lem:graphembed}, $G_n^A$ is a subgraph of $G_n^{A'}$, so implies that $P_{A,n}(w(n),u(n))$ is a combinatorial arc in $G_n^{A'}$. Since $G_n^{A'}$ is a combinatorial tree, it follows that $P_{A,n}(w(n),u(n)) = P_{A',n}(w(n),u(n))$; that is, $k=k'$ and $v_i=v'_i$ for all $i\in \{1,\dots,k\}$. Since, $v'_i\in A^{\N}$, we have $\Delta_{\ba}(v_i)=\Delta_{\ba'}(v_i)$. By Lemma \ref{lem: geodesicity}, 
$$\rho_{A',\ba'}(w,u)=\lim_{n\to\infty}\ell_{\ba'}(\mathcal{C}_{A',n}(w,u))=\lim_{n\to\infty}\ell_{\ba}(\mathcal{C}_{A,n}(w,u))=\rho_{A,\ba}(w,u).$$
		
For the second claim, assume that $A\subsetneq A'$. Since $\Phi(\mathbb{T}^{A,\ba})\subset \mathbb{T}^{A',\ba'}$,
$$\dist_H(\Phi(\mathbb{T}^{A,\ba}),\mathbb{T}^{A',\ba'})= \sup_{v\in (A')^{\N}\setminus A^{\N}}\inf_{w\in A^{\N}}\rho_{A',\ba'}(w,v).$$
Let $v\in (A')^{\N}\setminus A^{\N}$ and write $v=u_1ju$ where $u_1\in A^*$, $j\in A'\setminus A$ and $u\in (A')^{\N}$. By Lemma \ref{lem:ss}, Remark \ref{rem:equivalent classes of 12^oo}, and Corollary \ref{cor: diameters of tiles},
\begin{align*}
\inf_{w\in A^{\N}}\rho_{A',\ba'}(w,v) &\leq \rho_{A',\ba'}(u_112^{(\infty)}, u_1ju)\\
&=\Delta_{\ba'}(u_1) \rho_{A',\ba'}(12^{(\infty)}, ju)\\ 
&=\Delta_{\ba'}(u_1) \rho_{A',\ba'}(j1^{(\infty)}, ju)\\
&\leq \Delta_{\ba'}(u_1)\ba'(j) \\
&\leq \max_{j\in A'\setminus A}\ba'(j). \qedhere
\end{align*}
\end{proof}
	
Based on Lemma \ref{lem:isom-embed}, we make the following convention for the rest of the paper.
	
\begin{convention}
If $A\subset A'$ are two alphabets and $\ba$ is a weight as in Section \ref{sec:construction}, we write from now on $\T^{A,\ba} \subset \T^{A',\ba}$ identifying $\T^{A,\ba}$ with its isometric embedded image in $\T^{A',\ba}$.
		
If $A=\{1,\dots,n\}$ for some $n\in\{2,3,\dots\}$ and $\ba$ is a weight, we write $\T^{n,\ba} := \T^{A,\ba}$. If $A=\N$, then we write $\T^{\infty,\ba} := \T^{\N,\ba}$. Under the above notation and convention, we have for a weight $\ba$,
\[ \T^{2,\ba} \subset \T^{3,\ba} \subset \cdots \subset\T^{\infty,\ba}\]
and $\mathbb{T}^{n,\ba}$ converge in the Hausdorff sense to $\mathbb{T}^{\infty,\ba}$.
\end{convention}

\section{Branch points and valence of trees $\T^{m,\ba}$}\label{sec:branching}
	
Recall the notation $\T^{m,\ba}$ from the end of Section \ref{sec:geodesicity}. In this section we study the valence of trees $\T^{m,\ba}$. 
	
\begin{prop}\label{prop:valence}
Let $\ba$ be a weight. The space $\T^{2,\ba}$ is a metric arc. Moreover, if $m\in \{2,3,\dots\}$ and $\ba(m)>0$, then $\T^{m,\ba}$ is uniformly $m$-branching.
\end{prop}
	
Fix for the rest of this section an integer $m\in\{2,3,\dots\}$ and a weight $\ba$, and let $A=\{1,\dots,m\}$. The case where $m=2$ is given in \cite[Lemma 6.1]{DV}. Therefore, we may assume for the rest of this section that $m\geq 3$.
	
Recall the similarity maps $\phi^{A,\ba}_i$ from Lemma \ref{lem:ss}. To simplify the notation below, we drop the superindexes $A,\ba$ and simply write $\phi_i$. Given $w=i_1\cdots i_n \in A^*$ and $u\in A^{\N}$, we write
\[ \phi_w([u]) := (\phi_{i_1}\circ\cdots \circ \phi_{i_n})([u]) = [wu] \]
with the convention that $\phi_{\e}$ is the identity map. We also write for each $w\in A^*$
\[ \T^{m,\ba}_w = \phi_w(\T^{m,\ba}). \]
Note that $\T^{m,\ba}_w$ are subtrees of $\T^{m,\ba}$.
	
We split the proof of Proposition \ref{prop:valence} into several steps. First we show how different $\T^{m,\ba}_i$, $\T^{m,\ba}_j$ intersect in the following Lemma.
	
\begin{lem}\label{lem: preimage 12^00} 
\begin{enumerate}[(i)]
\item[(i)] For distinct $i, j\in A$, we have $\T^{m,\ba}_i\cap \T^{m,\ba}_j=\{ [12^{(\infty)}] \}$.
\item[(ii)] We have $[12^{(\infty)}]= \{12^{(\infty)},21^{(\infty)}, 31^{(\infty)}, \dots, m1^{(\infty)}\}$.
\item[(iii)] We have that $[1^{(\infty)}]\in \T^{m,\ba}_1\setminus(\bigcup_{j\geq 2}\T^{m,\ba}_j)$ and $[2^{(\infty)}]\in \T^{m,\ba}_2\setminus(\bigcup_{j\neq 2}\T^{m,\ba}_j)$.
\end{enumerate}
\end{lem} 
	
\begin{proof}
For (i), note that by Remark \ref{rem:equivalent classes of 12^oo} $[12^{(\infty)}]=[j1^{(\infty)}]\in\T^{m,\ba}_j$. Assume for a contradiction that there is $p\in \T^{m,\ba}_i\cap \T^{m,\ba}_j\setminus \{[12^{(\infty)}]\}$ for some distinct $i, j\in A$. Suppose $p=[u]=[v]$ where $u\in A_i^\N$ and $v\in A_j^\N$. By Lemma \ref{lem: geodesicity},
$$ d_{A,\ba}([u],[v])=\rho_{A,\ba}(u,v)=\lim_{n\rightarrow \infty} \ell_{\ba}(\mathcal{C}_{A,n}(u,v))=0.$$ 
However, $u(n)\in A_i^n$ and $v(n)\in A_j^n$, so the combinatorial arc $P_{A,n}(u(n),v(n))$ needs to cross the vertex $12^{(n-1)}$, implying that
$$P_{A,n}(u(n),v(n))= P_{A,n}(u(n), 12^{(n-1)})\cdot P_{A,n}(12^{(n-1)}, v(n))).$$
Hence,
$$\rho(u,12^{(\infty)})=\lim_{n\rightarrow \infty} \ell_{\ba}(\mathcal{C}_{A,n}(u,12^{(\infty)}))\leq \lim_{n\rightarrow \infty} \ell_{\ba}(\mathcal{C}_{A,n}(u,v))=0,$$ 
which contradicts $p=[u]\neq [12^{(\infty)}]$.
		
For (ii), the left inclusion follows by Remark \ref{rem:equivalent classes of 12^oo}. For the right inclusion let $\tilde{w}=i_1 i_2 \cdots\in A^\N \setminus\{12^{(\infty)},21^{(\infty)}, 31^{(\infty)}, \dots, m1^{(\infty)}\} $ with $[\tilde{w}]=[12^{(\infty)}]$. Assume that $i_1=1$; the case $i_1\neq 1$ can be treated similarly, by replacing $12^{(\infty)}$ with $j1^{(\infty)}$ in the following arguments. We show that $\rho_{A,\ba}(\tilde{w}, 12^{(\infty)})>0$. Let $n\geq2$ be the smallest integer such that $i_n\neq 2$. Then 
$$\rho_{A,\ba}(\tilde{w}, 12^{(\infty)})= \rho_{A,\ba}(1i_2 i_3 \cdots, 12^{(\infty)})=\ba(1)\ba(i_2)\cdots\ba(i_{n-1})  \rho_{A,\ba}(w, 2^{(\infty)}),$$ 
where $w=i_n i_{n+1} \cdots\notin A_2^\N$. Since $w\in A^{\N}\setminus A^{\N}_2$, by (i) and the geodesicity of $\T^{m,\ba}$,
\[ d_{A,\ba}([w],2^{(\infty)}) \geq d_{A,\ba}(12^{(\infty)},2^{(\infty)}) = \ba(2) \rho_{A,\ba}(1^{(\infty)}, 2^{(\infty)}) = \tfrac12.\]
Therefore, $\rho_{A,\ba}(\tilde{w}, 12^{(\infty)}) >0$.
		
Claim (iii) follows by (i) and (ii).
\end{proof} 
	
We need the following general result about connected components of subsets of a metric space. 
	
\begin{lem}[{\cite[Theorem 11.5.3]{Sear07}}]\label{Prop: metric compon}
Let $X$ be a metric space and $E$ a non-empty subset of $X$. If $A_1, \dots, A_n\subset E$, $n\in \N$, are non-empty, pairwise disjoint, relatively closed, and connected sets with $E=A_1\cup \dots \cup A_n$, then these are the components of $E$.  
\end{lem}
	
The next lemma shows that all points $\{[u12^{(\infty)}] : u\in A^*\}$, have valence $m$.
	
\begin{lem}\label{lem: images of 12^oo are m-val}
Suppose $\ba(m)>0$.
\begin{enumerate}[(i)]
\item[{(i)}] The  components of $\T^{m,\ba}\setminus\{[12^{(\infty)}]\}$ are exactly the sets $\T^{m,\ba}_i\setminus\{[12^{(\infty)}]\}$, $i\in A$. 
\item[{(ii)}] \ If $u\in A^*$, then $\T^{m,\ba}\setminus\{[u12^{(\infty)}]\}$ has exactly $m$ components. Each set $\T^{m,\ba}_{ui}\setminus\{[u12^{(\infty)}]\}$, for $i\in A$, lies in a different component of $\T^{m,\ba}\setminus\{[u12^{(\infty)}]\}$. 
\end{enumerate} 
\end{lem}
	
\begin{proof}
For (i), by Remark \ref{rem:leaves}, the sets $\T^{m,\ba}\setminus\{[1^{(\infty)}]\}$ and $\T^{m,\ba}\setminus\{[2^{(\infty)}]\}$ are  non-empty,  and connected.  Note that $\T^{m,\ba}_1\setminus\{[12^{(\infty)}]\}= \phi_1(\T^{m,\ba}\setminus\{[2^{(\infty)}]\})$ and $\T^{m,\ba}_j\setminus\{[12^{(\infty)}]\}=\phi_j(\T\setminus\{[1^{(\infty)}]\})$ for all $j\in A\setminus\{1\}$ by Lemma \ref{lem: preimage 12^00}(ii). Hence, the sets $\T^{m,\ba}_i\setminus\{[12^{(\infty)}]\}$, $i\in A$ are non-empty and connected. They are also relatively closed in $\T^{m,\ba}\setminus\{[12^{(\infty)}]\}$ and  pairwise disjoint by Lemma \ref{lem: preimage 12^00}(i). Then
$$\T^{m,\ba}\sm\{[12^{(\infty)}]\} = \bigcup_{j\in A} \T^{m,\ba}_j\sm\{[12^{(\infty)}]\}.$$ 
and by Lemma \ref{Prop: metric compon}, the sets $\T_j^{m,\ba}\sm\{[12^{(\infty)}]\}$, $j\in A$, are the components of $\T^{m,\ba}\setminus\{[12^{(\infty)}]\}$.
		
For (ii), we apply induction on the length $n\geq 0$ of $u\in A^*$. If $n=0$, then $u=\e$ and the statement follows by (i).
		
Suppose the statement is true for all words of length $n-1$ and let $u=iu'\in A^n$ for some $u'\in A^{n-1}$. Assume as we may that $i=1$; the cases $i\in A\sm\{1\}$ follow similarly. Note that $1u'12^{(\infty)} \not\in \{ 12^{(\infty)}, 21^{(\infty)}, \dots, m1^{(\infty)} \}$ so by Lemma \ref{lem: preimage 12^00}(ii) $[u12^{(\infty)}]\ne [12^{(\infty)}]$. Since the first letter of $u$ is $1$, $[u12^{(\infty)}]\in \T^{m,\ba}_1\sm\{[12^{(\infty)}]\}$. 
		
By induction hypothesis, $\T^{m,\ba}\sm \{[u'12^{(\infty)}]\}$ has exactly $m$ connected components $V_1, \dots, V_m$. After re-ordering the components, we may assume that $ \T^{m,\ba}_{u'k}\sm  \{[u'12^{(\infty)}]\} \subset V_k$ for all $k\in A$. It follows that 
$$\phi_1(\T^{m,\ba}\sm \{[u'12^{(\infty)}]\})= \T^{m,\ba}_1\sm \{[u12^{(\infty)}]\}$$ 
has exactly $m$ connected components $U_k=\phi_1(V_k)\subset \T^{m,\ba}_1$, $k\in A$, with 
$$\T^{m,\ba}_{uk}\sm\{[u12^{(\infty)}]\} = \phi_1(\T^{m,\ba}_{u'k}\sm  \{[u'12^{(\infty)}]\}) \subset \phi_1( V_k)=U_k.$$ 
		
Fix  $k\in A$. Since  $V_k$ is a component of $\T^{m,\ba}\sm \{\phi_{u'}([12^{(\infty)}])\}$, it is relatively closed in $\T^{m,\ba}\sm \{[u'12^{(\infty)}]\}$, which implies $V_k=\overline {V_k}\cap (\T^{m,\ba}\sm \{[u'12^{(\infty)}]\})$. As a result,
\begin{align*}
U_k&=\phi_1(V_k) =\overline {\phi_1(V_k)}\cap (\T_1^{m,\ba}\sm \{[u12^{(\infty)}]\})= \overline {U_k} \cap  (\T_1^{m,\ba}\sm \{[u12^{(\infty)}]\}).
\end{align*}   
Since $\T_1^{m,\ba}\subset \T^{m,\ba}$ is compact, $U_k\subset \T_1^{m,\ba}$ implies $\overline {U_k} \subset \T_1^{m,\ba}$.  As a result, all limit points of $U_k$ other than $[u12^{(\infty)}]$ lie in $U_k$, implying that $U_k$ is relatively closed in $\T^{m,\ba}\sm  \{[u12^{(\infty)}]\}$. 
		
Exactly one of the components of $\T_1^{m,\ba}\sm \{[u12^{(\infty)}]\}$ needs to contain $[12^{(\infty)}]\in \T_1^{m,\ba}\sm \{[u12^{(\infty)}]\}$. Suppose $[12^{(\infty)}]\in U_1$ and the proof is analogous otherwise. Then $U'_1\coloneqq U_1\cup \bigcup_{j= 2}^m \T_j^{m,\ba}$ is a relatively closed subset of $\T^{m,\ba} \sm\{[u12^{(\infty)}]\}$. This set is also connected, because $U_1$, $\T_2^{m,\ba}, \dots, \T_m^{m,\ba}$  are connected and intersect at $[12^{(\infty)}]$.  Hence  the connected sets $U'_1$, $U_2,\dots, U_m$  are pairwise disjoint, relatively closed in $ \T^{m,\ba} \sm  \{[u12^{(\infty)}]\}$, and 
$$ \T^{m,\ba} \sm  \{[u12^{(\infty)}]\}=  (\T_1^{m,\ba} \sm \ \{[u12^{(\infty)}]\})\cup \left(\bigcup_{j= 2}^m\T_j^{m,\ba}\right)=U'_1\cup \left(\bigcup_{j=2}^m U_j\right). $$ 
This implies by Lemma \ref{Prop: metric compon} that $\T^{m,\ba} \sm  \{[u12^{(\infty)}]\}$ has exactly $m$ connected components that are $U'_1$, $U_2, \dots, U_m$. Moreover,  the sets $\T_{uj}^{m,\ba}\sm\{[u12^{(\infty)}]\}$, $j\in A$, lie in the different components $U_1', U_2, \dots, U_m$ of $\T^{m,\ba} \sm  \{[u12^{(\infty)}]\}$, respectively. This completes the proof.  
\end{proof}
	
We need the following property of relative boundaries in $\T^{m,\ba}$, in order to show that all branch points of $\T$ have a specific form.

\begin{lem}\label{lem: bdry Tu}
For all $n\in \N$ and $u\in A^n$, $\partial \T^{m,\ba}_u\subset \{[u1^{(\infty)}], [u2^{(\infty)}]\}$. Moreover, $\partial \T^{m,\ba}_u\neq \emptyset$ and if  $p\in \partial \T^{m,\ba}_u$, there exists $w\in \bigcup_{k=0}^{n-1}A^{k}$ such that $p=[w12^{(\infty)}]$.
\end{lem}

\begin{proof} 
The proof is by induction on $n$. For $n=1$, let $u\in A$. Then $\T^{m,\ba}_u\sm\{[12^{(\infty)}]\}$ is a  component of $\T^{m,\ba}\sm\{[12^{(\infty)}]\}$  by Lemma~\ref{lem: images of 12^oo are m-val}(i). Since $\T^{m,\ba}$ is a metric tree, this implies that $\T^{m,\ba}_u\sm\{[12^{(\infty)}]\}$ is a relatively  open set, its points lie in the relative interior of  $\T_u^{m,\ba}$, do not lie in $\partial \T_u^{m,\ba}$, and  $\partial\T_u^{m,\ba}=\{ [12^{(\infty)}] \}$ (see, for instance, \cite[Lemma 3.2]{BT_CSST}).  Since $A^0=\{\e\}$ and
$$[12^{(\infty)}]= \phi_\e([12^{(\infty)}])=\phi_1([2^{(\infty)}])=\phi_2([1^{(\infty)}])=\dots =\phi_m([1^{(\infty)}]), $$
the lemma is true for $n=1$. 
		
Assume the statement of the lemma to be true for all $v\in A^n$ for fixed $n\in \N$. Let  $u\in A^{n+1}$ with $u=vj$ for $v\in A^n$ and $j\in A$. Without loss of generality, we assume that $j=1$. Since  $\T_1^{m,\ba}\sm \{[12^{(\infty)}]\}$ is open in $\T^{m,\ba}$, the set
$$\phi_v( \T_1^{m,\ba}\sm \{[12^{(\infty)}]\})=\T_u^{m,\ba}\sm\{[v12^{(\infty)}]\}$$ 
is a relatively open subset  of $\T_v^{m,\ba}$. Suppose $p\in \T_u^{m,\ba}$ is not an interior point of $\T_u^{m,\ba}$ in $\T^{m,\ba}$. Then either $p=[v12^{(\infty)}]$ or $p$ lies on the boundary of $\T_v^{m,\ba}$. By the inductive hypothesis for $n$, 
\begin{equation}\label{eq: bdry points}
\partial \T_u^{m,\ba}\subset \{[v12^{(\infty)}]\}\cup \partial \T_v^{m,\ba}\subset \{[v12^{(\infty)}], [v1^{(\infty)}], [v2^{(\infty)}]\}. 
\end{equation}
By the inductive hypothesis, every element of $\partial \T_v^{m,\ba}$ can be written in the form $[w12^{(\infty)}]$ with $w\in \bigcup_{k=1}^{n-1}A^k$. Thus, every element of $\partial \T_u^{m,\ba}$ can be written in the form $[w12^{(\infty)}]$ with $w\in \bigcup_{k=1}^{n}A^k$.

Since $\T_u^{m,\ba}$ is compact, it is closed in $\T^{m,\ba}$, implying $\partial \T_u^{m,\ba}\subset \T_u^{m,\ba}$. As a result, by \eqref{eq: bdry points} it is enough to show that $\T_u^{m,\ba}$ contains at most two of  $[v12^{(\infty)}]$, $[v1^{(\infty)}]$, $[v2^{(\infty)}]$. Since  $[2^{(\infty)}]\not\in \T_1^{m,\ba}$ by Lemma \ref{lem: preimage 12^00}(iii), we have $[v2^{(\infty)}]\not\in \phi_v(\T_1^{m,\ba})=\T_u^{m,\ba}$. It follows that $ \partial \T_u^{m,\ba}\subset \{ [v12^{(\infty)}], [v1^{(\infty)}]\}$.
		
Assume towards contradiction that $[v12^{(\infty)}]$, $[v1^{(\infty)}]$ lie in the interior of $\T_u^{m,\ba}$, i.e., $\partial\T_u^{m,\ba}=\emptyset$. Then there are relatively open neighborhoods $N_{v12^{(\infty)}}, N_{v1^{(\infty)}}\subset \T_u^{m,\ba}$ of $[v12^{(\infty)}], [v1^{(\infty)}]$, respectively. If $v=v_1\dots v_n$, this implies that 
\[ [v_2\dots v_n12^{(\infty)}]\in \phi_{v_1}^{-1}(N_{v12^{(\infty)}})\subset \T_{v_2\dots v_n1}^{m,\ba}\] 
and $[v_2\dots v_n1^{(\infty)}]\in \phi_{v_1}^{-1}(N_{v1^{(\infty)}})\subset \T_{v_2\dots v_n1}^{m,\ba}$. Since $\phi_{v_1}$ is continuous, the two points $[v_2\dots v_n12^{(\infty)}], [v_2\dots v_n1^{(\infty)}]$ lie in the interior of $\T_{v_2\dots v_n1}^{m,\ba}$, which contradicts the inductive hypothesis. Therefore, $\partial\T_u^{m,\ba}\neq\emptyset$ and the inductive step is complete.
\end{proof} 
	
The next lemma gives us all the branch points of $\T^{m,\ba}$.
	
\begin{lem}\label{lem:branch points are 12oo}
Every branch point of $\T^{m,\ba}$ is of the form $[u12^{(\infty)}]$ for some $u\in A^*$.
\end{lem}
	
\begin{proof}
Assume towards a contradiction that $[w]$ is a branch point of $\T^{m,\ba}$ with $[w]\not\in \{[u12^{(\infty)}]: u\in A^*\}$. Let $U_1$, $U_2$, $U_3$ be three distinct components of $\T^{m,\ba}\sm\{[w]\}$ and fix $x_i\in U_i$ for all $i\in \{1,2,3\}$. Since the diameter $\diam{\T_{w(n)}^{m,\ba}}$ decreases to 0 as $n$ increases, we can find some $n\in \N$ such that
$$\min \{ d_{A,\ba}(x_i,[w])|: \, i=1, 2, 3 \}> \diam \T_{w(n)}^{m,\ba}.$$
Hence, $x_i\notin \T_{w(n)}^{m,\ba}$ for all $i\in \{1,2,3\}$. Since $[w]\not\in \{[u12^{(\infty)}]: u\in A^*\}$ Lemma \ref{lem: bdry Tu} implies that $[w]$ lies in the relative interior of $\T_{w(n)}^{m,\ba}$ in $\T^{m,\ba}$. Let $\g_i$ be the unique arc in $\T^{m,\ba}$ joining $x_i$ with $[w]$ and let $y_i \in \partial\T_{w(n)}^{m,\ba}$ be the first boundary point of $\T_{w(n)}^{m,\ba}$ on $\g_i$ going from $x_i$ towards $[w]$. Since $[w]$ lies in the interior of $\T_{w(n)}^{m,\ba}$, we have $y_i\neq [w]$ for all $i$. Let $\g_i'$ be the subarc of $\g_i$  with endpoints  $x_i$, $y_i$. Since $\g_i'\cap (\T^{m,\ba}\sm\{[w]\})$ is connected and since $x_i\in \g_i'$, we have that $\g_i'\subset U_i$, and, hence,  $y_i\in U_i$ for all $i\in \{1,2,3\}$. However, the points $y_1, y_2, y_3 \in \partial \T_{w(n)}^{m,\ba}$ cannot be distinct, since by  Lemma \ref{lem: bdry Tu} the set $\partial \T_{w(n)}^{m,\ba}$ consists of at most two points. This contradiction completes the proof.
\end{proof}
	
In the next lemma we calculate the height of the branch points of $\T^{m,\ba}$. Recall that the height of a branch point is the third largest branch (in terms of diameter) at that branch point, see \eqref{eq:height}.
	
\begin{lem}\label{lem:height formula}
If $w\in A^*$ and $B_1,\dots, B_m $ are the branches of the branch point $p=[w12^{(\infty)}]$ of $\T^{m,\ba}$, then with suitable labeling we have
$$\T_{w1}^{m,\ba} \subset B_1, \ \T_{w2}^{m,\ba} \subset B_2, \  \T_{wj}^{m,\ba}=B_j,$$ 
for all $j\in\{3,\dots,m\}$. In particular, $H_{\T^{m,\ba}}(p)=\ba(3)\Delta(w)$. 
\end{lem}
	
\begin{proof}
Let $w \in A^*$, then $p = [w12^{(\infty)}]$ is a branch point of $\T_w^{m,\ba}$ with $m$ distinct branches $\T_{w1}^{m,\ba}, \dots, \T_{wm}^{m,\ba}$. Hence there exist unique distinct branches $B_1,\dots, B_m$ of $p$ in $\T^{m,\ba}$ such that $\T_{wk}^{m,\ba} \subset B_k$ for $k\in\{1,\dots,m\}$ by \cite[Lemma 3.2(i)]{BM22}. By Lemma \ref{lem: bdry Tu} we have $\partial \T_w^{m,\ba}\subset \{[w1^{(\infty)}], [w2^{(\infty)}]\}$. Moreover, $[w1^{(\infty)}], [w2^{(\infty)}]\notin \T_{w3}^{m,\ba}$, since $[1^{(\infty)}], [2^{(\infty)}]\notin \T_j^{m,\ba}$ for all $j\in\{3,\dots,m\}$. Thus $\T_{w3}\cap \partial\T_w^{m,\ba}=\emptyset$. Hence it follows by \cite[Lemma 3.2 (iii)]{BM22} that $\T_{wj}^{m,\ba}$ is actually a branch of $p$ in $\T^{m,\ba}$ for all $j\in\{3,\dots,m\}$, and so $\T_{wj}^{m,\ba} = B_j$. 
		
For the second part of the statement, we have $\diam B_k\geq\diam \T_{wk}^{m,\ba}=2^{-1}\Delta(w)$ for $k=1,2$. But for $j\in\{3,\dots,m\}$, $\diam B_j=\diam \T_{wj}^{m,\ba}=\ba(j)\Delta(w)$. Hence, $H_{\T^{m,\ba}}(p)=\diam B_3=\ba(3)\Delta(w)$. 
\end{proof}
	
The next lemma is a useful characterization of the words that belong in the same class of $\mathbb{T}^{m,\ba}$.
	
\begin{lem}\label{lem: branch points belong in the same class}
Let $v,w\in A^{\N}$ with $v\neq w$. Then $[v]=[w]$ if and only if there exists $u\in A^*$ such that $v,w\in\{u12^{(\infty)}\}\cup\{uj1^{(\infty)} : j\geq 2\}$. In this case, $[v]=[w]=[u12^{(\infty)}]$.
\end{lem}
	
\begin{proof}
Suppose that $[v]=[w]$ for some $v,w\in A^{\N}$ with $v\neq w$. Let $u\in A^*$ be the longest common initial word of $v$ and $w$ i.e. $v=uv'$ and $w=uw'$ with $v',w'\in A^{\N}$ and $v'(1)\neq w'(1)\in A$. We have 
$$0=\rho_{A,\ba}(v,w)=\rho_{A,\ba}(uv',uw')=\Delta(u)\rho_{A,\ba}(v',w'),$$
which means that $\rho_{A,\ba}(v',w')=0$. Hence, $[v']=[w']$. But $[v']\in\mathbb{T}_{v'(1)}^{m,\ba}$ and $[w']\in\mathbb{T}_{w'(1)}^{m,\ba}$. So, due to $v'(1)\neq w'(1)$ and Lemma \ref{lem: preimage 12^00}(i), $[v']=[12^{(\infty)}]=[w']$. Therefore, $[v]=[u12^{(\infty)}]=[w]$, which by Lemma \ref{lem:branch points are 12oo} are also branch points of $\mathbb{T}^{m,\ba}$. The reverse implication follows from \ref{lem: preimage 12^00}(ii) and the above shrinking property of  $\rho_{A,\ba}$.
\end{proof}
	
We now estimate the distance between branch points.
	
\begin{lem}\label{lem:dist of branch points}
If $v,w\in A^*$, $v\neq w$ then 
$$d_{A,\ba}([v12^{(\infty)}],[w12^{(\infty)}])\geq \frac{1}{2}\min\{\Delta_{\ba}(v),\Delta_{\ba}(w)\}.$$
\end{lem}
	
\begin{proof}
We start the proof by making a useful observation. By Lemma \ref{lem:ss}, Corollary \ref{cor:dist of 1^oo to 2^oo}, and Lemma \ref{lem: preimage 12^00}, if $i\in A$ and $u\in A^*$, then 
\begin{align}\label{eq:branch dist}
d_{A,\ba}([u12^{(\infty)}],[ui12^{(\infty)}]) =\Delta_{\ba}(u)\rho_{A,\ba}(12^{(\infty)},i12^{(\infty)})
&= \Delta_{\ba}(u)\rho_{A,\ba}(i1^{(\infty)},i12^{(\infty)})\\ 
&= \frac12\Delta_{\ba}(ui).\notag
\end{align}
		
Let $v,w\in A^*$ with $v\neq w$, and let $u\in A^*$ be the longest common initial word of $v$ and $w$. Then $v=uv'$ and $w=uw'$ with $v',w'\in A^*$, where either $v'$ or $w'$ is the empty word, or both $v'$ and $w'$ are non-empty, but have different initial letter. Accordingly, we consider two cases.
		
\emph{Case 1}: $v'=\e$ or $w'=\e$. We may assume that $w'=\e$. Then $v'\neq \e$, and so $|v|\geq |w|$. Write $v'=i_1\cdots i_n$. By triangle inequality and \eqref{eq:branch dist}
\begin{align*}
&d_{A,\ba}([v12^{(\infty)}],[w12^{(\infty)}])\\ 
&= \Delta_{\ba}(u)d_{A,\ba}([v'12^{(\infty)}],[12^{(\infty)}])\\
&\geq \Delta_{\ba}(u)\left( d_{A,\ba}([v'(1)12^{(\infty)}]),[12^{(\infty)}]) - \sum_{j=1}^{n-1}d_{A,\ba}([v'(j)12^{(\infty)}],[v'(j+1)12^{(\infty)}])\right)\\
&= \frac12\Delta_{\ba}(u)\D_{\ba}(v'(1)) - \Delta_{\ba}(u)\frac12\sum_{j=1}^{n-1}\D_{\ba}(v'(j+1))\\
&\geq \frac{1}{2}\ba(i_1)\Delta_{\ba}(u) - \Delta_{\ba}(u)\frac12\ba(i_1)\sum_{j=1}^{n-1}2^{-j}\\
&\geq \frac{1}{2}\ba(i_1)\Delta_{\ba}(u)2^{1-n}\\
&\geq \frac{1}{2}\Delta_{\ba}(u)\Delta_{\ba}(v')\\
&= \frac{1}{2}\Delta_{\ba}(v).
\end{align*}
		
\emph{Case 2}: $v',w'\neq \e$. Then $v'$ and $w'$ necessarily start with different letters, that is, $v'=iv''$ and $w'=jw''$ with $i,j\in A$, $i\neq j$, and $v'',w''\in A^*$. Then $[v'12^{(\infty)}]\in \mathbb{T}_i^{m,\ba}$ and $[w'12^{(\infty)}]\in \mathbb{T}_j^{m,\ba}$. By Proposition \ref{prop:geodesic} and Case 1 we have that
\begin{align*}
d_{A,\ba}([v12^{(\infty)}],&[w12^{(\infty)}])\\
&= \Delta_{\ba}(u)d_{A,\ba}([v'12^{(\infty)}],[w'12^{(\infty)}]) \\
&=\Delta_{\ba}(u)\left( d_{A,\ba}([v'12^{(\infty)}],[12^{(\infty)}]) + d_{A,\ba}([12^{(\infty)}],[w'12^{(\infty)}]) \right) \\
&\geq \frac{1}{2}\Delta_{\ba}(v) + \frac{1}{2}\Delta_{\ba}(w).\qedhere
\end{align*}
\end{proof}
	
\begin{proof}[Proof of Proposition \ref{prop:valence}]
By Lemma \ref{lem: images of 12^oo are m-val} and Lemma \ref{lem:branch points are 12oo}, $\mathbb{T}^{m,\ba}$ is an $m$-valent metric tree. 
		
Fix a branch point $p\in \T^{m,\ba}$. By Lemma \ref{lem:branch points are 12oo}, $p=[w12^{(\infty)}]$ for some $w\in A^*$ and $\T^{m,\ba}\setminus \{p\}$ has exactly $m$ branches. Therefore, $\mathbb{T}^{m,\ba}$ is an $m$-valent metric tree and by Lemma \ref{lem:height formula} for all $j,j' \in \{3,\dots,m\}$
\[  \frac{\diam{B^j_{\mathbb{T}^{m,\ba}}(p)}}{\diam{B^{j'}_{\mathbb{T}^{m,\ba}}(p)}} \leq \frac{\ba(3)}{\ba(m)} \]
which shows the uniform branch growth of $\T^{m,\ba}$. 
		
Let $p,q\in\mathbb{T}^{m,\ba}$ be two distinct branch points of $\mathbb{T}^{m,\ba}$. Then $p=[v12^{(\infty)}]$ and $q=[w12^{(\infty)}]$ with $v,w\in A^*$, $v\neq w$. By Lemmas \ref{lem:height formula},\ref{lem:dist of branch points},
\begin{align*}
d_{A,\ba}(p,q) \geq \frac{1}{2}\min\{\Delta_{\ba}(v),\Delta_{\ba}(w)\} &\geq \ba(3)\min\{\Delta_{\ba}(v),\Delta_{\ba}(w)\}\\
&= \min\{H_{\mathbb{T}^{m,\ba}}(p),H_{\mathbb{T}^{m,\ba}}(q)\}.
\end{align*}
This shows that $\mathbb{T}^{m,\ba}$ has branch points that are uniformly relatively separated. 
		
To establish uniform density, let $[w],[v]\in \mathbb{T}^{m,\ba}$ with $[w]\neq [v]$. Let $u\in A^*$ be the longest common initial word of $v$ and $w$, that is, $w=uw'$ and $v=uv'$ with $w',v'\in A^{\N}$ and $w'(1)\neq v'(1)$. Hence $[w]\in \mathbb{T}_{uw'(1)}^{m,\ba}$ and $[v]\in \mathbb{T}_{uv'(1)}^{m,\ba}$. Note that by Remark \ref{rem:equivalent classes of 12^oo} and Lemma \ref{lem: preimage 12^00}(i) the unique arc in $\T^{m,\ba}$ that joins $[w]$ with $[v]$ goes through the branch point $p=[u12^{(\infty)}]$ which satisfies
\begin{align*}
H_{\mathbb{T}^{m,\ba}}(p) =\ba(3)\Delta_{\ba}(u) &\geq \ba(3)\Delta_{\ba}(u)d_{A,\ba}([w'],[v'])\\
&=\ba(3)d_{A,\ba}([w],[v]).
\end{align*}
This completes the proof of the proposition.
\end{proof}
	
\begin{rem}\label{rem:constants of T}
Note that $\T^{\infty,\ba}$ has also uniformly separated and uniformly dense branch points (but $\T^{\infty,\ba}$ doesn't have uniform branch growth). If $x$ is a branch point of $\T^{\infty,\ba}$ then it is a branch point of $\T^{m,\ba}$ for some $m\in\N$ and $H_{\T^{\infty,\ba}}(x)=H_{\T^{m,\ba}}(x)$ by Lemma \ref{lem:isom-embed}. Observe that the uniform separation and uniform density constant of $\T^{m,\ba}$ is $\ba(3)^{-1}$ and the uniform branch growth constant is $\ba(3)/\ba(m)$.
\end{rem}
	
For $w\in A^n$, we call the sets of the form $\mathbb{T}_w^{m,\ba}=\phi_w^{A,\ba}(\mathbb{T}^{m,\ba})$ the $n$-\textit{tiles} of $\mathbb{T}^{m,\ba}$. We finish this section with two lemmas that are useful in later sections.
	
\begin{lem}\label{lem: tiles}
The tile $\mathbb{T}_{11}^{m,\ba}$ is the only tile $X$ of $\mathbb{T}^{m,\ba}$ with $[1^{(\infty)}]\in X$ and $[112^{(\infty)}]\in\partial X$. The tile $\mathbb{T}_{22}^{m,\ba}$ is the only tile $X$ of $\mathbb{T}^{m,\ba}$ with $[2^{(\infty)}]\in X$ and $[212^{(\infty)}]\in\partial X$. 
\end{lem}

\begin{proof}
First, by Lemma \ref{lem: bdry Tu} $[1^{(\infty)}]\in\mathbb{T}_{11}^{m,\ba}$ and  $\partial\mathbb{T}_{11}^{m,\ba}\subset\{[1^{(\infty)}],[112^{(\infty)}]\}$. By Remark \ref{rem:leaves} $[1^{(\infty)}]$ is a leaf hence $\partial\mathbb{T}_{11}^{m,\ba}=\{[112^{(\infty)}]\}$ (note that by Lemma \ref{lem: bdry Tu} the tile $\mathbb{T}^{m,\ba}$ is the only tile with empty boundary). For the uniqueness let $\mathbb{T}_u^{m,\ba}$, $u\in A^n$ with $n\in\N_0$, be a tile such that $[1^{(\infty)}]\in\mathbb{T}_u^{m,\ba}$ and $[112^{(\infty)}]\in\partial\mathbb{T}_u^{m,\ba}$. That implies that $u=1^{(n)}$. Also, by Lemma \ref{lem: bdry Tu} $\partial\mathbb{T}_u^{m,\ba}\subset \{[u1^{(\infty)}],[u2^{(\infty)}]\}=\{[1^{(\infty)}], [1^{(n)}2^{(\infty)}]\}$. Hence either $[112^{(\infty)}]=[1^{(\infty)}]$ or $[112^{(\infty)}]=[1^{(n)}2^{(\infty)}]$. For the first case, that would imply that 
$$0=\rho_{A,\ba}(112^{(\infty)},1^{(\infty)})={\ba(1)}^2\rho_{A,\ba}(1^{(\infty)},2^{(\infty)})>0,$$
where the last follows from Corollary \ref{cor:dist of 1^oo to 2^oo}. For the second case we have
$$0=\rho_{A,\ba}(112^{(\infty)},1^{(n)}2^{(\infty)})=\ba(1)^2\rho_{A,\ba}(2^{(\infty)},1^{(n-2)}2^{(\infty)})>0$$
where the last inequality holds for $n>2$ since the only words that belong to the same class are of the form $u12^{(\infty)},uj1^{(\infty)}$ from Lemma \ref{lem: branch points belong in the same class}. Hence, $n=2$ and $u=11$. The second statement is similar and we omit the details.
\end{proof}
	
\begin{lem}\label{lem:neighbortiles}
Let $\ba$ be a weight, let $n\in\N$, and let $w,u \in A^n$ be distinct with $\T^{m,\ba}_w\cap \T^{m,\ba}_u \neq \emptyset$. Then
\begin{equation}\label{eq:neighbortiles}
2\min_{i=1,\dots,m}\ba(i) \leq \frac{\D_{\ba}(w)}{\D_{\ba}(u)} \leq \left(2\min_{i=1,\dots,m}\ba(i)\right)^{-1}.
\end{equation} 
\end{lem}
	
\begin{proof}
The proof is by induction on the length $n$. If $n=1$, then $\T^{m,\ba}_i\cap \T^{m,\ba}_j \neq \emptyset$ for all $i,j\in A$ and \eqref{eq:neighbortiles} is clear. Assume now that for some $n\in\N$, \eqref{eq:neighbortiles} is true whenever $w,u \in A^n$ are distinct and $\T^{m,\ba}_w\cap \T^{m,\ba}_u \neq \emptyset$. Fix now distinct $w,u \in A^{n+1}$ such that $\T^{m,\ba}_w\cap \T^{m,\ba}_u \neq \emptyset$. If there exists $i\in A$ such that $w = iw'$ and $u = iu'$, then \eqref{eq:neighbortiles} follows by the inductive hypothesis and the fact that $\D_{\ba}(w)/\D_{\ba}(u) = \D_{\ba}(w')/\D_{\ba}(u')$. Assume now that there exist distinct $i,j\in A$ such that $w\in A_{i}^{n+1}$ and $u\in A_{j}^{n+1}$. Without loss of generality, we may assume that $i=1$. By Lemma \ref{lem: preimage 12^00}(i), 
\[\T^{m,\ba}_w\cap \T^{m,\ba}_u \subset \T^{m,\ba}_1\cap \T^{m,\ba}_j =\{[12^{(\infty)}]\}. \]
Since $[12^{(\infty)}]=[j1^{(\infty)}]$ (by Remark \ref{rem:equivalent classes of 12^oo}) and $w\in A_1^{n+1}$, we have that $w=12^{(n)}$ and $u=j1^{(n)}$ and
\[ \frac{\D_{\ba}(w)}{\D_{\ba}(u)} = \frac{2^{-n-1}}{\ba(j)2^{-n}} = \frac{1}{2\ba(j)}\]
and \eqref{eq:neighbortiles} follows. We work similarly if $j=1$, or if neither of $i,j$ is equal to 1.
\end{proof}

\section{Uniform 4-branching of the Vicsek fractal}\label{sec: Vicsek unif branching}
Recall that the Vicsek fractal $\mathbb{V}$ is defined as the attractor of the iterated function system $\mathcal{F} = \{\phi_1,\dots,\phi_5\}$ on $\mathbb{C}$ where
\[\phi_1(z) = \tfrac13(z-2-2i), \quad \phi_2(z) = \tfrac13(z-2+2i), \quad \phi_3(z) = \tfrac13(z+2+2i),\]
\[\phi_4(z) = \tfrac13(z+2-2i), \quad \phi_5(z) = \tfrac13z.\]
The purpose of this section is to show the following proposition for $\V$.
	
\begin{prop}\label{prop: Vicsek universal}
The Vicsek fractal $\V$ is a uniformly $4$-branching quasiconformal tree.  
\end{prop}

Fix for the rest of this section $A=\{1,\dots,5\}$. Given $w=i_1\cdots i_k \in A^{k}$ define
\[ \phi_{w} = \phi_{i_1}\circ\cdots \circ \phi_{i_k}, \quad\text{and}\quad \mathbb{V}_w = \phi_w(\mathbb{V}).\]
	
In \cite[\textsection 6.2]{DV}, it was shown that $\mathbb{V}$ is a quasiconformal tree. Moreover, there exist an equivalence relation $\sim$ on $A^{\N}$, and  a bi-Lipschitz homeomorphism 
\[ f : (A^{\N}/\sim,d) \to \mathbb{V}, \quad f(A^{\N}_w/\sim) =\mathbb{V}_w\text{ for all $w\in A^*$}.\]
Henceforth, we identify points in $A^{\N}/\sim$ with their images in $\mathbb{V}$ under $f$. Also, the equivalence relation in $A^{\N}$ implies that
\[ [13^{(\infty)}]=[51^{(\infty)}], \quad [24^{(\infty)}]=[52^{(\infty)}], \quad [35^{(\infty)}]=[53^{(\infty)}], \quad [42^{(\infty)}]=[45^{(\infty)}].\]
	
The proof of the following is almost identical with that  in \cite[Proposition 4.2]{BT_CSST}, hence we omit the details. 
	
\begin{lem}
For each $n\in\N$, let $J_n=\bigcup_{w\in A^n}\phi_w(X)$ and $K_n=\bigcup_{w\in A^n}\phi_w([-1,1]^2)$, where $X\subset \R^2$ is the union of the diagonals of $[-1,1]^2$. Then, for all $n\in\N$, $J_{n}$ and $K_{n}$ are compact, and 
\[ J_{n} \subset J_{n+1}\subset \mathbb{V}\subset K_{n+1}\subset K_{n}.\] 
Moreover, $\overline{\bigcup_{n\in\N_0}J_n}=\mathbb{V}=\bigcap_{n\in\N_0}K_n$.
\end{lem}
	
The above lemma implies that $\diam \V=\sqrt{2}$. We also need to classify the branch points of $\V$, and the boundary points of the subtrees $\V_w$.
	
\begin{lem}\label{lem: valenve of branch points of Vicsek}
Let $w\in A^*$ .
\begin{enumerate}[(i)]
\item The set $\mathbb{V}\setminus \mathbb{V}_{w5}$ has exactly $4$ components, and
\[ (\overline{\mathbb{V}\setminus \mathbb{V}_{w5}}) \cap \mathbb{V}_{w5} = \{[w51^{(\infty)}],[w52^{(\infty)}],[w53^{(\infty)}],[w54^{(\infty)}]\}. \]
\item If $i\in \{1,2,3,4\}$, then $\mathbb{V}\setminus \mathbb{V}_{wi}$ has either $1$ or $2$ components. 
\begin{enumerate}
\item If $\mathbb{V}\setminus \mathbb{V}_{wi}$ has one component, then there exists $j\in\{1,2,3,4\}$ such that $(\overline{\mathbb{V}\setminus \mathbb{V}_{wi}}) \cap \mathbb{V}_{wi}=\{[wij^{(\infty)}]\}$.
\item If $\mathbb{V}\setminus \mathbb{V}_{wi}$ has two components, then $(\overline{\mathbb{V}\setminus \mathbb{V}_{wi}}) \cap \mathbb{V}_{wi}$ is either the set $\{[wi1^{(\infty)}],[wi3^{(\infty)}]\}$, or the set $\{[wi2^{(\infty)}],[wi4^{(\infty)}]\}$.
\end{enumerate}
\item If w is not ending with 5, then 
\begin{enumerate}
\item either the closures of the components of $\mathbb{V}\setminus \mathbb{V}_{w5}$ are $\mathbb{V}_{w1},\mathbb{V}_{w3}$ and $\V_{w2}$,$\V_{w4}$ lie in different connected components,
\item or the closure of the components of $\mathbb{V}\setminus \mathbb{V}_{w5}$ are  $\mathbb{V}_{w2},\mathbb{V}_{w4}$ and $\V_{w1},\V_{w3}$ lie in different connected components.
\end{enumerate}
\end{enumerate}
\end{lem}
	
\begin{proof}
For (i), note first that
\[ \overline{\mathbb{V}\setminus \mathbb{V}_{5}} = \overline{\mathbb{V}\setminus \phi_5([-1,1]^2)} = \mathbb{V}\cap\bigcup_{i=1}^4\phi_{i}([-1,1]^2) = \mathbb{V}_1\cup \cdots\cup \mathbb{V}_4.\] 
Moreover, if $i\in \{1,2,3,4\}$, (say $i=1$), then
\[ \{[13^{(\infty)}]\} = \{[51^{(\infty)}]\} \subset \mathbb{V}_1\cap \mathbb{V}_5 \subset \phi_1([-1,1]^2)\cap \phi_5([-1,1]^2),\]
where the latter set has exactly one point. Therefore,
\[ (\overline{\mathbb{V}\setminus \mathbb{V}_{5}}) \cap \mathbb{V}_{5} = \{[51^{(\infty)}],[52^{(\infty)}],[53^{(\infty)}],[54^{(\infty)}]\}. \]
		
Fix now $w\in A^*$. On the one hand,
\[ \mathbb{V}_w \setminus \mathbb{V}_{w5} = \phi_w(\mathbb{V} \setminus \mathbb{V}_5),\]
which has 4 components. On the other hand, $\mathbb{V}$ is a tree and $\mathbb{V}_w$ is a subtree, so $\mathbb{V} \setminus \mathbb{V}_{w5}$ has as many components as those of $\mathbb{V}_w \setminus \mathbb{V}_{w5}$. Moreover,
\[ (\overline{\mathbb{V}\setminus \mathbb{V}_{w5}})\cap \mathbb{V}_{w5} = (\overline{\mathbb{V}_w\setminus \mathbb{V}_{w5}})\cap \mathbb{V}_{w5} = \{[w51^{(\infty)}],[w52^{(\infty)}],[w53^{(\infty)}],[w54^{(\infty)}]\}.\]
		
The proof of (ii) is by induction on $|w|$. If $w=\epsilon$, and if $i\in \{1,2,3,4\}$, then
\[\overline{\mathbb{V}\setminus \mathbb{V}_i} = \mathbb{V}_5\cup \bigcup_{j\in\{1,2,3,4\}\setminus\{i\}}\mathbb{V}_j, \]
which is a connected set. For the second part of (ii), assume without loss of generality that $i=1$, and the proof is similar in other cases. Then,
\[ \{[13^{(\infty)}]\}=\{[51^{(\infty)}]\} \subset (\overline{\mathbb{V}\setminus \mathbb{V}_{1}})\cap \V_1 \subset \phi_1([-1,1]^2) \cap \phi_5([-1,1]^2).\]
with the latter set having only one element. Therefore, $(\overline{\mathbb{V}\setminus \mathbb{V}_{1}})\cap \V_1 = \{13^{(\infty)}\}$.
		
Suppose now that (ii) is true for some $w\in A^*$ and let $j\in A$. By the inductive hypothesis and by (i), there are three possible cases.
		
\emph{Case 1:} Suppose that $j=5$. Then, by (i),
\[ (\overline{\mathbb{V}\setminus \mathbb{V}_{w5}}) \cap \mathbb{V}_{w5} = \{[w51^{(\infty)}], [w52^{(\infty)}], [w53^{(\infty)}], [w54^{(\infty)}]\}.\]
Assume that $i=1$, and the case $i\in\{2,3,4\}$ is similar. Note that
\[ \mathbb{V}_{w51} \cap (\overline{\mathbb{V}_{w5} \setminus \mathbb{V}_{w51}}) = \{[w513^{(\infty)}]\},\quad\text{and}\quad \mathbb{V}_{w51} \cap (\overline{\mathbb{V} \setminus \mathbb{V}_{w5}}) = \{[w51^{(\infty)}]\}.\]
Therefore, $\mathbb{V}\setminus \mathbb{V}_{w51}$ has $2$ components. 
		
\emph{Case 2:} Suppose that $j\in\{1,2,3,4\}$ and $\mathbb{V} \setminus \mathbb{V}_{wj}$ has $2$ components. By the inductive hypothesis, either
\begin{equation}\label{eq: Vic branch case 1}
\overline{\V\setminus\V_{wj}}\cap\V_{wj}=\{[wj1^{(\infty)}],[wj3^{(\infty)}]\},
\end{equation}
or
\begin{equation}\label{eq: Vic branch case 2}
\overline{\V\setminus\V_{wj}}\cap\V_{wj}= \{[wj2^{(\infty)}],[wj4^{(\infty)}]\}.
\end{equation}
Assume that $i=1$ (the case $i\in\{2,3,4\}$ is similar). In the  case of \eqref{eq: Vic branch case 1}, we have
\[ \mathbb{V}_{wj1} \cap (\overline{\mathbb{V} \setminus \mathbb{V}_{wj}}) = \{[wj1^{(\infty)}]\},\]
and $\mathbb{V}\setminus \mathbb{V}_{wj1}$ has $2$ components. In the case of \eqref{eq: Vic branch case 2}, we have $\mathbb{V}_{wj1} \cap (\overline{\mathbb{V} \setminus \mathbb{V}_{wj}}) =\emptyset$, which implies
\[ \mathbb{V}_{wj1} \cap (\overline{\mathbb{V}_{wj} \setminus \mathbb{V}_{wj1}}) = \{[wj13^{(\infty)}]\},\]
and $\mathbb{V}\setminus \mathbb{V}_{wj1}$ has $1$ component.
		
\emph{Case 3:} Suppose that $j\in\{1,2,3,4\}$ and $\mathbb{V} \setminus \mathbb{V}_{wj}$ has $1$ component. By the inductive hypothesis, there exists $k\in\{1,2,3,4\}$ such that $\overline{\V\setminus\V_{wj}}\cap\V_{wj}=\{[wk1^{(\infty)}]\}$.
		
Assume that $i=1$ (the case $i\in\{2,3,4\}$ is similar). If $k=1$, then
\[\mathbb{V}_{wj1} \cap (\overline{\mathbb{V} \setminus \mathbb{V}_{wj}})=\{[wj1^{(\infty)}]\},\] 
while if $k=2,3,4$, then
\[\mathbb{V}_{w'1} \cap (\overline{\mathbb{V} \setminus \mathbb{V}_{w'}})=\emptyset, \]
and as before we have
\[ \mathbb{V}_{wj1} \cap (\overline{\mathbb{V}_{wj} \setminus \mathbb{V}_{wj1}}) = \{[wj13^{(\infty)}]\}.\]
Hence, $\mathbb{V}\setminus \mathbb{V}_{wj1}$ has $2$ components when $k=1$, and $1$ component when $k=2,3,4$.
		
We now show (iii). Suppose that $w=ui$, with $u\in A^*$ and $i\in\{1,2,3,4\}$. By (ii), $\V\setminus\V_w$ has either $1$ or $2$ components, and there are two cases to consider. 
		
\emph{Case 1: } Suppose that $\V\setminus\V_w$ has 1 component. Let $U = \overline{\V\setminus\V_w}$, and let $B_1,\dots,B_4$ be the closures of the components of $\V\setminus\V_{w5}$. On the one hand, since $\overline{\V_w\setminus\V_{w5}}=\bigcup_{i=1}^4\V_{wi}$, we may assume that $\V_{wi}\subset B_i$ for all $i$, potentially by relabeling if necessary. 
		
On the other hand, there exists $i\in\{1,2,3,4\}$ such that $\overline{\V\setminus\V_w}\cap\V_w=\{wi^{(\infty)}\}$. Since $\overline{\V\setminus\V_w}\subset\overline{\V\setminus\V_{w5}}$, we have that $U\cup\V_{wi}\subset B_i$. Moreover, note that 
\[\overline{\V_w\setminus\V_{w5}} = \overline{(\V\setminus\V_{w5})\setminus(\V\setminus\V_w)},\]
which implies that 
\[\V_{w1}\cup\V_{w2}\cup\V_{w3}\cup\V_{w4}=\overline{B_i\setminus U}\cup \bigcup_{l\in \{1,2,3,4\}\setminus \{i\}}B_l.\]
The fact that $\V_{wi}\subset B_i$ and that $\V_{wl}$, $B_l$ are connected components, imply that, after relabeling, $B_i=\V_{wi}\cup U$ and $\V_{wl}=B_l$, for $l\in\{1,2,3,4\}\setminus\{i\}$.
		
\emph{Case 2: } Suppose that $\V\setminus\V_w$ has $2$ components. Denote these components by $U_1$, $U_2$. Denote again by $B_i$, $i\in\{1,2,3,4\}$, the closures of the components of $\V\setminus\V_{w5}$, and assume without loss of generality that $\overline{\V\setminus\V_w}\cap\V_w=\{[w1^{(\infty)}],[w3^{(\infty)}]\}$. We may further assume that $[w1^{(\infty)}]\in U_1$ and that $[w3^{(\infty)}]\in U_2$. Working as in Case 1, it follows that $U_1\cup\V_{w1}\subset B_1$ and  $U_2\cup\V_{w3}\subset B_2$. Furthermore, note again that
\[\overline{\V_w\setminus\V_{w5}} = \overline{(\V\setminus\V_{w5})\setminus(\V\setminus\V_w)},\]
which implies that 
\[\V_{w1}\cup\V_{w2}\cup\V_{w3}\cup\V_{w4}=\overline{B_1\setminus U_1}\cup \overline{B_2\setminus U_2}\cup B_3\cup B_4.\]
Since $\V_{w1}\subset B_1$, $\V_{w3}\subset B_2$, and all $\V_{wi}$, $B_i$ are connected components, after relabeling, we may assume that $B_1=\V_{w1}\cup U_1$, $B_2=\V_{w3}\cup U_2$ and $B_3=\V_{w2}$, $B_4=\V_{w4}$.
\end{proof}
	
The next lemma provides us with a complete classification of the branch points of $\V$ and their heights.
	
\begin{lem}\label{lem:branch pts of Vicsek}
A point $[w]$ is a branch point of $\V$ if, and only if, $w=u5^{(\infty)}$ for some $u\in A^*$. Moreover, every branch point $[u5^{(\infty)}]$ has valence $4$, and 
$$H_{\V}([u5^{(\infty)}])=\diam B_{\V}^4([u5^{(\infty)}])=\tfrac1{\sqrt{2}}3^{-|u|}.$$
\end{lem}
	
\begin{proof}
Assume first that $w=u5^{(\infty)}$. We may assume that $u$ is not ending with 5. Note that $[w]$ is the unique element in $\bigcap_{n=1}^{\infty}\V_{u5^{(n)}}$. Therefore, $\V\setminus\{[w]\}=\bigcup_{n=1}^{\infty}(\overline{\V\setminus\V_{u5^{(n)}}})$. For each $n\in\N$, denote by $B_i^n$, $i\in\{1,2,3,4\}$, the closures of the components of $\V\setminus\V_{u5^{(n)}}$ given by Lemma \ref{lem: valenve of branch points of Vicsek}(i). Note that $(\overline{\V\setminus\V_{u5^{(n)}}})_{n\in \N}$ is an increasing sequence, so we may assume that $B_i^{n+1}\subset B_i^{n}$. It follows that 
\[\V\setminus\{[w]\}=\bigcup_{i\in\{1,2,3,4\}}\bigcup_{n=1}^{\infty}B_i^n.\]
Since the sets $\bigcup_{n=1}^{\infty}B_i^n$ are connected for each $i\in\{1,2,3,4\}$ and disjoint for $i\neq j$, they are the components of $\V\setminus [w]$. Thus, $[w]$ is a branch point of valence $4$.
		
For the converse, let $w=u_1i_1u_2i_2\dots \in A^{\N}$ where $u_k\in A^*$ and $i_k\in \{1,2,3,4\}$ for $k\in \N$. Set $w_n=u_1i_1\dots u_ni_n$. By Lemma \ref{lem: valenve of branch points of Vicsek}(ii), $\V\setminus\V_{w_n}$ has either $1$ or $2$ components. We consider two cases.
		
\emph{Case 1.} Suppose that for all $n\in\N$, $\V\setminus\V_{w_n}$ has $1$ component denoted by $K_n$. Since $\V_{w_{n+1}}\subset\V_{w_n}$ it follows that $K_n\subset K_{n+1}$. Note, also, that $\{[w]\}=\bigcap_{n=1}^{\infty}\V_{w_n}$. Thus, $\V\setminus\{[w]\}=\bigcup_{n=1}^{\infty}K_n$, and the latter set is connected, hence is a leaf.
		
\emph{Case 2.} Suppose that there exists $n_0\in\N$ such that $\V\setminus\V_{w_{n_0}}$ has $2$ components, denoted by  $K_{n_0}^1,K_{n_0}^2$. We claim that for all $n\geq n_0$, $\V\setminus\V_{w_n}$ has $2$ components $K_n^1$, $K_n^2$, such that $K_{n_0}^1\subset K_n^1$ and $K_{n_0}^2\subset K_n^2$. Assume towards contradiction that there exists $m\geq n_0$ such that the claim doesn't hold, and let $U_m=\overline{\V\setminus\V_{w_m}}$. By Lemma \ref{lem: valenve of branch points of Vicsek}(ii), $U_m$ has $1$ or $2$ components. If it has $1$ component, then $K_{n_0}^1\cup K_{n_0}^2\subset U_m$, since $(\overline{\V\setminus\V_{w_n}})_{n\in\N}$ is an increasing sequence. This is a contradiction, since $U_m$ is a connected set, and $K_{n_0}^1$, $K_{n_0}^2$ are disjoint. If $U_m$ has $2$ components, say $U_m^1,U_m^2$, then, similarly, $K_{n_0}^1\cup K_{n_0}^2$ would have to lie in only one of them. But then, $U_m$ would have at least three components, which is  contradiction. Moreover, let $U_k=\overline{\V\setminus\V_{w_k}}$ for $k\in\{1,\dots, n_0-1\}$. Note that $(U_k)_{k\in \N}$ is an increasing sequence of  connected sets, since $n_0$ is minimal. Hence the set $\bigcup_{k=1}^{n_0-1}U_k$ lies in either $K_{n_0}^1$ or $K_{n_0}^2$. It follows that 
$$\V\setminus \{[w]\}=\bigcup_{n=n_0}^{\infty}K_n^1\cup\bigcup_{n=n_0}^{\infty}K_n^2.$$ 
Since $K^i:=\bigcup_{n=n_0}^{\infty}K_n^i$, $i\in\{1,2\}$, are connected and disjoint, it follows that $K^1,K^2$ are  all the components of $\V\setminus\{[w]\}$, which implies that $[w]$ is a double point. 
		
For the last claim, we show that two of the components (in fact the ones with the smallest diameter) are $\bigcup_{n=0}^{\infty}\V_{u5^{(n)}i}$ and $\bigcup_{n=0}^{\infty}\V_{u5^{(n)}j} $, where $(i,j)\in\{(1,3),(2,4)\}$. By Lemma \ref{lem: valenve of branch points of Vicsek}(iii), one of the components of $\overline{\V\setminus\V_{w5}}$ is $\V_{wi}$ for $i\in\{1,2,3,4\}$. Suppose $\V_{wi}=\V_{w1}$, and other cases are similar. Let $B$ be the component of $\overline{\V\setminus\V_{w55}}$ for which $\V_{w1}\subset B$. Moreover, the fact that the set $\V_{w51}$ is a component of $\overline{\V_{w5}\setminus\V_{w55}}$, along with $\V_{w1}\cap\V_{w51}=\{[w13^{(\infty)}]\}$, imply that $\V_{w1}\cup\V_{w51}\subset B$. On the other hand, note that 
\[\overline{\V_{w5}\setminus\V_{w55}}=\overline{(\V\setminus\V_{w55})\setminus(\V\setminus\V_{w5})},\]
and that one of the components of the latter set is $\overline{B\setminus\V_{w1}}$. Since $\V_{w51}\subset B$, it follows that $B=\V_{w1}\cup\V_{w51}$. The rest of the claim follows similarly by applying Lemma \ref{lem: valenve of branch points of Vicsek}(i) and the fact that for each $n\in\N$, $\V_{u5^{(n)}}\setminus\V_{u5^{(n+1)}}=\bigcup_{i\in\{1,2,3,4\}}\V_{u5^{(n)}i}$. Note that by Lemma \ref{lem: valenve of branch points of Vicsek}(iii), it can be similarly shown that the rest of the components have larger diameter. Hence, 
\[H_{\V}(p)=\diam B_{\V}^4(p)=\sqrt{2}\sum_{n=0}^{\infty}3^{-|u|-1-n}= 3^{-|u|-1}\sqrt{2}\sum_{n=0}^{\infty}3^{-n}= \tfrac1{\sqrt{2}}3^{-|u|}. \qedhere\]
\end{proof}
	
\begin{proof}[Proof of Proposition \ref{prop: Vicsek universal}]
By Lemma \ref{lem:branch pts of Vicsek}, it follows that every branch point of $\V$ has valence $4$, and that $\V$ has uniform branch growth. It remains to show uniform branch separation and uniform density. 
		
To show uniform branch separation, fix $v,w\in A^*$ with $v\neq w$, and let $[v5^{(\infty)}]$, $[w5^{(\infty)}]$ be two branch points of $\V$. We may assume that $v,w$ are not ending with 5. Let $u\in A^*$ be the longest common initial word of $v$ and $w$. Then $v=uv'$ and $w=uw'$ with $w',v'\in A^*$ and $w'(1)\neq v'(1)$. Without loss of generality, assume that $|w'|\geq |v'|$, and that $w'\neq\epsilon$. Since $\phi_{w'}([5^{(\infty)}])\in\phi_{w'}([-1,1]^2)$ and $\phi_{v'}([5^{\infty)}])\in\phi_{v'}([-1,1]^2)$, an elementary geometric estimate shows that 
\[ |\phi_{v'}([5^{(\infty)}])-\phi_{w'}([5^{(\infty)}])|\geq \tfrac123^{-|w'|}.\] 
Hence, we have
\begin{align*}
|\phi_v([5^{(\infty)}]) - \phi_w([5^{(\infty)}])| &= |\phi_u(\phi_{v'}([5^{(\infty)}]) - \phi_u(\phi_{w'}([5^{(\infty)}]))| \\
&= 3^{-|u|} |\phi_{v'}([5^{(\infty)}]) - \phi_{w'}([5^{(\infty)}])| \\
&\geq \tfrac12 3^{-|u|}3^{-|w'|}\\
&\gtrsim \min \{ H_{\V}([v5^{(\infty)}]), H_{\V}([w5^{(\infty)}]) \}.
\end{align*}
		
To show uniform density, let $[w],[v]\in \mathbb{V}$ with $[w]\neq [v]$. Let $u\in A^*$ be the longest common initial word of $v$ and $w$, that is, $w=uw'$ and $v=uv'$ with $w',v'\in A^{\N}$ and $w'(1)\neq v'(1)$. Hence, $[w]\in \mathbb{V}_{uw'(1)}$ and $[v]\in \mathbb{V}_{uv'(1)}$. Note that for each $n\in\N$, the unique combinatorial arc that connects $uw'(n)$ with $uv'(n)$ passes through $u5^{(n)}$. Therefore, by \cite[Claim 3.19]{DV}, the unique arc in $\V$ that joins $[w]$ with $[v]$ contains the branch point $p=[u5^{(\infty)}]$, which satisfies
\begin{align*}
H_{\mathbb{V}}(p) =\tfrac1{\sqrt{2}}3^{-|u|}\gtrsim 3^{-|u|}|[w']-[v']|=|[w]-[v]|.
\end{align*}
This completes the proof of the proposition.
\end{proof}

\section{Ahlfors regularity and dimensions}\label{sec:dimension}
	
In this section we study the dimensions of trees $\T^{m,\ba}$ where $m\in\{2,3,\dots\}\cup\{\infty\}$. Recall that a metric space $X$ is \emph{$s$-Ahlfors regular} if there exists a measure $\mu$ on $X$ and a constant $C>1$ such that 
\[ C^{-1}r^s \leq \mu(B(x,r)) \leq C r^s, \qquad \text{for all $x\in X$ and $r\in (0,\diam{X})$.}\]
	
\begin{prop}\label{prop:doubling}
Let $\ba$ be a weight. For each $m\in\{2,3,\dots\}$ define $\Psi_{m,\ba}(t) = \ba(1)^t + \cdots + \ba(m)^t$. 
\begin{enumerate}[(i)]
\item If $m\in \{2,3,\dots\}$, then $\mathbb{T}^{m,\ba}$ is Ahlfors $s$-regular where $s$ is the unique solution of the equation $\Psi_{m,\ba}(s) = 1$.
\item Suppose that there exists $s>0$ such that $\lim_{m\to\infty}\Psi_{m,\ba}(s) < 1$. Then the Hausdorff dimension of $\T^{\infty,\ba}$ is at most $s$.
\end{enumerate} 
\end{prop}
	
We start with the second claim of the proposition.
	
\begin{proof}[{Proof of Proposition \ref{prop:doubling}(ii)}]
Fix $\d>0$. By Lemma \ref{lem:isom-embed} and our assumption, there exist integers $m,n\geq 2$ such that $2^{-n}<\delta/4$, $\dist_H(\T^{m,\ba},\T^{\infty,\ba}) < \d/2$, and $\Psi_{m,\ba}(s)<1$. Fix for each $w\in\{1,\dots,m\}^n$ a point $x_w\in \T^{m,\ba}$. Since $\T^{m,\ba}_w \subset B(x_w,2\D_{\ba}(w))$ we have that
\[ \{ B(x_w,2\D_{\ba}(w)) : w\in\{1,\dots,m\}^n \}\]
is an open cover of $\T^{\infty,\ba}$ by balls with diameters $4\D_{\ba}(w) \leq 4(1/2)^n < \delta$. Therefore,
\begin{align*}
\mathcal{H}^s_{\delta}(\T^{\infty,\ba}) \leq \sum_{w\in \{1,\dots,m\}^n} (4\D_{\ba}(w))^s&= 4^s\sum_{i_1,\dots,i_n \in \{1,\dots,m\}}\ba(i_1)^s\cdots\ba(i_n)^s \\
&=4^s \left( \Psi_{m,\ba}(s) \right)^n\\ 
&< 4^s.
\end{align*}
Thus, as $\delta\to 0$, we get that $\mathcal{H}^s(\T^{\infty,\ba})\leq 4^s < \infty$, so the Hausdorff dimension of $\T^{\infty,\ba}$ is at most $s$.
\end{proof}
	
We now turn to the proof of Proposition \ref{prop:doubling}(i). Fix for the rest of this section an integer $m\in \{2,3,\dots\}$ and set $A=\{1,\dots,m\}$. Let $s$ be the unique solution of the equation $\Psi_{m,\ba}(t) = 1$. Then for each $n\in \N$
\[ \sum_{w\in A^n}\D_{\ba}(w)^s = \sum_{i_1,\dots,i_n \in \{1,\dots,m\}}\ba(i_1)^s\cdots\ba(i_n)^s = \left( \Psi_{m,\ba}(s)\right)^n =1.\]
	
Denote by $\Sigma$ the $\sigma$-algebra generated by the cylinders $A^{\N}_{w}$ where $w\in A^{*}$. There exists a unique probability measure $\mu:\Sigma \to [0,1]$ such that for all $w\in A^{*}$,
\[ \mu(A^{\N}_{w})= \D_{\ba}(w)^s, \]
(e.g. see \cite[\textsection 3.1]{Str93}). Define also the projection map $\pi : A^{\N} \to \T^{m,\ba}$ given by $\pi(w) = [w]$.
	
\begin{lem}\label{lem: non-injective words}
The $\sigma$-algebra generated by the sets $\T^{m,\ba}_w$ is the same as the Borel $\sigma$-algebra on $\T^{m,\ba}$.
\end{lem}
	
\begin{proof}
Let $\mathcal{A}$ denote the $\sigma$-algebra generated by the sets $\T^{m,\ba}_w$, and let $\mathcal{B}(\T^{m,\ba})$ be the Borel $\sigma$-algebra on $\T^{m,\ba}$. That $\mathcal{A}\subset \mathcal{B}(\T^{m,\ba})$ is clear, as each $\T^{m,\ba}_w$ is a Borel set. For the reverse inclusion, fix $x\in \T^{m,\ba}$ and $r>0$. Let $W(x,r)=\{w\in A^*: \T^{m,\ba}_w\subset B(x,r)\}$ and note that $W(x,r)$ is a countable set since $A^*$ is countable, and so $\bigcup_{w\in W(x,r)} \T^{m,\ba}_w\in \mathcal{A}$. Furthermore for each point $y\in B(x,r)\cap \T^{m,\ba}$ there is a word $w\in A^\N$ with $[w]=y$, and since $B(x,r)\cap \T^{m,\ba}$ is an open set and $\lim_{n\to\infty}\diam(\T^{m,\ba}_{w(n)})=0$ there is some $n\in \N$ large enough that $w(n)\in W(x,r)$. Therefore, $\bigcup_{w\in W(x,r)} \T^{m,\ba}_w=B(x,r)\cap \T^{m,\ba}$.
\end{proof}
	
We say a word $w\in A^\N$ is \emph{non-injective} if there exists $u\in A^{\N}\setminus \{w\}$ such that $[w]= [u]$. 
	
\begin{lem}\label{lem:injfull}
The set $\mathcal{N}$ of non-injective words is a $\mu$-null set.
\end{lem}
	
\begin{proof}
Suppose that $w$ is a non-injective word. Then there exists $u\in A^{\N}\setminus \{w\}$ such that $[w]=[u]$. There exist $v\in A^*$ and $j,j' \in A$ such that $j\neq j'$, $w\in A^{\N}_{vj}$, and $u\in A^{\N}_{vj'}$. Since $[w]=[u]$, it follows that $[w] \in \T^{m,\ba}_{vj}\cap \T^{m,\ba}_{vj'} \neq \emptyset$. However, by Lemma \ref{lem: preimage 12^00}(i) we have that $[w]=[v12^{(\infty)}]$ and by Lemma \ref{lem: preimage 12^00}(ii) we have that $w\in \{v12^{(\infty)}, v21^{(\infty)}, \dots, vm1^{(\infty)}\}$. Therefore, 
\[\mathcal{N} \subset \{v12^{(\infty)} : v\in A^*\}\cup\{vj1^{(\infty)} : v\in A^*, \, j\in A\} \]
which implies that $\mathcal{N}$ is countable. Therefore, $\mu(\mathcal{N})=0$.
\end{proof}
	
We now prove the first claim of Proposition \ref{prop:doubling}.
	
\begin{proof}[{Proof of Proposition \ref{prop:doubling}(i)}]
Set $c=\min_{i=1,\dots,m}\ba(i)$. We claim that there exists $C>1$ depending only on $m$, $c$, $s$ such that the pushforward measure $\pi\#\mu$ defined on $\mathcal{B}(\T^{m,\ba})$ by Lemma \ref{lem: non-injective words} satisfies
\begin{equation}\label{eq:reg}
C^{-1}r^{s} \leq \pi\#\mu(B(x,r)) \leq C r^{s}, \qquad \text{for all $x\in \T^{m,\ba}$ and $r\in (0,2c)$.}
\end{equation}
		
First, for any $w\in A^*$, we have $A^\N_{w} \subset \pi^{-1}(\T^{m,\ba}_w)\subset\mathcal{N}\cup A^\N_{w}$, where $\mathcal{N}$ is the set of non-injective words in $A^{\N}$. Hence, by Lemma \ref{lem:injfull} for any $w\in A^*$,
\[ \pi\#\mu(\T^{m,\ba}_w)=\D_{\ba}(w)^{s}.\]
Fix for the rest of the proof a point $x\in \T^{m,\ba}$, a radius $r\in (0,2c)$, and a word $w\in A^\N$ such that $[w]=x$. 
		
For the upper bound of \eqref{eq:reg}, let $n\in\N$ such that 
\[\D_{\ba}(w(n+1)) \leq (2c)^{-1}r < \D_{\ba}(w(n)).\]
By Lemma \ref{lem: preimage 12^00}(ii) and Lemma \ref{lem: bdry Tu}, if $w_1,\dots,w_k$ are exactly the words in $A^n\setminus \{w(n)\}$ such that $\T^{m,\ba}_{w}\cap \T^{m,\ba}_{w_i}\neq \emptyset$, then $k\leq 2m$. Moreover, by \eqref{eq:neighbortiles}, for each $i=1,\dots,k$,
\[ r < 2c\diam{\T^{m,\ba}_{w(n)}} \leq \diam{\T^{m,\ba}_{w_i}} \leq (2c)^{-1} \diam{\T^{m,\ba}_{w(n)}}.\]
Therefore, since $r> (2c)\D_{\ba}(w(n+1)) > (2c^2)\D_{\ba}(w(n))$,
\begin{align*}
\pi\#\mu(B(x,r)) \leq \pi\#\mu(\T^{m,\ba}_{w(n)}) + \sum_{i=1}^k\pi\#\mu(\T^{m,\ba}_{w_i})&\leq (1+2m(2c)^{-s})\D_{\ba}(w(n))^s\\
&\leq (1+2m(2c)^{-s})(2c^2)^{-s}r^s.
\end{align*}
		
For the lower bound of \eqref{eq:reg}, let $n\in\N$ such that 
\[\D_{\ba}(w(n+1)) < r \leq \D_{\ba}(w(n)).\]
Then,
\[\pi\#\mu(B(x,r)) \geq \pi\#\mu(\T^{m,\ba}_{w(n+1)}) = \D_{\ba}(w(n+1))^{s} \geq c^{s}r^s .\qedhere\]
\end{proof}
	
\section{Quasisymmetric uniformization of uniformly $m$-branching trees}\label{sec:unif}
	
In this section we prove the following quantitative version of Theorem \ref{thm:uniformization}. 
	
\begin{thm}\label{thm:uniformization2}
Let $m\in\{2,3,\dots\}$ and let $\ba$ be a weight. A metric space $T$ is a uniformly $m$-branching quasiconformal tree if and only if it is quasisymmetrically equivalent to $\mathbb{T}^{m,\ba}$.
\end{thm}
	
By Proposition \ref{prop:geodesic} and Proposition \ref{prop:valence}, $\mathbb{T}^{2,\ba}$ is a geodesic metric arc, hence isometric to the (Euclidean) unit interval $[0,1]$. Therefore, the case $m=2$ in Theorem \ref{thm:uniformization2} follows by the quasisymmetric uniformization of quasi-arcs by Tukia and V\"ais\"al\"a \cite{TV80}. Thus, we may assume for the rest of this section that $m\geq 3$. 
	
One direction of Theorem \ref{thm:uniformization2} is simple and we record it as a lemma.
	
\begin{lem}
Suppose that $T$ is quasisymmetrically equivalent to $\mathbb{T}^{m,\ba}$ for some $m\in\{3,4,\dots\}$ and some weight $\ba$. Then, $T$ is a uniformly $m$-branching quasiconformal tree.
\end{lem}
	
\begin{proof}
By Lemma \ref{lem:metrictree} and Proposition \ref{prop:valence}, $\mathbb{T}^{m,\ba}$ is an $m$-valent metric tree. Since $T$ is homeomorphic to $\mathbb{T}^{m,\ba}$, it follows that $T$ is an $m$-valent metric tree as well.
		
Moreover, by Lemma \ref{lem:metrictree} and Proposition \ref{prop:doubling}, $\mathbb{T}^{m,\ba}$ is doubling and bounded turning, and since both these properties are quasisymmetrically invariant, it follows that $T$ is bounded turning and doubling, hence a quasiconformal tree.
		
By Proposition \ref{prop:valence}, $\mathbb{T}^{m,\ba}$ is uniformly $m$-branching. By \cite[Lemma 4.3]{BM22}, $T$ has uniformly relatively separated branches and by \cite[Lemma 4.5]{BM22} $T$ has uniformly dense branch points. To show that $T$ is uniformly $m$-branching is to show that $T$ has uniform branch growth, which follows by Lemma \ref{lem:qs maps preserve comparability of branches}.
\end{proof}
	
Note that it is enough to show Theorem \ref{thm:uniformization2} for a fixed weight $\ba$ since $\T^{m,\ba}$ is uniformly $m$-branching for all choices of weights $\ba$ by Proposition \ref{prop:valence}.
	
For the rest of Section \ref{sec:unif}, we fix $m\in \{3,4,\dots\}$, we fix a uniformly $m$-branching tree $T$, we set $A=\{1,\dots,m\}$, and we fix a weight $\ba$ such that $\ba(i)=1/2$ for all $i\in A$. By rescaling the metric on $T$ if necessary, we may assume that $\diam(T)=1$ for convenience. We show below that $T$ is quasisymmetrically equivalent to $\T^{m,\ba}$. The proof follows closely the proof of \cite[Theorem 1.4]{BM22}, which is the special case $m=3$. Since the proof and techniques are almost identical, we mostly sketch the steps.
	
\subsection{Quasi-visual subdivision of $T$}\label{sec:qvs of T}
By \cite[Section 3]{BM22} a finite set $\textbf{V}\subset T$ that does not contain leaves of $T$ decomposes $T$ into a set of tiles $\textbf{X}$ (in the sense of Definition \ref{def: qv approx}). These tiles are subtrees of $T$. We want to map them to tiles of $\mathbb{T}^{m,\ba}$, i.e., sets of the form $\mathbb{T}^{m,\ba}_w$, $w\in A^*$. By Lemma \ref{lem: bdry Tu} each tile $\mathbb{T}^{m,\ba}_w$ with $|w|>0$ has one or two boundary points. For this reason, we are interested in decompositions such that every tile $X\in \textbf{X}$ has one or two boundary points. 
	
We call $\textbf{X}$ an \textit{edge-like decomposition} of $T$ if $\textbf{V}=\emptyset$ and $\textbf{X}=\{T\}$ (as a degenerate case), or if $\textbf{V}\neq \emptyset$ and $\card(\partial X)\leq 2$ for each $X\in \textbf{X}$. Note that in the latter case $1\leq \card(\partial X)\leq 2$ by \cite[Lemma 3.3(iii)]{BM22}. We say that a tile $X\in\textbf{X}$ in an edge-like decomposition $\textbf{X}$ of $T$ is a \textit{leaf-tile} if $\card(\partial X)=1$ and an \textit{edge-tile} if $\card(\partial X)=2$.
	
We now set $\textbf{V}^{0}=\emptyset$, fix $\delta\in (0,1)$, and for $n\in\N$ define
\begin{equation}\label{eq:V^n}
\textbf{V}^n=\{p\in T:p \ \text{is a branch point of}\ T\ \text{with}\ H_T(p)\geq \delta^n\},
\end{equation}
where the height $H_T$ is as defined in \eqref{eq:height}. Note that $(\textbf{V}^n)_{n\in \N_0}$ is an increasing nested sequence and none of the sets $\textbf{V}^n$ contains a leaf of $T$. Denote by $\textbf{X}^n$ the set of tiles in the decomposition of $T$ induced by $\textbf{V}^n$. Then the sequence $(\textbf{X}^n)$ forms a subdivision of $T$ in the sense of Definition \ref{Subdivison} (see also the discussion in \cite{BM22} after (3.3)). The next proposition is a direct analogue of \cite[Proposition 6.1]{BM22}.
	
\begin{prop}\label{prop: quasi-visual subdivision}
Let $\delta\in(0,1)$, and $(\textbf{V}^n)_{n\in \N_0}$ be as in \eqref{eq:V^n}. Then the following statements are true:
\begin{enumerate}[(i)]
\item[(i)] $\textbf{V}^n$ is a finite set for each integer $n\geq 0$.
\item[(ii)]\label{prop: qs sbd 2} $(\textbf{X}^n)_{n\in\N_0}$ is a quasi-visual subdivision of $T$.
\end{enumerate}
Let $n\geq 0$, $X\in\textbf{X}^n$, and $\textbf{V}_X^{n+1}:=\textbf{V}^{n+1}\cap\interior(X)$. Then we have:
\begin{enumerate}[(i)]
\item[(iii)]\label{prop: qs sbd 3} $\card(\partial X)\leq 2$, and if we denote by $\textbf{X}_X^{n+1}$ the decomposition of $X$ induced by $\textbf{V}_X^{n+1}$, then $\textbf{X}_X^{n+1}$ is edge-like.
\item[(iv)] There exists $N\in\N$ depending on $\delta$, but independent of $n$ and $X$, such that $\card(\textbf{V}_X^{n+1})\leq N$. 
\end{enumerate}
If $\delta\in(0,1)$ is sufficiently small (independent of $n$ and $X$), then we also have:
\begin{enumerate}[(i)]
\item[(v)] $\card(\textbf{V}_X^{n+1})\geq 2$.
\item[(vi)] If $\card(\partial X)=2$ and $\partial X=\{u,v\}\subset T$, then $(u,v)\cap \textbf{V}_X^{n+1}$ contains at least three elements.
\end{enumerate}
\end{prop}
	
\begin{proof}
The proof of this proposition is almost identical to the proof of \cite[Proposition 6.1]{BM22} so we only sketch it. The specific value $m=3$ is only implicitly used in the proof of (ii). Hence, we only focus on that part, since (i) and (iii)-(vi) above follow in the exact same way as in the $m=3$ case. Note that (iii) is proven before (ii) in \cite[Proposition 6.1]{BM22}, so we may assume $\card(X)\leq 2$ for all $X\in\textbf{X}^n$.
		
To show that $\textbf{X}^n$ is a quasi-visual subdivision of $T$, by \cite[Lemma 2.4]{BM22}, it suffices to verify the following two conditions: 
\begin{enumerate}[(a)]
\item $\diam{X} \simeq \delta^n$ for all $n\geq 0$ and $X \in \textbf{X}^n$,
\item $\dist(X,Y) \gtrsim \d^n$ for all $n\geq 0$ and $X,Y\in \textbf{X}^n$ with $X\cap Y = \emptyset$.
\end{enumerate}
Relations $\diam(X)\lesssim \delta^n$ and (b) can be proved for $m\geq 3$ verbatim as in the case $m=3$ in \cite[Proposition 6.1(ii)]{BM22}. 
		
For the relation $\diam(X)\gtrsim \delta^n$, assume that $\partial X\neq\emptyset$, otherwise $X=T$ and so $\diam(X)=1\geq \delta^n$. If $\partial X$ consists of one point $p\in T$, then $p\in \partial X\subset\textbf{V}^n$ (due to \cite[Lemma 3.3(ii)]{BM22} and so $p$ is a branch point of $T$ with $H_T(p)\geq \delta^n$. In this case, $X$ is a branch of $p$ in $T$ (see \cite[Lemma 3.3(vi)]{BM22}). Hence, either $X=B_j$, $j\in \{1, 2\}$, which implies $\diam(X)\geq H_T(p)$, or
$$\diam(X)\simeq H_T(p)\geq \delta^n,$$
since $T$ has uniform branch growth. If $\partial X$ consists of two distinct points $u,v\in\textbf{V}^n$, then we have 
$$\diam(X)\geq |u-v|\gtrsim\min\{H_T(u),H_T(v)\}\geq \delta^n$$
where the second inequality follows from the fact that $T$ has uniformly relatively separated branch points.
\end{proof}
	
\subsection{Quasi-visual subdivision of $\T^{m,\ba}$}
	
The following proposition, which is the direct analogue of \cite[Proposition 7.1]{BM22}, gives us a quasi-visual subdivision of the metric tree $\mathbb{T}^{m,\ba}$ that is isomorphic to that of $T$ constructed in \textsection\ref{sec:qvs of T}.
	
\begin{prop}\label{prop:isomorphic subd}
Let $(\textbf{V}^n)_{n\in \N_0}$ be as in \eqref{eq:V^n}, with $\delta\in (0,1)$ small enough so that, for all $n\in \N_0$, all the statements in Proposition \ref{prop: quasi-visual subdivision} are true for the decomposition $\textbf{X}^n$ of $T$ induced by the set $\textbf{V}^n$. Then there exists a quasi-visual subdivision $(\textbf{Y}^n)_{n\in \N_0}$ of $\mathbb{T}^{m,\ba}$ that is isomorphic to the quasi-visual subdivision $(\textbf{X}^n)_{n\in \N_0}$ of $T$. 
\end{prop}
	
The proof of Proposition \ref{prop:isomorphic subd} follows that of \cite[Proposition 7.1]{BM22} very closely.
	
By Proposition \ref{prop: quasi-visual subdivision}(iii) the decompositions $\textbf{X}^n$ of $T$  are edge-like. The points in $\partial X$ are leaves of $X$ and the only points where $X$ intersects other tiles of the same level (see \cite[Lemma 3.3(iii), (iv)]{BM22}). We follow the definition of \textit{marked leaves} as in the proof of Lemma \ref{lem:homeo of leaves} with the difference that points in $\partial X$ do not carry a sign here. Moreover, $X\in\textbf{X}^n$ (viewed as a tree) has an edge-like decomposition $\textbf{X}_X^{n+1}$ induced by $\textbf{V}_X^{n+1}=\textbf{V}^{n+1}\cap\interior(X)$.
	
We want to find a decomposition of $\mathbb{T}^{m,\ba}$ into tiles $\T^{m,\ba}_w$ that is isomorphic to $\textbf{X}_X^{n+1}$ and respects the marked leaves. As in \cite{BM22}, we consider $[1^{(\infty)}]$ or $[2^{(\infty)}]$ as a marked leaf of $\mathbb{T}^{m,\ba}$ if $\card(\partial X)=1$, or both $[1^{(\infty)}]$ and $[2^{(\infty)}]$ as marked leaves of $\mathbb{T}^{m,\ba}$ if $\card(\partial X)=2$. Since $\mathbb{T}^{m,\ba}$ is uniformly $m$-branching, for fixed leaves $p_1, p_2\in T$, by Lemma \ref{lem:homeo of leaves} there exists a homeomorphism $f:T\to\mathbb{T}^{m,\ba}$ such that $f(p_1)=[1^{(\infty)}]$ and $f(p_2)=[2^{(\infty)}]$. 
	
If $w\in A^*$, then we say that the \textit{level} of the tile $\mathbb{T}_w^{m,\ba}$ (denoted by $\texttt{L}(\mathbb{T}_w^{m,\ba})$) is equal $|w|$. We also say that a homeomorphism $F:T\to\mathbb{T}^{m,\ba}$ is a \textit{tile-homeomorphism} (for $\textbf{X})$ if $F(X)$ is of the form $\T_w^{m,\ba}$ for each $X\in\textbf{X}$. Then $F$ maps tiles in $T$ to tiles $\mathbb{T}^{m,\ba}$ and so the level $\texttt{L}(F(X))$ (of $F(X)$ as a tile of $\mathbb{T}^{m,\ba}$) is defined for each $X\in\textbf{X}$. 
	
The following lemma is a useful characterization of decompositions of $\mathbb{T}^{m,\ba}$ of the form $\mathbb{T}_w^{m,\ba}$ where $w\in A^*$.
\begin{lem}\label{lem: level of tiles}
Let $V\subset\mathbb{T}^{m,\ba}$ be a finite set of branch points of $\mathbb{T}^{m,\ba}$, and let $\mathbb{X}$ be the set of tiles in the decomposition of $\mathbb{T}^{m,\ba}$ induced by $V$. If $\mathbb{X}$ is of the form $\mathbb{T}_w^{m,\ba}$ for $w\in A^*$, then the levels of tiles in $\mathbb{X}$ satisfy 
$$\texttt{L}(X)\leq \card(V)\ \text{for each}\ X\in\mathbb{X}.$$
If, in addition, $V\neq \emptyset$, then we have $\texttt{L}(X)\geq 1$ for each $X\in\mathbb{X}$.
\end{lem}
	
\begin{proof}
The proof is identical to that of \cite[Lemma 5.5]{BM22} for $m=3$. We solely point out that the point $0\in\mathbb{C}$ of the CSST corresponds in our setup to the point $[12^{(\infty)}]\in\mathbb{T}^{m,\ba}$.
\end{proof}
	
The following lemma is an analogue of \cite[Lemma 7.3]{BM22}.
	
\begin{lem}\label{lem:tile homeo}
Let $S$ be an $m$-valent metric tree whose branch points are dense in $S$. Suppose $\textbf{V}\subset S$ is a finite set of branch points of $S$ that induces an edge-like decomposition of $S$ into the set of tiles $\textbf{X}$.
		
Assume that either $S$ has no marked leaves, or $S$ has exactly one marked leaf $p$, or $S$ has exactly two marked leaves $p,q\in S$. If $\textbf{V}\neq \emptyset$, we also assume that the marked leaf $p$ (if it exists) lies in a leaf-tile $P\in\textbf{X}$, and the other marked leaf $q$ (if it exists) lies in a leaf-tile $Q\in\textbf{X}$ distinct from $P$.
		
Then there exists a tile-homeomorphism $F:S\to\mathbb{T}^{m,\ba}$ for $\textbf{X}$ such that the following statements are true.
\begin{enumerate}[(i)]
\item[(i)]\label{lem: tile homeo 1} $F(p)=[1^{(\infty)}]$, or alternatively, $F(p)=[2^{(\infty)}]$, if $S$ has one marked leaf $p$; or $F(p)=[1^{(\infty)}]$ and $F(q)=[2^{(\infty)}]$, if $S$ has two marked leaves $p$ and $q$.
\item[(ii)]\label{lem: tile homeo 2} If $S$ has one marked leaf $p$ and $\card(\textbf{V})\geq 2$, then we may also assume  that $F$ satisfies $\texttt{L}(F(P))=2$.
\item[(iii)]\label{lem: tile homeo 3} If $S$ has two marked leaves $p,q\in S$ and $[p,q]\cap \textbf{V}$ contains at least three points, then we may also assume that $F$ satisfies $\texttt{L}(F(P))=\texttt{L}(F(Q))=2$.
\item[(iv)]\label{lem: tile homeo 4} For each $X\in\textbf{X}$ we have $\texttt{L}(F(\textbf{X}))\leq\card(\textbf{V})$, and if $\textbf{V}\neq\emptyset$, then $\texttt{L}(F(X))\geq 1$.
\end{enumerate}
\end{lem}
	
\begin{proof}
The proof is almost identical to that of \cite[Lemma 7.3]{BM22} (the case $m=3$) with slight adjustments to notation and lemmas employed. Namely, the points $0, -1, 1\in \mathbb{C}$ of the CSST in the proof of \cite[Lemma 7.3]{BM22} correspond to $[12^{(\infty)}], [1^{(\infty)}], [2^{(\infty)}]\in \T^{m,\ba}$, respectively. Lemma \ref{lem:homeo of leaves} and Lemma \ref{lem: level of tiles} are used for the case $m\geq 3$ in place of \cite[Theorem 7.2]{BM22} and \cite[Lemma 5.5]{BM22}, respectively. Lastly, \cite[Lemma 3.3(viii)]{BM22} (which states that tiles of $m$-valent trees with dense branch points are also $m$-valent trees with dense branch points) is stated and proved for $m=3$, but the proof for $m\geq 3$ is verbatim.
\end{proof}

We are now ready to prove Proposition \ref{prop:isomorphic subd}. The desired quasi-visual subdivision $(\textbf{Y}^n)$ will be constructed inductively from auxiliary tile-homeomorphisms $F^n:T\to\mathbb{T}^{m,\ba}$ for $\textbf{X}^n$, $n\in\N_0$. Set 
$$\textbf{Y}^n:=F^n(\textbf{X}^n)=\{F^n(X): X\in\textbf{X}^n\}.$$
	
\begin{proof}[Proof of Proposition \ref{prop:isomorphic subd}]
Suppose $T$ is an uniformly $m$-branching quasiconformal tree, and let the set $\textbf{V}^n$ be as in \eqref{eq:V^n}, for $n\in\N_0$, with $\delta\in(0,1)$ small enough for Proposition \ref{prop: quasi-visual subdivision} to hold for edge-like decompositions $\textbf{X}^n$ of $T$ induced by $\textbf{V}^n$. Note that $\textbf{V}^{0}=\emptyset$, the sequence of sets $(\textbf{V}^n)$ consists of sets of branch points,  is increasing, and the induced sequence $(\textbf{X}^n)$ is a subdivision of $T$. Proposition \ref{prop: quasi-visual subdivision}(v) implies that $\textbf{V}^n\neq\emptyset$ for $n\in\N$.
		
Following the proof of \cite[Proposition 7.1]{BM22}, we can construct homeomorphisms $F^n:T\to\mathbb{T}^{m,\ba}$ for all $n\in \N_0$ with the following properties:
\begin{enumerate}
\item[(A)] For each $n \in \N_0$ the map $F^n: T \to \T^{m,\ba}$ is a tile-homeomorphism for $\textbf{X}^n$, i.e., $F^n$ is a homeomorphism from $T$ onto $\T^{m,\ba}$ such that $F^n(X)$ is of the form $\T_w^{m,\ba}$, with $w \in A^*$, for $X \in \textbf{X}^n$.
\item[(B)] For all $n\in \N$, the maps $F^{n-1}$ and $F^n$ are compatible, in the sense that 
\begin{equation*}
F^{n-1}(X) = F^n(X), \quad \text{for all $X\in \textbf{X}^{n-1}$.}
\end{equation*} 
\item[(C)] For all $n\in \N$, if $X\in \textbf{X}^{n-1}$ and $Y \in \textbf{X}^n$ with $Y \subset X$, then
\begin{equation*}
\texttt{L}(F^{n-1}(X)) + 1 \leq \texttt{L}(F^n(Y)) \leq \texttt{L}(F^{n-1}(X)) + N,
\end{equation*}where $N \in \N$  is independent of $n$ and $X$.
\item[(D)] For all $n\in \N$, if $X\in \textbf{X}^{n-1}$, $Y \in \textbf{X}^n$ with $Y \subset X$ and $Y \cap \partial X \neq \emptyset$, then
\begin{equation*}
\texttt{L}(F^n(Y)) = \texttt{L}(F^{n-1}(X)) + 2.
\end{equation*}
\end{enumerate}
The construction of $F^n$ for $m\geq 3$ is identical to that for $m=3$ and we note that Lemma \ref{lem:homeo of leaves},  Proposition \ref{prop: quasi-visual subdivision} and Lemma \ref{lem:tile homeo} are used for the former.
		
It is enough to show that $(F^n(\textbf{X}^n))$ is a quasi-visual subdivision of $\mathbb{T}^{m,\ba}$ isomorphic to $(\textbf{X}^n)$. Due to the choice of weights $\ba(i)=\frac{1}{2}$ for all $i\in A$, this follows by the above properties similarly to the case $m=3$ in \cite[Proposition 7.1]{BM22}. In particular, the fact that $(F^n(\textbf{X}^n))$ is a subdivision of $\mathbb{T}^{m,\ba}$ and that the subdivisions $(\textbf{X}^n)$ and $(F^n(\textbf{X}^n))$ are isomorphic follow from properties (A) and (B) of the homeomorphisms $F^n(\textbf{X}^n)$. To show that $(F^n(\textbf{X}^n))$ is a quasi-visual approximation of $\mathbb{T}^{m,\ba}$, the conditions (i)-(iv) in Definition \ref{def: qv approx} need to be verified. This follows by properties (C) and (D) of the homeomorphisms $F^n(\textbf{X}^n)$, as well as the geodesicity  of $\T^{m,\ba}$ (instead of the quasiconvexity of the CSST used in\cite[Proposition 7.1]{BM22}).
\end{proof}
	
We finish this section by proving the second direction of Theorem \ref{thm:uniformization2}.
	
\begin{proof}[{Proof of Theorem \ref{thm:uniformization2}}]
Recall that we have fixed $m\in \{3,4,\dots\}$, a uniformly $m$-branching tree $T$ with $\diam{T}=1$, and a weight $\ba$ such that $\ba(i)=1/2$ for all $i\in A$. 
		
There exists $\delta>0$ small enough so that the subdivision $(\textbf{X}^n)$ induced by the sets $\textbf{V}^n$ as defined in \eqref{eq:V^n} satisfies properties of Proposition \ref{prop: quasi-visual subdivision}. In particular, $(\textbf{X}^n)$ is a quasi-visual subdivision of $T$. By Proposition \ref{prop:isomorphic subd} there exists a quasi-visual subdivision $(\textbf{Y}^n)$ of $\mathbb{T}^{m,\ba}$ isomorphic to $(\textbf{X}^n)$. By Proposition \ref{prop: qs match} there exists a quasisymmetric homeomorphism $F:T\to\mathbb{T}^{m,\ba}$ that induces the isomorphism between $(\textbf{X}^n)$ and $(\textbf{Y}^n)$.
\end{proof}

\section{Bi-Lipschitz and quasisymmetric embedding into $\R^2$}\label{sec:BLembed}
	
In this section we show the following proposition.
	
\begin{prop}\label{prop: bi-Lip embeding}
For all $m\geq 2$ there exists a weight $\ba$ such that $\mathbb{T}^{m,\ba}$ bi-Lipschitz embeds into $\R^2$ with the embedded image being a quasiconvex subset of $\mathbb{R}^2$. 
\end{prop}
	
In conjunction with Theorem \ref{thm:uniformization2}, Proposition \ref{prop: bi-Lip embeding} implies that all trees $\T^{m,\ba}$ quasisymmetrically embed in $\R^2$ for all weights $\ba$.
	
\begin{cor}
For any $m\in\{2,3,\dots\}$ and any weight $\ba$, there exists a quasiconvex tree in $\R^2$ that is quasisymmetric equivalent to $\T^{m,\ba}$.
\end{cor}
	
If $m=2$, then by Proposition \ref{prop:valence}, $\mathbb{T}^{2,\ba}$ is a  metric arc and by Proposition \ref{prop:geodesic} and Corollary \ref{cor: diameters of tiles}, it is a geodesic space of diameter 1. Therefore,  $\mathbb{T}^{2,\ba}$ is isometric to the unit interval $[0,1]$.
	
Fix for the rest of this section an integer $m\geq 3$. Let $A=\{1,\dots,m\}$, let $\theta = \frac{\pi}{3m-3}$ and let $\ba$ be a weight with $\ba(1)=\ba(2)=1/2$, $\ba(j) = \frac{1}{2}\sin\theta$ for all $j\in\{3,\dots,m\}$.
	
Let $R\subset \mathbb{R}^2$ be the convex hull of the 4 points $\{\pm\frac{1}{2}, \pm\frac{i}{2}\tan\theta\}$ and let $\{\psi_j: \mathbb{C} \to \mathbb{C}\}_{j=1}^{m}$ be contracting similarities given (in complex notation) by
\[ \psi_1(z) = \ba(1)(\overline{z} - \tfrac12), \quad \psi_2(z) = \ba(2)(z + \tfrac12), \]
and
\[\psi_j(z) = \ba(j)\left(z + \tfrac12\right)e^{i(3j-6)\theta}\quad \text{for $j\in\{3,\dots,m\}$}.\]
See Figure \ref{fig:IFS-m=4} for the case $m=4$.
	
	\begin{figure}[h]
		\centering
		\includegraphics[width=\textwidth]{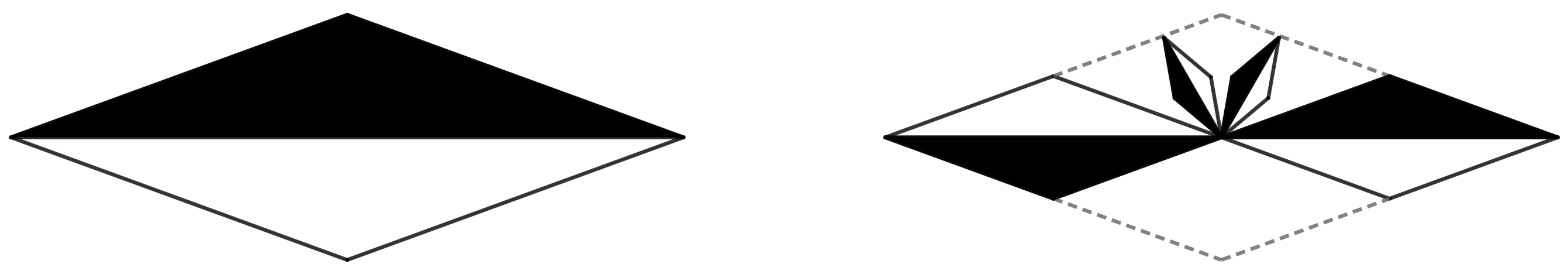}
		\caption{The images of $R$ under similarities $\psi_1,\dots,\psi_4$. Here we have chosen $m=4$.}
		\label{fig:IFS-m=4}
	\end{figure}

For each $w=j_1\cdots j_n\in A^{*}$, define 
$$\psi_w:=\psi_{j_1}\circ \psi_{j_2}\circ \cdots \circ \psi_{j_n}$$
with the convention that $\psi_{\e}$ is the identity map on $\mathbb{C}$. Note that the scaling factor of $\psi_w$ is $\Delta_{\ba}(w)=\ba(j_1)\cdots\ba(j_n)$. For each $w\in A^*$, set $R_w = \psi_w(R)$. It is elementary to show that the choice of the angle $\theta$ implies that the rhombuses $R_j$, $j\in A$, lie in $R$ and are well separated in the following sense.
	
\begin{lem}\label{lem:geometry of rhombuses}
For each distinct $j,j' \in A$, $R_j \subset R$, $R_j \cap R_{j'} = \{0\}$, and there exists a double cone centered at 0 and of angle $\theta$ that separates the interior of $R_j$ from the interior of $R_{j'}$.
\end{lem}
	
By Lemma \ref{lem:geometry of rhombuses}, $\psi_w(R)\subset R$ for all $w\in A^{*}$. By \cite{Hu81}, the iterated function system $\{\psi_j\}_{j\in A}$ has an attractor $\mathcal{T}^m$. That is, $\mathcal{T}^m$ is the unique nonempty compact set such that $\mathcal{T}^m = \bigcup_{j\in A}\psi_j(\mathcal{T}^m)$. For each $w\in A^*$, we set $\mathcal{T}^m_w = \psi_w(\mathcal{T}^m)$. See Figure \ref{fig:attractor(m=4)} for the case $m=4$.\footnote{Figure \ref{fig:attractor(m=4)} was generated using the IFS construction kit in \href{https://adelapo.github.io/ifs.html}{https://adelapo.github.io/ifs.html}.}
	
	\begin{figure}[h]
		\centering
		\includegraphics[width=\textwidth]{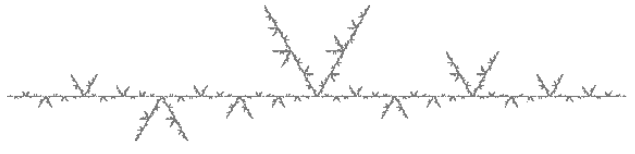}
		\caption{The attractor $\mathcal{T}^4$ of the iterated function system $\{\psi_1,\psi_2,\psi_3,\psi_4\}$.}
		\label{fig:attractor(m=4)}
	\end{figure}
	
The proof of the following is almost identical with that  in \cite[Proposition 4.2]{BT_CSST}, hence, we omit the details. 
\begin{lem}\label{lem:tree-skeletons}
Let $J_0=[-\frac{1}{2},\frac{1}{2}]\subset \mathbb{C}$. For $n\geq 0$ define 
$$J_n=\bigcup_{w\in A^{n}}\psi_w(J_0) \quad \text{and} \quad K_n=\bigcup_{w\in A^{n}}\psi_w(R).$$
Then the sets $J_n$ and $K_n$ are compact for all integers $n\geq 0$ and
\begin{equation*}
J_n\subset J_{n+1}\subset \mathcal{T}^m\subset K_{n+1}\subset K_n.
\end{equation*}
\end{lem}
	
Since $J_0\subset \mathcal{T}^m\subset R$ and $\diam(J_0)=\diam(R)=1$, we have $\diam(\mathcal{T}^m)=1$. Therefore, $\diam{\mathcal{T}_w^{m}}=\Delta(w)$ for all $w\in A^*$.
	
Observe that $0=\psi_1(\frac{1}{2})=\psi_2(-\frac{1}{2})=\cdots = \psi_m(-\frac12)$, so $0\in \mathcal{T}_k^m\subset R_k$ for all $k\in A$. Therefore, if $j,j'\in\{1,\dots, m\}$ are distinct, Lemma \ref{lem:geometry of rhombuses} implies that $\mathcal{T}_k^m\cap\mathcal{T}_l^m=\{0\}$.
	
\begin{lem}\label{lem: J_n connected}
For all $n\in\N_0$, $J_n$ is connected.
\end{lem}

\begin{proof}
The proof is done by induction. For $n=0$ it's clear. Assume that $J_n$ is connected for $n\geq 1$. Then,
$$J_{n+1}=\bigcup_{w\in A^{n+1}}\psi_w(J_0)=\bigcup_{i=1}^m\bigcup_{w'\in A^n}\psi_{iw'}(J_0)=\bigcup_{i=1}^m\psi_i\big(\bigcup_{w'\in A^n}\psi_{w'}(J_0)\big)=\bigcup_{i=1}^m\psi_i(J_n).$$
The sets $\psi_i(J_n)$ are connected by continuity of  $\psi_i$ for each $i\in\{1,\dots,m\}$. Moreover, $0\in \psi_i(J_n)$ for all $i$, since $\psi_i(J_0)\subset\psi_i(J_n)$ by Lemma \ref{lem:tree-skeletons} and $0\in \psi_i(J_0)$ by the above observation. Hence, $J_{n+1}$ is connected.
\end{proof}

\begin{lem}\label{lem:quasiconvex}
The set $\mathcal{T}^m$ is a quasiconvex metric tree in $\mathbb{C}$.
\end{lem}
	
\begin{proof}
That $\mathcal{T}^m$ is a metric tree, can be proved using verbatim the same techniques as in \cite[Proposition 1.4]{BT_CSST} (see also \cite[Lemma 4.3, Lemma 4.4 and Lemma 4.5]{BT_CSST}). The fact that $\mathcal{T}^m$ is quasiconvex follows from the existence of double cones separating distinct $R_j,R_{j'}$ in Lemma \ref{lem:geometry of rhombuses}, and applying the techniques of \cite[Proposition 1.4]{BT_CSST} verbatim. 
\end{proof}
	
Define a geodesic metric $\rho$ on $\mathcal{T}^m$ by setting $\rho(a,b)=\text{length}(\gamma)$ for $a,b\in\mathcal{T}^m$, where $\gamma$ is the unique arc in $\mathcal{T}^m$ joining $a$ and $b$. Then, by Lemma \ref{lem:quasiconvex}, the metric space $(\mathcal{T}^m,\rho)$ is bi-Lipschitz equivalent to $\mathcal{T}^m$ (equipped with the Euclidean metric).  
	
\begin{lem}\label{lem: isometric embeding of J_n}
For each $n\geq 0$, the metric space $(J_n,\rho)$ isometrically embeds into $\mathbb{T}^{m,\ba}$. 
\end{lem}
	
\begin{proof}
First, for each integer $n\geq 0$, $J_n$ is closed (since it's compact by Lemma \ref{lem:tree-skeletons}) and connected by Lemma \ref{lem: J_n connected}; hence, as a  subset of $\mathcal{T}^m$, it is itself a geodesic metric tree \cite[Lemma 3.3]{BT_CSST}. 
		
Let $\gamma : I\rightarrow \mathbb{T}^{m,\ba}$ be the unique arc connecting $[1^{(\infty)}]$ with $[2^{(\infty)}]$, where $I$ denotes the unit interval. Define for each integer $n\geq 0$, 
\[ \mathcal{J}_n := \bigcup_{w\in A^{n}}\phi_w(\gamma(I))\subset \mathbb{T}^{m,\ba}.\] 
We claim that for all integers $n\geq 0$ there exists an isometric map $f_n : J_n \to \mathcal{J}_n$ satisfying $f_n(-\frac{1}{2})=[1^{(\infty)}]$ and $f_n(\frac{1}{2})=[2^{(\infty)}]$. We prove the claim by induction on $n$. 
		
The case $n=0$ is trivial, since $(\gamma(I),\rho_{\ba})$ and $(J_0,\rho)$ are both isometric to $I$, due to $\T^{m,\ba}$ being geodesic. 
		
Assume now that for some integer $n\geq 0$ there exists an isometric homeomorphism $f_n : J_n \to \mathcal{J}_n$ with $f_n(-\frac{1}{2})=[1^{(\infty)}]$ and $f_n(\frac{1}{2})=[2^{(\infty)}]$. For each $j\in A$, set $J_{n,j} = \psi_j(J_n)$ and $\mathcal{J}_{n,j}= \phi_j(\mathcal{J}_n)$. Since $\psi_j$ has the same scaling factor as $\phi_j$, we have that the map
\[g_{n,j}=\phi_j \circ f_n \circ(\psi_j|J_{n,j})^{-1}:J_{n,j} \to \mathcal{J}_{n,j}\]
is an isometric homeomorphism. Note that $-\frac{1}{2}\in J_{n,1}$ and $\frac{1}{2}\in J_{n,2}$ since $-\frac{1}{2}=\psi_1(\psi_{1^{(n)}}(-\frac{1}{2}))$ and $\frac{1}{2}=\psi_2(\psi_{2^{(n)}}(\frac{1}{2}))$. Hence, 
\[ g_{n,1}(-\tfrac{1}{2})=g_{n,1}(\psi_1(-\tfrac{1}{2}))=\phi_1(f_n(-\tfrac{1}{2}))=\phi_1([1^{(\infty)}])=[1^{(\infty)}].\] 
Similarly, we get $g_{n,2}(\frac{1}{2})=[2^{(\infty)}]$. Note that by $-\frac{1}{2}, \frac{1}{2}\in J_{n}$ for all $n$, we also have $0\in J_{n,j}$ for all $j\in A$. As a result, for distinct $j,j' \in A$ we have
\begin{equation}\label{eq: 0 is J_n,j intersection}
\{0\}\subset J_{n,j}\cap J_{n,j'}=\psi_j(J_n)\cap \psi_{j'}(J_n)\subset \psi_j(\mathcal{T}^m)\cap\psi_{j'}(\mathcal{T}^m)\subset \{0\}.
\end{equation}
Hence the sets $J_{n,j}$ have only 0 as common point.  Similarly, the sets $\mathcal{J}_{n,j}$ have only $[12^{(\infty)}]$ as common point (see Lemma \ref{lem: preimage 12^00}). Observe that
\[ J_{n+1}= \bigcup_{w\in A^{n+1}}\psi_w(J_0)=\bigcup_{j\in A}\bigcup_{w\in A^n}\psi_j(\psi_w(J_0))=\bigcup_{j\in A} J_{n,j}.\]
Similarly, $\mathcal{J}_{n+1}= \bigcup_{w\in A^{n+1}}\phi_w(\gamma(I))=\bigcup_{j=1}^m \mathcal{J}_{n,j}$. 
		
Define $f_{n+1}:(J_{n+1},\rho) \to (\mathcal{J}_{n+1},\rho_{\ba})$ by $f_{n+1}|J_{n,j} = g_{n,j}$. By \eqref{eq: 0 is J_n,j intersection} the map $f_{n+1}$ is well defined, since $f_{n+1}(0)=g_{n,j}(0)=[12^{(\infty)}]$ for all $j\in\{1,\dots, m\}$, $f_{n+1}(-\frac{1}{2})=g_{n,1}(-\frac{1}{2})=[1^{(\infty)}]$, and $f_{n+1}(\frac{1}{2})=g_{n,2}(\frac{1}{2})=[2^{(\infty)}]$. 
		
To finish the inductive step, we show that $f_{n+1}$ is an isometry. To this end, fix $x,y \in J_{n+1}$. The case $x,y\in J_{n,j}$ for some $j\in A$ follows by the fact that each $g_{n,j}$ is an isometry. Assume now that $x\in J_{n,j}$ and $y\in J_{n,j'}$ for two distinct $j,j'\in A$. By the geodesicity of $\T^{m,\ba}$ and $(J_{n+1},\rho)$,
\begin{align*}
d_{\mathbf{a}}(f_{n+1}(x),f_{n+1}(y)) &= d_{\mathbf{a}}(g_{n,j}(x),[12^{(\infty)}]) + d_{\mathbf{a}}([12^{(\infty)}], g_{n,j'}(y))\\ 
&= \rho(x,0) + \rho(0,y)\\ 
&= \rho(x,y). \qedhere
\end{align*}
\end{proof}
	
\begin{proof}[{Proof of Proposition \ref{prop: bi-Lip embeding}}]
Recall the definitions of sets $(J_n)_{n\in \N_0}$, $(\mathcal{J}_n)_{n\in \N_0}$, and isometric homeomorphisms $f_n: (J_n,\rho) \to \mathcal{J}_n$ from the proof of Lemma \ref{lem: isometric embeding of J_n}.
		
By Lemma \ref{lem:tree-skeletons}, each geodesic tree $(J_n,\rho)$ is a subset of $(\mathcal{T}^m,\rho)$. By Lemma \ref{lem: isometric embeding of J_n} we have 
\begin{align*}
\dist_H(\mathbb{T}^{m,\ba}, \mathcal{J}_n)=\sup_{[v]\in \mathbb{T}^{m,\ba}}d_{\ba}([v],\mathcal{J}_n) &\leq \sup_{[v]\in \mathbb{T}^{m,\ba}}\inf_{w\in A^n}d_{\ba}([v],[w12^{(\infty)}])\\ 
&\leq \sup_{w\in A^n}\diam{\T^{m,\ba}_w}\\ 
&\leq 2^{-n}
\end{align*}
where in the third inequality we used the fact that $\{[w12^{(\infty)}]:w\in A^n\} \subset \mathcal{J}_n$.
		
Therefore, we have
\begin{equation}\label{eq: d_GH of Tma and Jn}
\lim_{n\to\infty}\dist_{GH}(\mathbb{T}^{m,\ba},(J_n,\rho)) \leq \lim_{n\to\infty}\dist_H(\mathbb{T}^{m,\ba}, \mathcal{J}_n) =0.
\end{equation}
		
Fix $n\geq 1$ and $x\in \mathcal{T}^m$. Then there exists $w(n) \in A^n$ such that $x \in \mathcal{T}^m_{w(n)}$. By Lemma \ref{lem:quasiconvex}, $\mathcal{T}^m$ is $L$-bi-Lipschitz homeomorphic to $(\mathcal{T}^m,\rho)$ for some $L\geq 1$. Therefore,
		
\begin{align*}
\rho(x,\psi_{w(n)}(0)) \leq \diam_{\rho}{\mathcal{T}^m_{w(n)}} \leq L \diam{\mathcal{T}^m_{w(n)}} \leq L\Delta_{\ba}(w(n)) \leq L 2^{-n},
\end{align*}
which implies that
\begin{equation}\label{eq: d_GH of calT_m and Jn}
\lim_{n\to \infty} \dist_H((\mathcal{T}^m,\rho), (J_n,\rho)) \leq \lim_{n\to \infty} \sup_{x\in \mathcal{T}^m}\inf_{w\in A^n}\rho(x,\psi_w(0)) =0.
\end{equation}
		
Hence, by triangle inequality and \eqref{eq: d_GH of Tma and Jn}, \eqref{eq: d_GH of calT_m and Jn}, the Gromov-Hausdorff distance of $(\mathcal{T}^m,\rho)$ and $\mathbb{T}^{m,\ba}$ is zero, which means that they are isometric to each other. As a result, $\mathbb{T}^{m,\ba}$ is bi-Lipschitz homeomorphic to $\mathcal{T}^m$ (with the Euclidean metric).
\end{proof}

\section{Preliminaries on geodesic trees}\label{sec:glue}

For the proof of Theorem \ref{thm:main}, we require two preliminary notions for geodesic trees.

\subsection{Geodesic gluing}

We start by defining a notion of ``gluing" for geodesic metric spaces.

\begin{definition}\label{def:geodesic gluing}
Let $(X,d_X,\bx_I)$, $(Y_i,d_{Y_i},y_i)$ be pointed geodesic metric spaces, where $I$ is a countable subset of $\N$, $x_i\in X$, $y_i\in Y_i$ for all $i\in I$, and $\bx_I:=(x_1, x_2, \dots)$. In what follows, we identify $X\times\{0\}$ with $X$, and for given $i\in I$ we identify $Y_i\times\{i\}$ with $Y_i$.
		
Define a pseudo-metric on $(X\times \{0\})\cup\bigcup_{i\in I} (Y_i\times \{i\}) = X\cup\bigcup_{i\in I} Y_i$ by
$$
\rho(z,w)= 
\begin{cases}
d_X(z,w) & \text{if $z,w \in X$}\\
d_{Y_i}(z,w) & \text{if $z,w \in Y_i$ for some $i\in I$}\\
d_X(z,x_i)+d_{Y_i}(y_i,w) & \text{if $z\in X$, $w\in Y_i$ for some $i\in I$}, \\
d_X(w,x_i)+d_{Y_i}(y_i,z) & \text{if $z\in Y_i$, $w\in X$ for some $i\in I$}, \\
d_{Y_i}(z,y_i)+d_X(x_i,x_j)+d_{Y_j}(y_j,w) & \text{if $z\in Y_i$, $w\in Y_j$ for distinct $i,j\in I$}.\\
\end{cases} 
$$
		
The pseudometric $\rho$ gives rise to a metric $d$ on the quotient space $(X\cup\bigcup_{i\in I}Y_i)/\sim$ where $w_1\sim w_2$ if $\rho(w_1,w_2)=0$. Define now the \textit{geodesic gluing} of $X$ and $Y_i$ (at $\bx_I$ and $y_i$)
$$(X,d_X,\bx_I)\bigvee_{i\in I} (Y_i,d_{Y_i},y_i): =((X\cup\bigcup_{i\in I}Y_i)/\sim, d).$$ 
\end{definition}
	
It is easy to see that the geodesic gluing of geodesic spaces is itself a geodesic space. Intuitively, we are simply identifying $x_i$ with $y_i$ for all $i\in I$. It is clear that $X$ and $Y_i$ isometrically embed into the above geodesic gluing. We make the convention that points of $X$ and $Y_i$ are treated as points of a ``subspace" in the geodesic gluing metric space, in order to ease the notation.
	
\begin{lem}\label{lem:gluing is m.tree}
Let $X$, $Y_i$ be geodesic metric trees, for all $i\in I$. Assume that, either $\card I<\infty$, or $\card I =\infty$ and $\diam Y_i\to 0$ as $i\to\infty$. Then, the geodesic gluing $(X,d_X,\bx_I)\bigvee_{i\in I} (Y_i,d_{Y_i},y_i)$ is a geodesic metric tree.
\end{lem}
	
\begin{proof}
Denote by $(Z,d)$ the geodesic gluing of $X$ and $Y_i$ at $\bx_I$ and $y_i$.  The fact that $(Z,d)$ is a geodesic metric space follows by Definition \ref{def:geodesic gluing}. We prove the case where $I$ is infinite, and the case $\card I<\infty$ is similar. We first show that $Z$ is compact.  Let $(z_n)_{n\in \N}$ be a sequence in $Z$. If $z_n\in X$ for infinitely many $n$, then there is an accumulation point in $X$, and, hence, a convergent subsequence in $X\subset Z$. Similarly, if $z_n\in Y_i$ for some $i\in I$ and infinitely many $n$, then there is an accumulation point in $Y_i$, and a convergent subsequence in $Y_i$. For the last case, suppose that $z_n\in Y_n$ for all $n\in\N$, by passing to a subsequence if necessary. Consider the sequence $(x_n)_{n\in \N}$ in $Z$, where $x_n$ is the term of $\bx_I$ at which we glue the metric space $Y_n$. Since $X$ is compact, there exists a convergent subsequence $(x_{n_j})_{j\in \N}$ that converges to a point $z\in X$. We have
$$d(z_{n_j},z)\leq d(z_{n_j},x_{n_j})+d(x_{n_j},z)=d(z_{n_j},y_{n_j})+d(y_{n_j},z)\leq \diam Y_{n_j} +d(x_{n_j},z).$$
By letting $j\to \infty$, it follows that $z_{n_j}\to z\in Z$. Therefore, the arbitrary sequence $(z_n)_{n\in \N}$ in $Z$ always has a convergent subsequence, proving that $Z$ is compact. 
		
We show that there exists a unique arc between any two distinct points in $Z$; hence, $Z$ is connected. Let $x,y\in X$ be distinct. There exists a unique arc $[x,y]$ in $X$, which is also an arc in $Z$. Assume towards contradiction that there exists another arc $\alpha\neq [x,y]$ in $Z$ that connects $x$ and $y$. There is a fixed $i\in I$ for which $\alpha\cap(Y_i\setminus\{y_i\})\neq\emptyset$, because otherwise $\alpha\subset X$, which would contradict the uniqueness of $[x,y]$ in $X$. Let $z\in \alpha\cap(Y_i\setminus\{y_i\})$. We can write $X\cap Y_i=\{y_i\}$, due to $x_i\sim y_i$ and an abuse of notation that we follow henceforth, identifying $x_i$ with $y_i$ as points in $Z$. This implies that
$$\alpha=[x,z]\cup (z,y]=[x,y_i]\cup(y_i,z]\cup(z,y_i]\cup (y_i,y],$$
which shows that $\alpha$ is not an arc. If $x,y\in Y_i$ for some $i\in I$, the proof is similar. Suppose $x\in X$ and $y\in Y_i$ for some $i\in I$. The arc $\beta=[x,y_i]\cup(y_i,y]$ connects $x$ and $y$, but due to $X\cap Y_i=\{y_i\}$, any other possible arc connecting these two points must intersect $y_i$. Similarly to the previous case, it can be shown that $[x,x_i]\subset X$ and $[y_i,y]\subset Y$ are unique arcs in $Z$. Hence, by the identification $[x,x_i]=[x,y_i]$ in $Z$, we have that $\beta$ is the unique arc in $Z$ that connects $x$ and $y$. Suppose that $x\in Y_i$ and $y\in Y_j$ for $i\neq j$. It can be shown similarly that the arc $[x,y_i]\cup(y_i,y_j]\cup(y_j,y]$, which connects $x$ and $y$ in $Z$, is unique.
		
The geodesicity of $Z$ implies that $Z$ is locally connected. Therefore, $Z$ is a geodesic metric tree.    
\end{proof}
	
\subsection{A new notion of height}
	
Recall the notion of height $H_T(p)$ from \eqref{eq:height}. For most of this section, we need to replace it by a new notion. Given a metric tree $T$, a branch point $p$ in $T$, and a branch $B$ of $T$ at $p$, define $h_T (p,B)$ to be the maximal distance from $p$ within $B$, i.e.,
\begin{equation}\label{eq:height of a branch}
h_T (p,B) := \max_{x\in B} d(x,p).
\end{equation}
	
We need the following relation between branches of distinct branch points, in order to establish certain properties involving the above maximal distance.
	
\begin{lem}\label{lem: brances subset of branch}
Let $T$ be a tree, let $p\in T$, let $B_T^i(p)$ be a branch of $T$ at $p$, for some $i\in \{1,\dots,\val(p)\}$, and let $q\in B_T^i(p)$ that is not a leaf. If $B_T^j(q)$ is a branch of $q$ such that $p\in B_T^j(q)$, $j\in \{ 1, \dots, \val(q)\}$, then $B_T^k(q)\subset B_T^i(p)$ for all $k\neq j$.
\end{lem}
	
\begin{proof}
Note that if $B$ is a connected component of $T\setminus\{ p \}$, then $\partial B=\{ p \}$ by \cite[Lemma 3.2(ii)]{BT_CSST}. If $p$ is a leaf, then the result holds due to $B_T^i(p)=T\setminus\{p\}$ and $p\notin B_T^k(q)$ for all $k\neq j$. Suppose that $p$ is not a leaf. Let $z\in B_T^k(q)$ for $k\neq j$. Since $p\in B_T^j(q)$, it follows that $q\in [z,p]$ by \cite[Lemma 3.2(iii)]{BT_CSST}. If $z\in B_T^{\ell}(p)$ for $\ell\neq i$, then the arc $[z,p)$ lies in $B_T^{\ell}(p)$. But $[z,p]=[z,q]\cup (q,p]$, and due to $q\in B_T^i(p)$ we have that $[z,p]\cap B_T^i(p)\neq\emptyset$. This contradicts the fact that $B_T^i(p)$ and $B_T^{\ell}(p)$ are distinct connected components of $T\setminus\{p\}$. Hence, $z\in B_T^i(p)$. 
\end{proof}
	
If $T$ is a geodesic tree, then the maximum in \eqref{eq:height of a branch} occurs at leaves of $T$, i.e.,
\begin{equation}\label{eq:height of a branch2}
h_T (p,B) = \max \{d(x,p): \text{$x$ leaf on $B$}\}
\end{equation}
Indeed, assume towards contradiction that $T$ is a geodesic tree, and that the maximum in \eqref{eq:height of a branch} occurs at a point $x\in B$  that is not a leaf. This implies that there exists a branch $B_T(x)$ of $T$ at $x$ that does not contain $p$. By Lemma \ref{lem: brances subset of branch}, we have $B_T(x)\subset B$. Let $q\in B_T(x)$. Since $q$ and $p$ lie on different branches at $x$, we have $x\in[p,q]$ by \cite[Lemma 3.2(iii)]{BT_CSST}. Therefore, by geodesicity we have
$$ d(p,q)=d(p,x)+d(x,q). $$
This implies that $d(p,q)> d(p,x)$ for $q\in B_T(p)$, which contradicts the assumption. Hence, the maximum cannot be achieved at $x$, unless $x$ is a leaf.
	
Let $T$ be a geodesic tree, let $p\in T$ be a branch point, and let $B_T^1(p),B_T^2(p), \dots $ be an enumeration of the branches of $T$ at $p$ with $h_T (p,B_T^1(p)) \geq h_T (p,B_T^2(p)) \geq\dots$. We define the \textit{new height} of $p$ to be
	\begin{equation}\label{eq:new height}
		h_T(p) := h_T (p,B_T^3(p)) = \max \{d(x,p): \text{$x$ leaf on $B_T^3(p)$}\}.
	\end{equation}
To simplify the notation henceforth, we also denote $h_T (p,B_T^j(p))$ by $h_T (p,B_j)$, for all $j$.
	
The two notions of height are related in a way that allows us to interchange them in certain cases. We say that a metric tree $T$ is \textit{uniformly $n$-branching with respect to the new height $h_T$}, for some integer $n\geq 3$, if every branch point of $T$ has valence $n$, and $T$ has uniform branch growth, uniformly relatively separated branch points, and uniformly relative dense branch points \textit{with respect to $h_T$}, i.e., Definitions \ref{def:unifgrowth}, \ref{def:unifsep}, and \ref{def:unifden} are true for $T$ with $H_T(p)$ and $\diam B_T^i(p)$ replaced by $h_T(p)$ and $h_T(p,B_i(p))$, respectively. The following lemma shows that the new height and the one defined at \eqref{eq:height} are comparable at every branch point.
	
\begin{lem}\label{lem: comparable heights}
Given a geodesic tree and a branch point $p\in T$, 
\begin{equation}\label{eq:relation of two heights}
h_T(p)\simeq H_T(p),
\end{equation}
where $H_T(p)$ is defined at \eqref{eq:height}, and the constant $C(\simeq)$ is independent of the metric tree $T$ and the branch point $p$.
		
Moreover, $T$ is uniformly $n$-branching with respect to $h_T$, if, and only if, $T$ is uniformly $n$-branching with respect to $H_T$.
\end{lem}

\begin{proof}
Fix a branch point $p\in T$. We have two different enumerations of the branches of $T$ at $p$. We denote by $B_1^H,B_2^H,\dots$ the enumeration of the branches at $p$ in non-increasing order with respect to their diameters, and by $B_1^h,B_2^h,\dots$ the enumeration of the branches of $p$ in non-increasing order with respect to the maximal distances defined in \eqref{eq:height of a branch2}. We need to show that $\diam B_3^H\simeq h_T(p,B_3^h)$. Fix  a branch $B_k^h$ at $p$. If $x,y\in B_k^h$, then $d(x,y)\leq d(x,p)+d(p,y)\leq 2h_T(p, B_k^h)$, which implies that $\diam B_k^h\leq 2 h_T(p,B_k^h)$. On the other hand, if $x\in B_k^h$, then $d(x,p)\leq \diam \overline{B_k^h}=\diam B_k^h$. Therefore, we have shown that
\begin{equation}\label{eq: height on the same branch}
\diam B_k^h\simeq h_T(p,B_k^h).
\end{equation}
		
Set $B_3^h=B_i^H$ and $B_3^H=B_j^h$ for some $i,j\in\{1,\dots,\val(p)\}$. The rest of the proof is a case study on $i,j$.
		
If $i=3$, then the heights are comparable by \eqref{eq: height on the same branch}.
		
If $i\in \{1,2\}$, we have
$$\diam B_3^H\leq \diam B_i^H=\diam B_3^h\simeq h_T(p).$$
For the other side of \eqref{eq:relation of two heights}, note that there is integer $k\geq 3$ such that either $B_1^h=B_{k}^H$, or $B_2^h=B_{k}^H$. Suppose that $B_1^h=B_{k}^H$, and the proof in the other case is similar. We then have
$$\diam B_3^H\geq \diam B_{k}^H=\diam B_1^h\simeq h_T(p,B_1^h)\geq h_T(p),$$
and \eqref{eq:relation of two heights} follows in this case.
		
Suppose that $i> 3$. We have that
$$\diam B_3^H\geq \diam B_i^H=\diam B_3^h\simeq h_T(p).$$
For the other side of the inequality, if $j\geq 3$ then 
$$\diam B_3^H=\diam B_j^h\simeq h_T(p,B_j^h)\leq h_T(p).$$
If $j\in \{1,2\}$, then either $B_1^H=B_{\ell}^h$, or $B_2^H=B_{\ell}^h$, for some integer $\ell\geq 3$. Suppose that $B_1^H=B_{\ell_1}^h$, and the proof in the other case is similar. It follows that
$$\diam B_3^H\leq \diam B_1^H=\diam B_{\ell_1}^h\simeq h_T(p, B_{\ell_1}^h)\leq h_T(p).$$
Therefore \eqref{eq:relation of two heights} holds.
		
It remains to show that the uniform $n$-branching property with respect to $h_T$ is equivalent to that with respect to $H_T$. We prove one direction, namely that the uniform $n$-branching property with respect to $h_T$ implies that with respect to $H_T$, and the other direction is similar. Suppose that $T$ is uniformly $n$-branching with respect to $h_T$. The uniform branch separation and uniform density of branch points with respect to $H_T$ follow immediately by \eqref{eq:relation of two heights}. It remains to show the uniform branch growth of $T$ with respect to $H_T$. We assume that $\val(T)=n\geq 4$, since the case $\val(T)=3$ is trivial. Fix a branch point $p\in T$ with $\val(T,p)=m\in\{4,\dots,n\}$. We need to show that $\diam B_m^H\gtrsim \diam B_3^H$. Let $B_m^H=B_i^h$ for some $i\in \{1,\dots,m\}$. If $i\in\{1,2\}$, then 
$$\diam B_m^H=\diam B_i^h\geq h_T(p,B_i^h)\geq h_T(p,B_3^h)=h_T(p)\simeq H_T(p)=\diam B_3^H.$$
If $i\geq 3$, then 
$$\diam B_m^H=\diam B_i^h\geq h_T(p,B_i^h)\simeq h_T(p,B_3^h)=h_T(p)\simeq H_T(p)=\diam B_3^H,$$
by the uniform branch growth of $T$ with respect to $h_T$. Therefore, $T$ has uniform branch growth with respect to $H_T$.
\end{proof}
	
\section{Quasisymmetric embeddings}\label{sec:QSembed}
In this section, we show the following proposition. 
	
\begin{prop}\label{prop:qsembed}
If $T$ is a quasiconformal tree with $\val(T)=n$ and uniform branch separation, then $T$ quasisymmetrically embeds into a uniformly $n$-branching quasiconformal tree.
\end{prop}
	
For the rest of this section, we employ the new height $h_T$ defined in \eqref{eq:new height}, and the enumeration of branches at a given branch point is with respect to the non-increasing order induced by the maximal distance within each branch defined in \eqref{eq:height of a branch2} (and not the one induced by the diameter of each branch). In addition, we often refer to the uniform $n$-branching property without specifying the employed height notion, wherever that is clear from the context.
	
The proof of Proposition \ref{prop:qsembed} splits into three steps. At each step, we construct a quasisymmetric  embedding of a given quasiconformal tree $T$ into a tree $T'$ with one extra property. In particular, given a tree $T_0 \in \mathscr{QCT}^*(n)$,
\begin{itemize}
\item in \textsection\ref{sec:step1} we show that $T_0$ quasisymmetrically embeds into a tree $T_1\in \mathscr{QCT}^*(n)$ that has uniform branch growth,
\item in \textsection\ref{sec:step2} we show that $T_1$ quasisymmetrically embeds into an $n$-valent tree $T_2\in \mathscr{QCT}^*(n)$ that has uniform branch growth,
\item in \textsection\ref{sec:step3} we show that $T_2$ quasisymmetrically embeds into an $n$-valent tree $T_3\in \mathscr{QCT}^*(n)$ that has uniform branch growth and uniform branch density (hence $T_3$ is uniformly $n$-branching).
\end{itemize}
	
In each case, applying a uniformization result of Bonk and Meyer \cite{BM20}, we may assume that all trees are geodesic and our embeddings will be in fact isometric. The isometric embedding discussed in Definition \ref{def:geodesic gluing} allows us to view each $T_i$ ($i=0,1,2$) as a subspace of the geodesic gluing $T_{i+1}$. Henceforth, we do not address the aforementioned isometric embeddings directly, and we focus on simply constructing $T_{i+1}$ ``on top" of $T_{i}$. 
	
\subsection{Step 1: uniform branch growth}\label{sec:step1}
The first step in the proof of Proposition \ref{prop:qsembed} is the following proposition.
	
\begin{prop}\label{prop:step 1}
Let $T$ be a quasiconformal tree with $\val(T)=m$, and uniformly separated branch points. Then $T$ quasisymmetrically embeds into a doubling geodesic tree $T'$ with $\val(T')=m$,  uniformly separated branch points, and uniform branch growth.
\end{prop}
	
\begin{proof}
Any quasiconformal tree is quasisymmetrically equivalent to a geodesic metric tree by \cite[Theorem 1.2]{BM20}. Hence, we may assume  that the tree $(T,d)$ is a geodesic doubling metric tree with uniformly separated branch points and $\val(T)=m$ (since these properties are quasisymmetrically invariant). 
		
Let $S=\{p\in T : p \ \text{is a branch point of}\ T\ \text{with} \val(T,p)\geq 4\}$. If $S=\emptyset$, then $T$ trivially has uniform branch growth. Suppose that $S\neq\emptyset$. For all $p$ in $S$, let $i_p\in\{3,\dots,\val(T,p)\}$ be the maximal index such that $h_T(p)=h_T(p,B_{i_p})$.  The branches at branch points where $i_p=\val(T,p)$ satisfy the uniform branch growth property. Consider the set $S'=\{p\in S:i_p\in\{3,\dots,\val(T,p)-1\}\}$. Similarly, we may assume that $S'\neq\emptyset$. Fix a branch point $p$ in $S'$. By \eqref{eq:height of a branch2}, for every $j\in\{i_p+1,\dots,\val(T,p)\}$ there exists a leaf $q_p^j\in B_T^j(p)$ with $h_T(p,B_j)=d(p,q_p^j)$.  Note that for a given $j$, there could be uncountably many $q_p^j$ on the branch $B_T^j(p)$. Hence, for every $j\in\{i_p+1,\dots,\val(T,p)\}$ we fix some $q_p^j\in B_T^j(p)$ with $h_T(p,B_j)=d(p,q_p^j)$, and we define
$$\mathcal{L}_p=\{q_p^j\in T:\  h_T(p,B_j)=d(p,q_p)\ \text{for all}\ j\in\{i_p+1,\dots,\val(T,p)\}\}.$$ 
Note that $\card\mathcal{L}_p\leq\val(T,p)-3$, and suppose there exists $p\in T$ branch point with $\mathcal{L}_p\neq\emptyset$, because otherwise $T$ has uniform branch growth. Set $\mathcal{L}^*=\bigcup_{p\in S'}\mathcal{L}_p$. By \cite[Chapter V, (1.3)(iv)]{Whyburn_book} the branch points of any metric tree are countable. Therefore, $\mathcal{L}^*$ is also a countable set, i.e. $\mathcal{L}^*=\{x_1,x_2,\dots\}$. Let $I$ be an index set with $\card\mathcal{L}^*=\card I$. Suppose that $\card I=\infty$, set $\bx_I=(x_1,\dots)$, and the proof is similar in the case $\card I<\infty$. We construct the tree $T'$ into which we isometrically embed $T$ (seen as a subspace of $T'$), by gluing a Euclidean line segment of suitable length on $T$ at each $q_p^j$, for every $p\in S'$. This ensures that $T'$ has uniform branch growth, even for branches at $p$, where the property does not necessarily hold for $T$.  
		
For every $x\in \mathcal{L}^*$ there exists a $p\in S'$ such that $x=q_p^j$, for some $j\in\{i_p+1,\dots,\val(T,p)\}$. Let $\ell_x$ be a Euclidean line segment with $\diam\ell_x= h_T(p)-h_T(p,B_j)$, and fix $y_x$ to be one of the endpoints of $\ell_x$.  We first need to make sure that for every $x\in \mathcal{L}^*$, the line segment $\ell_x$ is well-defined, i.e., we have to show that there is a unique branch point $p\in S'$ such that $x=q_p^j$. It is enough to verify that $\mathcal{L}_p\cap\mathcal{L}_{p'}=\emptyset$ for all $p,p'\in S'$ distinct. 
		
To this end, suppose that $p\in B_T^i(p')$ for $i\in\{1,\dots,\val(T,p')\}$ and that $p'\in B_T^j(p)$ for $j\in\{1,\dots,\val(T,p)\}$. By Lemma \ref{lem: brances subset of branch},
\begin{align}
\quad & B_T^m(p')\subset B_T^j(p) \quad \text{for all } m\neq i, \label{eq:p'subsetp} \\
\quad & B_T^n(p)\subset B_T^i(p') \quad \text{for all } n\neq j. \label{eq:psubsetp'}
\end{align}
Assume towards contradiction that there is a leaf $q\in \mathcal{L}_p\cap\mathcal{L}_{p'}$. There are three possible cases to consider, depending on the intersection of the arcs $[p,q]$ and $[p',q]$. Note that $\{q\}\subsetneq [p,q]\cap [p', q]$ by the unique arc property of $T$.
		
\emph{Case 1.} Suppose $[p,q]\cap [p', q]= [p,q]$. Since $[p,q]\subset [p', q]$, by \eqref{eq:p'subsetp}, \eqref{eq:psubsetp'} we have $q\in B_T^i(p')$. Due to $q\in \mathcal{L}_p$, there is $j_q\in \{4, \dots, \val(T,p)\}$ with $h_T(p,B_{j_q})=d(p,q)$, and there are at least three branches $B_T^{j_1}(p)$, $B_T^{j_2}(p)$, $B_T^{j_3}(p)$, $j_1, j_2, j_3 \in \{1, \dots \val(T,p)\}$ , with 
$$\min\{h_T(p, B_{j_1}), h_T(p, B_{j_2}),  h_T(p, B_{j_3})\}> h_T(p, B_{j_q}),$$
by definition of $\mathcal{L}_p$. Hence, there exists $j'\in \{1, \dots \val(T,p)\}$ with $h_T(p, B_{j'})> h_T(p, B_{j_q})=d(p,q)$ and $B_T^{j'}(p)\neq B_T^{j}(p)$. By \eqref{eq:psubsetp'} we have $B_T^{j'}(p)\subset B_T^i(p')$, which due to $q\in B_T^i(p')$ and $q\in \mathcal{L}_{p'}$ implies
$$h_T(p',B_i)=d(p', q)=d(p',p)+d(p,q)<d(p,p')+h_T(p,B_{j'})=d(p,p')+d(p,q_{j'}),$$ 
for some leaf $q_{j'}\in B_T^{j'}(p)\subset B_T^i(p')$. The right hand side in the above relation is at most  $d(p', q_{j'})\leq h_T(p',B_i)$, which leads to a contradiction.
		
\emph{Case 2.} Suppose $[p,q]\cap [p', q]= [p',q]$. The proof is similar to that of Case 1.
		
\emph{Case 3.} Suppose $[p,q]\cap [p', q]= [r,q]$, for some $r\in T\setminus\{p,p'\}$. This implies by \cite[Lemma 3.2(iii)]{BT_CSST} that $p', q, r$ lie on the same branch of $T$ at $p$. Hence, $q\in B_T^j(p)$, and by $q\in \mathcal{L}_p$ we have
\begin{equation}\label{eq: h_T(q) bigger than ell 1}
h_T(p,B_j)=d(p,q)\geq d(p,\ell),
\end{equation} 
for any leaf $\ell\in B_T^j(p)$. By $p'\in B_T^j(p)$ and \eqref{eq:p'subsetp},  for any $m\in \{1,\dots, \val(T,p')\}$ with $m\neq i$ we have that
\begin{equation}\label{eq: h_T(q) bigger than ell 2}
d(p,\ell)=d(p,r)+d(r,p')+d(p',\ell),
\end{equation}
for any leaf $\ell\in B_T^m(p')$. Fix some $B_T^m(p')$ for $m\neq i$, and a leaf $\ell \in B_T^m(p')\subset B_T^j(p)$. By \eqref{eq: h_T(q) bigger than ell 1}, \eqref{eq: h_T(q) bigger than ell 2}, and $d(p,r)+d(r,q)=d(p,q)$, we have
$$d(p,r)+d(r,q)\geq d(p,r)+d(r,p')+d(p',\ell).$$ 
Since $r\neq p'$, the above implies that $d(r,q)>d(p',\ell)$. But $[r,q]\subset [p',q]$, so 
$$d(p',\ell)< d(r,q)\leq d(p',q).$$ 
The branch $B_T^m(p')$ and the leaf $\ell\in B_T^m(p')$ are arbitrary, which means that the above inequality shows that $h_T(p',B)< d(p',q)$, for all branches $B\neq B_T^i(p')$ of $T$ at $p'$. This contradicts the fact that $q\in \mathcal{L}_{p'}$, since there is no need to glue a line segment at $q$ to increase a branch of $T$ at $p'$.
		
Recall that $\diam\ell_x>0$ for all $x\in\mathcal{L}^*$, since $x$ is a leaf in $\mathcal{L}_p$ for some $p\in S'$, and by definition of $S'$. We define the geodesic gluing $(T',d'):=(T,d,\bx_I)\bigvee_{i\in I}(\ell_{x_i},d_i,y_{x_i})$, where $d_i$ is the Euclidean metric on $\ell_{x_i}$. By  Lemma \ref{lem: comparable heights}, \cite[Lemma 3.9]{BT_CSST},  and $\diam\ell_{x_i}\leq h_T(p_i)$ for all $i\in I$, we have that $\diam\ell_{x_i}\to 0$ as $i\to\infty$. By Lemma \ref{lem:gluing is m.tree} this implies that $(T',d')$ is a geodesic metric tree. Note that the gluing is between $T$ at leaves and line segments at endpoints, which ensures that the branch points of $T$ and of $T'$ are in one-to-one correspondence. Therefore, $\val(T')=m$. 
		
Addressing uniform branch separation and uniform branch growth requires to show that the height of branch points in $T$ is the same as in $T'$. Fix a branch point $p$ of $T$ and let $B_{T}^k(p)$ be a branch of $T$ at $p$ for $k\in\{1,\dots,\val(T,p)\}$. The definition of $T'$, the convention $T\subset T'$, and the fact that $T'$ is a tree imply that there exists a unique $\tilde{n}_k\in\{1,\dots,\val(T',p)\}$ such that $B_{T}^k(p)\subset B_{T'}^{\tilde{n}_k}(p)$. Recall that $i_p\in\{3,\dots,\val(T,p)-1\}$, and $h_T(p,B_3)=h_T(p,B_{i_p})$. We  show that $h_T(p)=h_{T'}(p,B_j)$ for all $j\in \{3, \dots, \val(T',p)\}$, by considering three cases.
		
\emph{Case (i).} Suppose that $k\in\{i_p+1,\dots,\val(T,p)\}$. Let $q_k\in B_T^k(p)\cap\mathcal{L}_p$ such that $h_T(p,B_k)=d(p,q_k)$. Let $x\in B_{T'}^{\tilde{n}_k}(p)$. If $x\in T$, then 
$$d'(p,x)=d(p,x)\leq d(p,q_k)=h_T(p,B_k)< h_T(p).$$
If $x\in T'\setminus T$, then $x\in \ell_{r}$ for some leaf $r\in B_T^k(p)$ that lies in $\mathcal{L}^*$. If $r=q_k\in \mathcal{L}_p$, then 
$$d'(p,x)=d'(p,q_k)+d'(q_k,x)\leq {h_T(p,B_k)+\diam\ell_{q_k}=h_T(p)}.$$
If $r\neq q_k$, there is $p'\in T$ branch point with $h_T(p',B_j)=d(p',r)$, for some $j\in\{i_{p'}+1,\dots,\val(T,p')\}$. Note that $p'\in B_T^k(p)$. Suppose otherwise that $p'\in B_T^{k'}(p)$ for some $k'\neq k$, say $k'=1$. Then, by Lemma \ref{lem: brances subset of branch}, $B_T^2(p)\subset B_T^j(p')$. Furthermore, let $q_2\in B_T^2(p)$ be such that $d'(p,q_2)=h_T(p,B_2)$. By $k\neq 2$ and \cite[Lemma 3.2(iii)]{BT_CSST} $p\in[p'q_2]\cap[p',r]$. Hence, by geodesicity of $T$
\[d'(p',q_2)=d'(p',p)+d'(p,q_2)=d'(p',p)+h_T(p,B_2)>d'(p',p)+d'(p,r)=d'(p',r)\]
which is contradiction since $q_2\in B_T^k(p)$ and $r\in\mathcal{L}^*$. Repeating Lemma \ref{lem: brances subset of branch} we have $B_T^i(p')\subset B_T^k(p)$ for all $i\neq n$, where $B_T^n(p')$ is the branch of $T$ at $p'$ that contains $p$. By $k\neq 1$ and Lemma \ref{lem: brances subset of branch} we have $B_T^1(p)\subset B_T^n(p')$. This implies $n=1$ due to $k\geq i_p+1$, since otherwise $h(p, B_k)\geq h(p, B_1)$, leading to contradiction. Thus, $x,p$ lie in different components of $p'$, implying $p'\in[p,x]$ by \cite[Lemma 3.2(iii)]{BT_CSST}. Similarly, $p'\in[r',p]$, where $r'\in B_T^2(p')\subset B_T^k(p)$ such that $h_T(p',B_2)=d(p',r')$. It follows that
\begin{align*}
d'(p,x)&=d'(p,p')+d'(p',x) \\&=d'(p,p')+d'(p',r)+d'(r,x)\\ 
&\leq d'(p,p')+h_T(p',B_j)+\diam\ell_{r} \\
&\leq d'(p,p')+d'(p',r') \\
&=d'(p,r')\\
&\leq h_T(p,B_k)\\
&\leq h_T(p).
\end{align*}
Therefore, since $x\in B_{T'}^{\tilde{n}_k}(p)$ is arbitrary, we have shown that $h_{T'}(p,B_{\tilde{n}_k})\leq h_T(p)$. For the other side of the inequality note that if $y\in B_T^k(p)$ is an endpoint of $\ell_{q_k}$ distinct from $q_k$, then $$d'(p,y)=d'(p,q_k)+d'(q_k,y)=h_T(p,B_k)+\diam\ell_{q_k}=h_T(p).$$ 
		
\emph{Case (ii).} Suppose that $k\in\{3,\dots,i_p\}$ and that $h_T(p,B_2)> h_T(p,B_3)$. Let $q_k\in B_T^k(p)$ with $h_T(p)=d(p,q_k)$. We claim that no line segment is glued on the leaf $q_k\in T$. Then it follows with similar arguments as in Case (i) that $h_{T'}(p,B_{\tilde{n}_k})=h_T(p)$. Assume towards contradiction that $\ell_{q_k}$ is a line segment glued on the leaf $q_k$. Hence, there exists $n\in\{1,\dots,\val(T,p)\}$ and $p''\in B_T^n(p)\cap S'$ branch point such that $h_T(p'',B_m)=d(p'', q_k)$ for some $m\in\{i_{p''}+1,\dots,\val(T,p'')\}$. If $n=k$, then $p'',q_k\in B_T^k(p)$. By Lemma \ref{lem: brances subset of branch} and similar arguments as in Case (i), we have $p\in B_T^1(p'')$, and the rest of the branches of $T$ at $p''$ are subsets of $B_T^k(p)$. Let $r\in B_T^2(p'')$ such that $h_T(p'',B_2)=d(p'',r)$. Similarly to Case (i), it can be shown that $p''\in [p,r]\cap [p,q_k]$. Hence, $d'(p'',r)>d'(p'',q_k)$ due to $p''\in S'$. Adding $d'(p,p'')$ to each side results in $d'(p,r)>h_T(p,B_k)$, which is a contradiction by $r\in B_T^k(p)$. The case where $n\in\{2,3,\dots,i_p\}\setminus\{k\}$ follows by Lemma \ref{lem: brances subset of branch} and \eqref{eq:height of a branch2}. Lastly, if $n=1$, then $p\in [p'',q_k]\cap[p'',q_2]$, since they lie on different branches of $p$ in $T$. Therefore, $d'(p,q_2)>d'(p,q_k)$ due to $h_T(p,B_2)>h_T(p,B_3)$. This implies $d'(p'',q_2)>d'(p'',q_k)=h_T(p'',B_k)$ which is contradiction. 
		
\emph{Case (iii).} Suppose that $k\in\{3,\dots,i_p\}$, $q_k\in B_T^k(p)$ with $h_T(p)=d(p,q_k)$, and that $h_T(p,B_2)= h_T(p,B_3)$. This case is almost identical to Case (ii), with the exception of the case $n=1$, where the fact that $h_T(p,B_2)>h_T(p,B_3)$ was crucial. We may assume that $\ell_{q_k}$ is a line segment glued on the leaf $q_k$ (since otherwise we are in Case (ii)), and $p''\in B_T^1(p)\cap S'$ is a branch point with $h_T(p'',B_m)=d(p'', q_k)$ for some $m\in\{i_{p''}+1,\dots,\val(T,p'')\}$. We claim that $ \mathcal{L}_{p''}\cap\{q_p^2,\dots,q_p^{i_p}\}=\{q_k\}$, where $q_p^a\in B_T^a(p)$ with $h_T(p,B_a)=d(p,q_k^a)$ for all $a\in\{2,\dots,i_p\}$. Without loss of generality assume that $i_p=3$. Suppose there are $q_{s_1},q_{s_2}\in\mathcal{L}^*$ with $s_1,s_2\in S'$ and  $q_{s_1}=q_p^2,q_{s_2}=q_p^3$. Similarly to Case (ii), we may assume that $s_1,s_2\in B_T^1(p)$. By Lemma \ref{lem: brances subset of branch} we have $q_{s_1},q_{s_2}\in B_T^{i_{s_1}+1}(s_1)\cap B_T^{i_{s_2}+1}(s_2)$ where $i_{s_1},i_{s_2}\geq 3$. Without loss of generality we may assume that $i_{s_1}=i_{s_2}=3$. If $s_2\in[s_1,p]$,   by Lemma \ref{lem: brances subset of branch} the branches $B_T^2(s_2),B_T^4(s_2)$ are contained in $B_T^4(s_1)$. Let $z\in B_T^2(s_2)$ with $h_T(s_2,B_2)=d'(s_2,z)$. Then, $d'(s_2,z)>d'(s_2,q_{s_2})$, which by  $h_T(p,B_2)=h_T(p,B_3)$ and $d'(s_1,q_{s_1})=h_T(s_1,B_4)$ implies
$$d'(s_1,z)>d'(s_1,q_{s_2})=d'(s_1,p)+d'(p,q_{s_2})=d'(s_1,p)+d'(p,q_{s_1})=h_T(s_1,B_4),$$
which is a contradiction. The case $s_1\in[s_2,p]$ is similar. If $s_1\notin[s_2,p]$ and $s_2\notin[s_1,p]$, then $s_2\in B_T^4(s_1)$ and $s_1\in B_T^4(s_2)$ by \cite[Lemma 3.2(iii)]{BT_CSST}, since $p\in B_T^4(s_1)\cap B_T^4(s_2)$. Then $B_T^1(s_2)\subset B_T^4(s_1)$ and $B_T^1(s_1)\subset B_T^4(s_2)$ by Lemma \ref{lem: brances subset of branch}. Hence, we have
$$h_T(s_2,B_1)\leq h_T(s_1,B_4)<h_T(s_1,B_1)\leq h_T(s_2,B_4),$$ 
which is a contradiction by $s_1, s_2\in S'$. Hence, there is no line segment glued on $q_p^2\in B_T^2(p)$, for which it's true that $d(p,q_p^2)=h_T(p,B_2)=h_T(p,B_3)$. It follows with similar arguments as in Case (i) that $h_{T'}(p,B_{\tilde{n}_k})=h_T(p)$.
		
Therefore, for any branch point $p\in T'$, we have shown $h_{T'}(p,B_{T'}^j(p))=h_T(p)$, for any branch $B_{T'}^j(p)$ of $T'$ at $p$. The uniform branch growth follows trivially, and the uniform branch separation follows by the uniform separation of $T$ (Lemma \ref{lem: comparable heights}).  Lastly, by Lemma \ref{lem:doubling} we have that $T'$ is doubling, completing the proof.
\end{proof}

\subsection{Step 2: uniform valence}\label{sec:step2}
The second step in the proof of Proposition \ref{prop:qsembed} is the next proposition.
	
\begin{prop}\label{prop: step 2}
Let $(T,d)$ be a quasiconformal tree with $\val(T)=m$, uniformly separated branch points, and uniform branch growth. Then, $T$ quasisymmetrically embeds into $(T',d')$, where $T'$ is an $m$-valent geodesic doubling tree with uniform branch separation and uniform branch growth. 
\end{prop}
	
Note that given the pointed metric spaces $(Y_i,d_{Y_i},y_i)$, for $i\in I$, we can define the geodesic gluing  $\bigvee_{i\in I}(Y_i,d_{Y_i},y_i)$ by identifying $y_i$ and $y_j$ for all $i,j\in I$. More specifically, denote by $X_0=\{0\}$ the trivial singleton metric space, and define $\bigvee_{i\in I}(Y_i,d_{Y_i},y_i):= (X_0,d_0,\bx_I)\bigvee_{i\in I} (Y_i,d_{Y_i},y_i)$ as in Definition \ref{def:geodesic gluing}, where $x_i=0$ for all $i\in I$. We make use of this specific type of gluing in what follows.
	
\begin{proof}[{Proof of Proposition \ref{prop: step 2}}]
Similarly to Step 1, by applying \cite[Theorem 1.2]{BM20}, we may assume that the tree $(T,d)$ is a geodesic doubling metric tree with uniformly separated branch points, uniform branch growth and $\val(T)=m$ (since these properties are quasisymmetrically invariant).
		
Let $S=\{p\in T:\val(T,p)\in\{3,\dots,m-1\}\}$. Unlike Step 1, we glue Euclidean line segments on branch points of $T$, and, specifically, on every  $p\in S$.
		
Let $\{p_1,p_2,\dots\}$ be a fixed enumeration of $S$. Set $n_k=\val(T,p_k)$, and  let $I_{p_k}$ be an index set with $\card I_{p_k}=m-n_k$, for all $k$.  For $p_k\in S$, let $\ell_k^i$ be Euclidean line segments with $\diam\ell_k^i=h_T(p_k,B_{n_k})$, for all $i\in I_{p_k}$. Fix $y_k^i$ to be one of the endpoints of $\ell_k^i$. For each $k\in\N$, define $Y_k=\bigvee_{i\in I_{p_k}}(\ell_k^i,d_i,y_k^i)$, where $d_i$ is the Euclidean metric on $\ell_k^i$.
		
Set $(T',d'):=(T,d,\mathbf{p}_I)\bigvee_{k\in I}(Y_k,d_{Y_k}, z_k)$, where $I$ is an index set with $\card I=\card S$, $\mathbf{p}_I=(p_1, p_2, \dots)$, and  $z_k\sim y_k^i$ for all $i\in I_{p_k}$, i.e., for each $k\in I$ the point $z_k$ is the ``identification" of all points $y_k^i\in\ell_k^i$. Suppose $\card I=\infty$, and the proof is similar in the case $\card I<\infty$. By \cite[Lemma 3.9]{BT_CSST}, we have $\diam Y_k\to 0$ as $k\to\infty$. Hence, $(T',d')$ is a geodesic metric tree  by Lemma \ref{lem:gluing is m.tree}. 
		
It remains to show that $T'$ is $m$-valent, with uniform branch separation and uniform branch growth. Note that the gluing is between $T$ at branch points and line segments at endpoints, which implies that the branch points of $T$ and those of $T'$ are in one-to-one correspondence. Therefore, by the choice of $n_k$ and $Y_k$, we have  $\val(T',p)=m$ for all branch points of $T'$.  
		
Let $p$ be a branch point of $T'$. Then $p$ is a branch point of $T$, and for every $j\in \{1,\dots, \val(T,p)\}$ there is a unique $\tilde{n}_j$ such that $B_T^j(p)\subset B_{T'}^{\tilde{n}_j}(p)$. Let $x\in B_{T'}^{\tilde{n}_j}(p)$. If $x\in B_T^j(p)$, then
$$d'(p,x)=d(p,x)\leq h_T(p,B_j).$$ 
Suppose that $x \notin B_T^j(p)$, which implies that $x\in \ell_k^i$ for some $p_k\in S\cap B_T^j(p)$ and $i\in I_{p_k}$. Then, by $\diam \ell_k^i=h_T(p_k,B_{n_k})$ we have
\begin{align*}
d'(p,x)&=d(p,p_k)+d_i(p_k,x)\\ 
&\leq d(p,p_k)+\diam \ell_k^i\\ 
&= d(p,p_k)+h_T(p_k,B_{n_k})\\ 
&= d(p,p_k)+ d(p_k,q_k),
\end{align*}
for some leaf $q_k\in B_T^{n_k}(p_k)$. Note that $B_T^{n_k}(p_k)\subset B_T^j(p)$, by $p_k\in B_T^j(p)$ and Lemma \ref{lem: brances subset of branch}. Hence, the above implies $d'(p,x)\leq h_T(p,B_j)$. Therefore, since $x\in B_{T'}^{\tilde{n}_j}(p)$ was arbitrary, we showed that $h_{T'}(p, B_{\tilde{n}_j})\leq h_T(p,B_j)$, and the other inequality follows by  $B_T^j(p)\subset B_{T'}^{\tilde{n}_j}(p)$. As a result, the maximal distances of branches of $T$ are the same as those of branches of $T'$ at every branch point. The uniform growth and separation for $T'$ follow by those of $T$, and the fact that $T'$ is doubling follows by Lemma \ref{lem:doubling}.
\end{proof}

\subsection{Step 3: uniform density of branch points}\label{sec:step3}
The final step in the proof of Proposition \ref{prop:qsembed} is the following proposition.
	
\begin{prop}\label{prop: step 3}
Let $(T,d)$ be an $m$-valent quasiconformal tree with uniform branch separation and uniform branch growth with $m\geq 3$. Then, $(T,d)$ quasisymmetrically embeds into a geodesic uniformly $m$-branching quasiconformal tree $(T',d')$.
\end{prop}
	
For the rest of this section fix an integer $m\geq 3$, an alphabet $A=\{1,\dots,m\}$, and a weight $\ba$. For some fixed constant $c\in(0,1)$ denote by $c\T^{m,\ba}$ the metric space $(A^{\N}/\sim,c d_{A,\ba})$, i.e., a rescaled copy of the tree $\T^{m,\ba}$ of diameter $c$. 
	
As with the previous steps, we construct $T'$ ``on top" of $T$ with appropriate geodesic gluing. By rescaling the metric, we may assume that $\diam T=1$. 
	
For this step, we use the decomposition of the tree $T$ as in \cite[Section 5]{BM20}. Namely, there exist finite sets $\textbf{V}^n\subset T$, for all $n\in\N$, consisting of double points of $T$ with $\textbf{V}^1\subset\textbf{V}^2\subset\dots$. Each point $v\in\textbf{V}^n$ is called an \textit{$n$-vertex}. The closure of a component of $T\setminus\textbf{V}^n$ is called an \textit{$n$-tile}, and the set of all $n$-tiles is denoted by $\textbf{X}^n$ of \textit{level}  $n$. It is also convenient to denote $\textbf{V}^0=\emptyset$, $\textbf{X}^0=\{T\}$ with $T$ being the only $0$-tile. This specific collection $\textbf{V}^1, \textbf{V}^2, \dots$ satisfies certain properties for some fixed small enough constants $\beta\geq 1$, $\gamma>0$, and $\delta\in (0,1/3\beta)$.
	
\begin{lem}\label{lem:tile_properties}
\begin{enumerate}[(i)]
\item For each $v\in\textbf{V}^n$, $h_T(v,B_2)\geq c\beta\delta^n$, where $c$ is a constant that solely depends on the comparability constant from \eqref{eq:relation of two heights}.
\item For each $v\in \textbf{V}^n$ and branch point $p\in T$, $d(v,p)\geq c\gamma\min\{h_T(p),c \delta^n\}$.
\item For $u,v\in\textbf{V}^n$ distinct, $d(u,v)\geq\delta^n$.
\item For each $n$-tile $X$, $\diam(X)\simeq\delta^n$. In particular, $\delta^n\leq\diam(X)\leq 3\beta\delta^n$.
\item Each $n$-tile $X$ is equal to the union of all $(n+ 1)$-tiles $X_0$ with $X_0\subset X$.
\item Each $n$-tile $X$ contains at least three $(n + 1)$-tiles.
\item Each $n$-tile $X$ is a subtree of $T$ with $\partial X\subset \mathbf{V}_n$.
\item If $v$ is an $n$-vertex and $X$ an $n$-tile with $v\in X$, then $v\in \partial X$. Moreover, there exists precisely one $(n + 1)$-tile $X_0\subset X$ with $v\in X_0$.
\item There is a constant $K\in \N$ such that for each $n\in \N_0$ and each $n$-tile $X$, the number of $(n + 1)$-tiles contained in $X$ are at most $K$.
\end{enumerate}
\end{lem}
	
\begin{proof}
Claims (i), (ii) are immediate by \cite[Equations (5.5), (5.6)]{BM20}, Lemma \ref{lem: comparable heights} and the fact that $0<c\leq 1$. 
		
Claims (iii), (iv) are proved in  \cite[Equations (5.2), (5.3)]{BM20}. Claim (v) is proved in \cite[Lemma 5.1 (viii)]{BM20}. Claim (vi) is proved in \cite[Lemma 5.4 (i)]{BM20}. Claim (vii) is proved in \cite[Lemma 5.1 (i)]{BM20}. Claim (viii) is proved in \cite[Lemma 5.1 (ix)]{BM20}. Claim (ix) is proved in \cite[Lemma 5.7 (ii)]{BM20}.
\end{proof}
	
We also need to determine the new height of branch points of $\T^{m,\ba}$.
	
\begin{lem}\label{lem: new height of T^m}
Let $w\in A^*$ and $p=[w12^{(\infty)}]$ be a branch point of $\T^{m,\ba}$. Then $h_{\T^{m,\ba}}(p)=\ba(3)\Delta(w)$, and $[1^{(\infty)}]$ lies either on $B_{\T^{m,\ba}}^1(p)$, or on $B_{\T^{m,\ba}}^2(p)$. 
\end{lem}
	
\begin{proof}
Let $w \in A^*$, then $p = [w12^{(\infty)}]$ is a branch point of $\T_w^{m,\ba}$ with $m$ distinct branches $\T_{w1}^{m,\ba}, \dots, \T_{wm}^{m,\ba}$ in $\T_w^{m,\ba}$.
		
Observe that if $i\neq 1$, then by Remark \ref{rem:equivalent classes of 12^oo}
$$h_{\T^{m,\ba}}(p,\T_{wi}^{m,\ba})\geq d_{A,\ba}( [wi1^{(\infty)}],[wi2^{(\infty)}])=\ba(i)\Delta(w),$$ 
and similarly $h_{\T^{m,\ba}}(p,\T_{w1}^{m,\ba})\geq\ba(1)\Delta(w)$. On the other hand, for each $i\in A$
$$h_{\T^{m,\ba}}(p,\T_{wi}^{m,\ba})\leq\diam\T_{wi}^{m,\ba}=\ba(i)\Delta(w).$$ 
Hence, $h_{\T^{m,\ba}}(p,\T_{wi}^{m,\ba})=\ba(i)\Delta(w)$ for all $i\in\{1,\dots,m\}$. By Lemma \ref{lem:height formula} it follows that $h_{\T^{m,\ba}}(p)=\ba(3)\Delta(w)$ (note that Lemma \ref{lem:height formula} is stated for an arbitrary order of the branches at $p$). For the second part of the statement, suppose that $[1^{(\infty)}]\in B_{\T^{m,\ba}}^i$. Let $n\geq 0$ be a maximal integer such that $w=1^{(n)}v$, $v\in A^{|w|-n}$, $v(1)\neq 1$. We have shown that 
$$h_{\T^{m,\ba}}(p,B_j)=\diam \T_{wj}^{m,\ba}=\Delta(wj)=\ba(j)\Delta(w)\leq 2^{-|w|-1},$$ for all $j\in\{3,\dots,m\}$. On the other hand,
\begin{align*}
h_{\T^{m,\ba}}(p,B_i) &\geq \rho([1^{(\infty)}],[1^{(n)}v12^{(\infty)}])\\ 
&= 2^{-n} \rho([1^{(\infty)}],[v12^{(\infty)}]) \notag \\
&= 2^{-n} \left(\rho([1^{(\infty)}],[12^{(\infty)}]) + \rho([12^{(\infty)}],[v12^{(\infty)}])\right) \notag \\
&\geq 2^{-n-1},
\end{align*}
where we used the geodesicity of $\T^{m,\ba}$ and the fact that $[12^{(\infty)}]$ is on the arc that connects $[1^{(\infty)}],[v12^{(\infty)}]$. Due to $n\leq|w|$, it follows that $i\in\{1,2\}$. 
\end{proof}
	
\begin{proof}[{Proof of Proposition \ref{prop: step 3}}]
Similarly to the previous steps, applying \cite[Theorem 1.2]{BM20}, we may assume that the tree $(T,d)$ is a geodesic doubling $m$-valent metric tree with uniformly separated branch points and uniform branch growth (since these properties are quasisymmetrically invariant).
		
If $T$ has uniformly dense branch points, the statement is trivially true. Suppose $T$ does not have uniformly dense branch points. We use the decomposition $\textbf{V}^1\subset \textbf{V}^2\subset\dots$ of $T$ as defined in \cite[Section 5]{BM20}, which provides a  countable set of double points dense in $T$, where we plan on gluing branches of sufficient (new) height, so that the resulting space is a geodesic, doubling, uniformly $m$-branching metric tree.
		
Let $\{v_1,v_2,\dots\}$ be a fixed enumeration of $\textbf{V}^*:=\bigcup_{n=1}^{\infty}\textbf{V}^n$. By Lemma \ref{lem:tile_properties} (vi) and (vii), each tile $X\in\textbf{X}^n$ contains at least 3 tiles of the next level $\textbf{X}^{n+1}$ and $\partial X\subset \textbf{V}^n$. Hence, the sequence $\textbf{V}^1\subset \textbf{V}^2\subset\dots$ is a strictly increasing sequence. Therefore, we can write $\textbf{V}^*=\bigcup_{n=1}^{\infty} W_n$, where the sets $W_n:=\mathbf{V}^n\setminus\mathbf{V}^{n-1}$ are disjoint, have cardinality $M_n:=\card W_n\in \N$, and assume that the aforementioned enumeration of $\mathbf{V}^*$ is the union of consecutive enumerations of each $W_n$, i.e., $\mathbf{V}^*=\{v_1^1,\dots, v^1_{M_1}, v_1^2, \dots, v_{M_2}^2, \dots\}$. Thus, if $x\in \textbf{V}^*$, then there exist unique $n\in\N$ and $i\in \{1,\dots, M_n\}$ for which $x=v_i^n\in W_n$. 
		
We plan on gluing $m-2$ isometric copies of $c\delta^n\T^{m,\ba}$ at $[1^{(\infty)}]$ and each double point $x\in W_n$. Here, $c$ is a constant that solely depends on the comparability constant from \eqref{eq:relation of two heights}. Thus, for every $n\in \N$, and every $i\in \{1,\dots,M_n\}$,  set 
$$S_0=0,\ S_n:=M_1+\dots +M_n,$$ 
$$\Omega_n^i=\{(S_{n-1}+i-1)(m-2)+1,\dots, (S_{n-1}+i)(m-2)\},$$ 
and $\Omega_n:= \bigcup_{i}\Omega_n^i= \{S_{n-1}(m-2)+1, \dots, S_n(m-2)\}$. We define
$$(Y_k^n,d_{k,n}):=c\delta^n\T^{m,\ba}\times \{k\},$$ 
for all $k\in \Omega_n$. Set $y_k^n:=([1^{(\infty)}],k)\in Y_k^n$, which we also denote by $[1^{(\infty)}]_k^n$, and write $[2^{(\infty)}]_k^n:=([2^{(\infty)}],k)\in Y_k^n$. Note that to define $Y_k^n$, we briefly used the more detailed notation from Definition \ref{def:geodesic gluing}, in order to point out how at every double point $x\in W_n$ we need isometric, but distinct copies of the scaled metric tree $c\delta^n\T^{m,\ba}$. Henceforth, we return to the convention where we identify points in $Y_k^n$ with points in $c\delta^n\T^{m,\ba}$. In addition, it is understood that if $n$ is fixed, then $k$ is used to denote integers in $\Omega_n$, and we write $Y_k=Y_k^n$, $d_k=d_{k,n}$,  $y_k=[1^{(\infty)}]_k=y_k^n$, and $[2^{(\infty)}]_k=[2^{(\infty)}]_k^n$. On the other hand, if $k$ is an arbitrary positive integer, then the corresponding $n=n_k$ is the unique integer such that $k\in \Omega_n$. For every $k\in \N$, set $u_k=v_i^n\in \mathbf{V}^*$, for all $i=i_k$ and the unique $n=n_k$ with $k\in \Omega_n^i$. Set $\mathbf{u}_\N:=(u_1, u_2,\dots)$, where the terms are not necessarily distinct, due to the choice of $u_k$. We claim that 
$$(T',d'):= (T,d,\mathbf{u}_\N)\bigvee_{k\in \N}(Y_k^n,d_{k,n},y_k^n)$$ 
is the space with the desired properties, into which $T$ isometrically embeds.
		
By Lemma \ref{lem: comparable heights}, \cite[Lemma 3.9]{BT_CSST}, and Lemma \ref{lem:gluing is m.tree}, it can be shown that $(T',d')$ is a geodesic metric tree (similarly to the proofs of Propositions \ref{prop:step 1}, \ref{prop: step 2}). By Lemma \ref{lem: comparable heights}, we can assume that $T$ has uniform branch separation and uniform branch growth with respect to $h_T$. In order to address uniform branch growth and uniform branch separation for $T'$, we need to determine the height $h_{T'}$ of branch points of $T'$. To this end, let $p\in T'$ branch point. If $p=v_i^n\in W_n$ for some $n\in\N$, $i\in \{1,\dots, M_n\}$, then, by Lemma \ref{lem:tile_properties} (i) and $\beta \geq 1$, the height of $p$ is determined by the distances within the corresponding $Y_k=Y_k^n$, $k\in \Omega_n^i$, which are glued at $p$. In particular,  since all the glued trees $Y_k$ are isometric, we have for any $k\in \Omega_n^i$ and $j=j_k=k-(S_{n-1}+i-1)(m-2)+2\in\{3,\dots,m\}$ that
$$ h_{T'}(p,B_j)=h_{T'}([1^{(\infty)}]_k,Y_k)\geq d_k([1^{(\infty)}]_k,[2^{(\infty)}]_k)=c\delta^n.$$
On the other hand, $h_{T'}([1^{(\infty)}]_k,Y_k)\leq c\delta^n \diam \T^{m,\ba}$, which with the above implies $h_{T'}(p,B_j)=c\delta^n$. If $p$ is a branch point of $T$, then again by Lemma \ref{lem:tile_properties} (i), and similar arguments as in the proofs of Propositions \ref{prop:step 1}, \ref{prop: step 2}, we have that $h_{T'}(p)=h_T(p)$. Lastly, if $p$ is a branch point of $Y_k^n$ for some $n\in\N$ and $k\in \Omega_n$, i.e., $p=[w12^{(\infty)}]_k$ for some $w\in A^*$, then, by Lemma \ref{lem: new height of T^m}, 
$$h_{T'}(p)=c\delta^n \ba(3)\Delta(w)\leq c\delta^n.$$
		
Note that the only branch points of $T'$ that are not branch points on $T$ or on some $Y_k^n$ are the points $v_i^n\in \mathbf{V}^*$. Since $T$ and $Y_k^n$ have uniform branch growth, and $h_{T'}(v_i^n,B_j)=c\delta^n$ for all $j\in\{3,\dots,m\}$, it follows that $T'$ also has  uniform branch growth (note that by Remark \ref{rem:constants of T}, the uniform branch growth constant is independent of $\diam \T^{m,\ba}$ and, hence, independent of $\diam Y_k^n$).
		
Let $p_1,p_2$ be two distinct branch points of $T'$, and let $c_1,c_2$ be the uniform branch separation constants of $T$ and $Y_k^n$, respectively (note that by  Remark \ref{rem:constants of T}, $c_2$ is independent of $\diam Y_k^n$ and, thus, independent of $k,n$). Suppose that $p_1\in T$. If $p_2\in T\setminus \mathbf{V}^*$, then
$$d'(p_1,p_2)=d(p_1,p_2)\geq c_1 \min\{h_T(p_1),h_T(p_2)\}= c_1 \min\{h_{T'}(p_1),h_{T'}(p_2)\}.$$
If $p_2\in W_n$ for some $n\in\N$, then, by Lemma \ref{lem:tile_properties}(ii),
$$d'(p_1,p_2)\geq c\gamma\min\{h_{T'}(p_1),c\delta^n\}= c\gamma\min\{h_{T'}(p_1),h_{T'}(p_2)\}.$$ 
Similarly, if $p_2\in Y_k^n$ for some $n, k\in\N$, and $v=u_k\in W_n$ is the $n$-vertex at which we glue $Y_k^n$, it follows that
$$d'(p_1,p_2)\geq d'(p_1,v)\geq c\gamma\min\{h_{T'}(p_1),c\delta^n\}\geq c\gamma\min\{h_{T'}(p_1),h_{T'}(p_2)\}.$$ 
Suppose that $p_1\in W_n$ for some $n\in\N$. If $p_2\in W_{n'}$ for some $n'\in\N$, then, without loss of generality, assume that $n\geq n'$. By Lemma \ref{lem:tile_properties} (v), $p_2$ lies in some $n$-tile $X$. Hence, by Lemma \ref{lem:tile_properties}(iii), 
$$d'(p_1,p_2)\geq \delta^n\geq h_{T'}(p_1)=\min\{h_{T'}(p_1),h_{T'}(p_2)\}.$$ 
If $p_2\in Y_{k'}^{n'}$ for some $n', k'\in\N$, then let $v'=u_{k'}\in \mathbf{V}^*$ be the unique $n'$-vertex to which we glue $Y_{k'}^{n'}$. If $v'\neq p_1$, then, by geodesicity and similar arguments to the previous case, it follows that 
$$d'(p_1,p_2)\geq d'(p_1,v') \geq \delta^n\geq \min\{h_{T'}(p_1),h_{T'}(v')\}\geq \min\{h_{T'}(p_1),h_{T'}(p_2)\}.$$ 
If $v'=p_1$, by similar arguments to those in the proof of Lemma \ref{lem: new height of T^m}, 
$$d'(p_1,p_2)=d_{k',n'}([1^{(\infty)}]_{k'},p_2)\geq h_{Y_{k'}^{n'}}(p_2)=h_{T'}(p_2)=\min\{h_{T'}(p_1),h_{T'}(p_2)\}.$$ 
If $p_1\in Y_k^n$ and $p_2\in Y_{k'}^{n'}$, for $n,n',k,k'\in\N$, then 
$$d'(p_1,p_2)\geq c_2\min\{h_{T'}(p_1),h_{T'}(p_2)\},$$
which follows similarly to previous cases, so we omit the details. Therefore, $T'$ has uniformly separated branch points. By Lemma \ref{lem:doubling}, $T'$ is doubling.

It remains to show that $T'$ has uniformly dense branch points. Let $x,y\in T'$ with $x\neq y$, and let $c_3$ be the uniform density constant of $Y_k^n$ for all $n,k\in \N$ (note that by Remark \ref{rem:constants of T}, $c_3$ is independent of $k,n$). If $x,y\in T$, then let $n\geq 0$ be the maximal integer such that $x,y\in X$, where $X$ is an $n$-tile. Hence, by Lemma \ref{lem:tile_properties} (v), we have $x\in X_1$ and $y\in X_2$, where $X_1,X_2\subset X$ are $(n+1)$-tiles with $X_1\neq X_2$. Since $x,y$ cannot both lie on the boundary of the same $(n+1)$-tile, we may assume that $x\notin\partial X_1$. Note that by Lemma \ref{lem:tile_properties} (vii), $X$ is a subtree of $T'$, which implies $[x,y]\subset X$. We claim that the appropriate branch point of $T'$ in $[x,y]$ for the uniform density condition is the unique point $p$ in $\partial X_1\cap [x,y]$. Since $X_1$ is an $(n+1)$-tile, we have that $p\in \textbf{V}^{n+1}$, by Lemma \ref{lem:tile_properties} (vii). If $p\in \textbf{V}^n$, then $p\in\partial X$, by Lemma \ref{lem:tile_properties} (viii). Therefore, $[p,y]\cap (T'\setminus X)\neq\emptyset$, which is a contradiction, since $[p,y]\subset[x,y]\subset X$. Hence, $p\in\textbf{V}^{n+1}\setminus\textbf{V}^n=W_{n+1}$, which implies that  $h_{T'}(p)=c\delta^{n+1}$. Also, by geodesicity, the diameter of the tile $X$ is equal to the sum of diameters of all $(n+1)$-tiles contained in $X$. By Lemma \ref{lem:tile_properties} (iv) and (ix), this implies that 
$$d'(x,y)\leq K 3\beta\delta^{n+1}= 3K c^{-1}\beta h_{T'}(p),$$
for some constant $K$ depending only on the doubling constant of the tree $T$.
		
Suppose that $x\in T$ and $y\in Y_k^n$ for some $n,k\in\N$. Let $v=u_k\in W_n$ be the vertex to which we glue $Y_k^n$. Suppose first that $x\neq v$. Let $n'\geq 0$ be the maximal integer such that $x,v\in Z$, where $Z$ is a $n'$-tile.  If $n'< n$, by geodesicity, and by applying the previous case to $x,v$, we have
$$d'(x,y)=d'(x,v)+d'(v,y)\leq K3\beta\delta^{n'+1}+\diam Y_k^n\leq(3Kc^{-1}\beta+1)h_{T'}(p),$$ 
where $p$ is the branch point lying on the intersection of $[x,v]$ and the boundary of the $(n'+1)$-tile that contains $x$. On the other hand, if $n'\geq n$, we claim that $v$ is the desired branch point of $T'$. Indeed, by $h_{T'}(v)=c \delta^n$ and Lemma \ref{lem:tile_properties} (iv),
$$d'(x,y)=d'(x,v)+d'(v,y)\leq 3\beta\delta^{n'}+\diam Y_k^n\leq(3c^{-1}\beta+1)h_{T'}(v).$$
Suppose that $x=v$. Then $d'(x,y)=d'(v,y)\leq \diam Y_k^n=c\delta^n=h_{T'}(v)$.
		
Lastly, if $x\in Y_k^n$ and $y\in Y_{k'}^{n'}$, for $n,k,n',k'\in\N$, it can be shown that
$$d'(x,y)=d'(x,v)+d'(v,y)\leq\max\{c_3, (6Kc^{-1}\beta+1)\}h_{T'}(p),$$ 
for some branch point $p\in T'$, with similar case studies as in the previous cases.
		
It follows by all the above cases that $T'$ has uniformly dense branch points. Therefore, $T'$ is uniformly $m$-branching.
\end{proof}

\section{Proofs of Theorem \ref{thm:main} and Theorem \ref{thm:uniformization}}\label{sec:Finale}
We are finally able to give the proof of Theorem \ref{thm:main} and Theorem \ref{thm:uniformization}. Fix for the rest of this section a number $\e>0$. Fix also a weight $\ba$ such that $\sum_{n=1}^{\infty}\ba(n)^{1+\e} < 1$. 
	
Recall from \textsection\ref{sec:isomembed} our convention that $\T^{m,\ba} \subset \T^{n,\ba} \subset \T^{\infty,\ba}$ if $2 \leq n \leq m$ are integers. Set $\T^{n} = \T^{n,\ba}$ for $n\in\{2,3,\dots\}$ and $\T^{\infty}=\T^{\infty,\ba}$. 
	
By Lemma \ref{lem:ss}, and the argument following the lemma, each $\T^n$ is self-similar in the sense of \textsection\ref{sec:ss}. By Lemma \ref{lem:metrictree}, each $\T^n$ and $\T^{\infty}$ are metric trees, and by Proposition \ref{prop:geodesic} each $\T^n$ and $\T^{\infty}$ are geodesic trees. By Proposition \ref{prop:doubling}, each $\T^n$ is Ahlfors regular and the Hausdorff dimension of $\T^{\infty}$ is at most $1+\e$.
	
By Theorem \ref{thm:uniformization2}, a metric space $T$ is an uniformly $n$-branching quasiconformal tree if and only if $T$ is quasisymmetric to $\T^{n}$. This proves Theorem \ref{thm:uniformization}. 
	
By Proposition \ref{prop: bi-Lip embeding} each $\T^n$ bi-Lipschitz embeds into $\R^2$ with the embedded image being quasiconvex. By Proposition \ref{prop:geodesic} and Proposition \ref{prop:valence}, $\T^2$ is isometric to the Euclidean unit interval $[0,1]$, while the CSST is quasisymmetric to $\T^3$ by Theorem \ref{thm:uniformization} and the fact that the CSST is uniformly 3-branching \cite{BM22}. Lastly, by Proposition \ref{prop: Vicsek universal}, the Vicsek fractal is uniformly 4-branching, so it is quasisymmetric equivalent to $\T^4$. This proves Theorem \ref{thm:main}(i).
	
By Lemma \ref{lem:isom-embed}, spaces $\T^n$ converge to $\T^{\infty}$ as $n\to \infty$ in the Gromov-Hausdorff sense.  This proves Theorem \ref{thm:main}(ii).
	
Lastly, let $T$ be a quasiconformal metric tree in $\mathscr{QCT}^*(n)$. By Proposition \ref{prop:qsembed}, $T$ quasisymmetrically embeds into a uniformly $n$-branching quasiconformal tree $T'$. By Theorem \ref{thm:uniformization} $T'$ is quasisymmetrically equivalent to $\T^{n}$. Hence, $T$ quasisymmetrically embeds into $\T^{n}$. This proves Theorem \ref{thm:main}(iii).


\begin{thebibliography}{DEBV23}

\bibitem[Ald91a]{Aldous1}
David Aldous.
\newblock The continuum random tree. {I}.
\newblock {\em Ann. Probab.}, 19(1):1--28, 1991.

\bibitem[Ald91b]{Aldous2}
David Aldous.
\newblock The continuum random tree. {II}. {A}n overview.
\newblock In {\em Stochastic analysis ({D}urham, 1990)}, volume 167 of {\em London Math. Soc. Lecture Note Ser.}, pages 23--70. Cambridge Univ. Press, Cambridge, 1991.

\bibitem[Azz15]{Azzam}
Jonas Azzam.
\newblock {Hausdorff dimension of wiggly metric spaces}.
\newblock {\em Arkiv för Matematik}, 53(1):1 -- 36, 2015.

\bibitem[BBI01]{BBI}
Dmitri Burago, Yuri Burago, and Sergei Ivanov.
\newblock {\em A course in metric geometry}, volume~33 of {\em Graduate Studies in Mathematics}.
\newblock American Mathematical Society, Providence, RI, 2001.

\bibitem[BC23a]{BaudoinChen2}
Fabrice Baudoin and Li~Chen.
\newblock Heat kernel gradient estimates for the {V}icsek set.
\newblock {\em ar{X}iv preprint}, 2023.

\bibitem[BC23b]{BaudoinChen1}
Fabrice Baudoin and Li~Chen.
\newblock Sobolev spaces and {P}oincar\'{e} inequalities on the {V}icsek fractal.
\newblock {\em Ann. Fenn. Math.}, 48(1):3--26, 2023.

\bibitem[Bes84]{Bestvina84}
Mladen Bestvina.
\newblock Characterizing {$k$}-dimensional universal {M}enger compacta.
\newblock {\em Bull. Amer. Math. Soc. (N.S.)}, 11(2):369--370, 1984.

\bibitem[Bes02]{Bestvinabook}
Mladen Bestvina.
\newblock {$\Bbb R$}-trees in topology, geometry, and group theory.
\newblock In {\em Handbook of geometric topology}, pages 55--91. North-Holland, Amsterdam, 2002.

\bibitem[Bis14]{Bishop}
Christopher~J. Bishop.
\newblock True trees are dense.
\newblock {\em Inventiones mathematicae}, 197(2):433--452, 2014.

\bibitem[BM20]{BM20}
Mario Bonk and Daniel Meyer.
\newblock Quasiconformal and geodesic trees.
\newblock {\em Fund. Math.}, 250(3):253--299, 2020.

\bibitem[BM22]{BM22}
Mario Bonk and Daniel Meyer.
\newblock Uniformly branching trees.
\newblock {\em Trans. Amer. Math. Soc.}, 375(6):3841--3897, 2022.

\bibitem[BT01]{BishopTyson}
Christopher~J. Bishop and Jeremy~T. Tyson.
\newblock Conformal dimension of the antenna set.
\newblock {\em Proc. Amer. Math. Soc.}, 129(12):3631--3636, 2001.

\bibitem[BT21]{BT_CSST}
Mario Bonk and Huy Tran.
\newblock The continuum self-similar tree.
\newblock In {\em Fractal geometry and stochastics {VI}}, volume~76 of {\em Progr. Probab.}, pages 143--189. Birkh\"{a}user/Springer, Cham, 2021.

\bibitem[CG93]{CG-dynamics}
Lennart Carleson and Theodore~W. Gamelin.
\newblock {\em Complex dynamics}.
\newblock Universitext: Tracts in Mathematics. Springer-Verlag, New York, 1993.

\bibitem[Cha80]{Charatonik}
Janusz~J. Charatonik.
\newblock Open mappings of universal dendrites.
\newblock {\em Bull. Acad. Polon. Sci. S\'er. Sci. Math.}, 28(9-10):489--494, 1980.

\bibitem[CJY94]{CJY}
Lennart Carleson, Peter~W. Jones, and Jean-Christophe Yoccoz.
\newblock Julia and {J}ohn.
\newblock {\em Bol. Soc. Brasil. Mat. (N.S.)}, 25(1):1--30, 1994.

\bibitem[CSW11]{CSW}
Sarah Constantin, Robert~S. Strichartz, and Miles Wheeler.
\newblock Analysis of the {L}aplacian and spectral operators on the {V}icsek set.
\newblock {\em Commun. Pure Appl. Anal.}, 10(1):1--44, 2011.

\bibitem[DEBV23]{DEBV}
Guy~C. David, Sylvester Eriksson-Bique, and Vyron Vellis.
\newblock Bi-{L}ipschitz embeddings of quasiconformal trees.
\newblock {\em Proc. Amer. Math. Soc.}, 151(5):2031--2044, 2023.

\bibitem[DV22]{DV}
Guy~C. David and Vyron Vellis.
\newblock Bi-{L}ipschitz geometry of quasiconformal trees.
\newblock {\em Illinois J. Math.}, 66(2):189--244, 2022.

\bibitem[FG23a]{FreeGart1}
David Freeman and Chris Gartland.
\newblock Lipschitz functions on quasiconformal trees.
\newblock {\em Fund. Math.}, 262(2):153--203, 2023.

\bibitem[FG23b]{FreeGart2}
David Freeman and Chris Gartland.
\newblock Lipschitz functions on unions and quotients of metric spaces.
\newblock {\em Studia Math.}, 273(1):29--61, 2023.

\bibitem[GH12]{GH}
Frederick~W. Gehring and Kari Hag.
\newblock {\em The ubiquitous quasidisk}, volume 184 of {\em Mathematical Surveys and Monographs}.
\newblock American Mathematical Society, Providence, RI, 2012.
\newblock With contributions by Ole Jacob Broch.

\bibitem[Gro81]{Gromov1}
Mikhael Gromov.
\newblock Groups of polynomial growth and expanding maps.
\newblock {\em Inst. Hautes \'{E}tudes Sci. Publ. Math.}, (53):53--73, 1981.

\bibitem[Gro99]{Gromov2}
Misha Gromov.
\newblock {\em Metric structures for {R}iemannian and non-{R}iemannian spaces}, volume 152 of {\em Progress in Mathematics}.
\newblock Birkh\"{a}user Boston, Inc., Boston, MA, 1999.
\newblock Based on the 1981 French original [MR0682063 (85e:53051)], With appendices by M. Katz, P. Pansu and S. Semmes, Translated from the French by Sean Michael Bates.

\bibitem[GT11]{BLtrees_Gupta}
Anupam Gupta and Kunal Talwar.
\newblock Making doubling metrics geodesic.
\newblock {\em Algorithmica}, 59(1):66--80, 2011.

\bibitem[Hei01]{Heinonen}
Juha Heinonen.
\newblock {\em Lectures on analysis on metric spaces}.
\newblock Universitext. Springer-Verlag, New York, 2001.

\bibitem[HM98]{HamblyMetz}
B.~M. Hambly and V.~Metz.
\newblock The homogenization problem for the {V}icsek set.
\newblock {\em Stochastic Process. Appl.}, 76(2):167--190, 1998.

\bibitem[Hut81]{Hu81}
John~E. Hutchinson.
\newblock Fractals and self-similarity.
\newblock {\em Indiana Univ. Math. J.}, 30(5):713--747, 1981.

\bibitem[Kin17]{Kinneberg}
Kyle Kinneberg.
\newblock Conformal dimension and boundaries of planar domains.
\newblock {\em Trans. Amer. Math. Soc.}, 369(9):6511--6536, 2017.

\bibitem[LG06]{LeGall}
Jean-Fran\c~cois Le~Gall.
\newblock Random real trees.
\newblock {\em Ann. Fac. Sci. Toulouse Math. (6)}, 15(1):35--62, 2006.

\bibitem[Li21]{Li2021}
Wenbo Li.
\newblock Quasisymmetric embeddability of weak tangents.
\newblock {\em Ann. Fenn. Math.}, 46(2):909--944, 2021.

\bibitem[LNP09]{BLtrees_Lee_Naor}
James~R. Lee, Assaf Naor, and Yuval Peres.
\newblock Trees and {M}arkov convexity.
\newblock {\em Geom. Funct. Anal.}, 18(5):1609--1659, 2009.

\bibitem[LR18]{RohdeLin}
Peter Lin and Steffen Rohde.
\newblock Conformal welding of dendrites.
\newblock {\em preprint}, 2018.

\bibitem[McM98]{McMullen}
Curtis~T. McMullen.
\newblock Kleinian groups and {J}ohn domains.
\newblock {\em Topology}, 37(3):485--496, 1998.

\bibitem[Met93]{Metz}
Volker Metz.
\newblock How many diffusions exist on the {V}icsek snowflake?
\newblock {\em Acta Appl. Math.}, 32(3):227--241, 1993.

\bibitem[Nad92]{Nadler}
Sam~B. Nadler, Jr.
\newblock {\em Continuum theory}, volume 158 of {\em Monographs and Textbooks in Pure and Applied Mathematics}.
\newblock Marcel Dekker, Inc., New York, 1992.
\newblock An introduction.

\bibitem[NV91]{NaVa}
Raimo N\"akki and Jussi V\"ais\"al\"a.
\newblock John disks.
\newblock {\em Exposition. Math.}, 9(1):3--43, 1991.

\bibitem[Sea07]{Sear07}
M\'{\i}che\'{a}l~\'{O}. Searc\'{o}id.
\newblock {\em Metric spaces}.
\newblock Springer Undergraduate Mathematics Series. Springer-Verlag London, Ltd., London, 2007.

\bibitem[Sta71]{Stanko71}
M.~A. Stanko.
\newblock Solution of {M}enger's problem in the class of compacta.
\newblock {\em Dokl. Akad. Nauk SSSR}, 201:1299--1302, 1971.

\bibitem[Str93]{Str93}
Daniel~W. Stroock.
\newblock {\em Probability theory, an analytic view}.
\newblock Cambridge University Press, Cambridge, 1993.

\bibitem[TV80]{TV80}
P.~Tukia and J.~V\"{a}is\"{a}l\"{a}.
\newblock Quasisymmetric embeddings of metric spaces.
\newblock {\em Ann. Acad. Sci. Fenn. Ser. A I Math.}, 5(1):97--114, 1980.

\bibitem[Vic83]{Vicsek}
T~Vicsek.
\newblock Fractal models for diffusion controlled aggregation.
\newblock {\em Journal of Physics A: Mathematical and General}, 16(17):L647, dec 1983.

\bibitem[Why63]{Whyburn_book}
Gordon~Thomas Whyburn.
\newblock {\em Analytic topology}, volume Vol. XXVIII of {\em American Mathematical Society Colloquium Publications}.
\newblock American Mathematical Society, Providence, RI, 1963.

\bibitem[Zho09]{Zhou}
Denglin Zhou.
\newblock Spectral analysis of {L}aplacians on the {V}icsek set.
\newblock {\em Pacific J. Math.}, 241(2):369--398, 2009.

\end{thebibliography}
\end{document}